\documentclass[12pt]{elsarticle}
\usepackage{amsmath}
\usepackage{amssymb}
\usepackage{amsthm}
\usepackage{verbatim}
\usepackage{subfigure}
\usepackage{latexsym}
\usepackage{lmodern}
\usepackage{array}
\usepackage{graphicx}
\usepackage{float}
\usepackage{multirow}
\usepackage{listings}
\usepackage{color}
\usepackage{bm}
\usepackage{url}
\usepackage[english]{babel}
\usepackage{hyperref}
\usepackage{algorithm}
\usepackage{algorithmic}
\usepackage{multicol,balance}
\usepackage{times}
\usepackage{cases}
\usepackage{titlesec}
\usepackage[title]{appendix}
\usepackage{mdwlist}
\usepackage{makecell}
\usepackage{booktabs}
\usepackage{mathabx}
\usepackage{epstopdf}
\usepackage{framed}
\usepackage{tikz}
\usepackage{dsfont}
\allowdisplaybreaks[4]

\newcommand{\R}{\mathbb{R}}

\newcommand*\xbar[1]{%
	\hbox{%
		\vbox{%
			\hrule height 0.5pt 
			\kern0.5ex
			\hbox{%
				\kern-0.1em
				\ensuremath{#1}%
				\kern-0.1em
			}%
		}%
	}%
}

\newtheorem{lemma}{Lemma}[section]

\newtheorem{prop}{Proposition}[section]
\newtheorem{cor}{Corollary}[section]
\newtheorem{theorem}{Theorem}[section]
\newtheorem{remark}{Remark}[section]

\DeclareMathOperator{\supp}{supp}

\numberwithin{equation}{section}

\normalsize \makeatletter

\newcommand*\bbar[1]{%
	\hbox{%
		\vbox{%
			\hrule height 0.5pt 
			\kern0.365ex
			\hbox{%
				\kern-0.1em
				\ensuremath{#1}%
				\kern-0.1em
			}%
		}%
	}%
} 

\begin{document}

\begin{frontmatter}
\title
{
	Global well-posedness and decay rates for the three dimensional incompressible active liquid crystals
}


\author[rvt1]{Fan Yang\corref{corresponding}}
\ead{fanyang-m@ntu.edu.cn}

\author[rvt2]{Xiongfeng Yang}
\ead{xf-yang@sjtu.edu.cn}

\cortext[corresponding]{Corresponding author.}

\address[rvt1]{School of Mathematics and Statistics, Nantong University, Nantong, 226019, China}

\address[rvt2]{School of Mathematical Sciences, MOE-LSC and CMA-Shanghai,\\ Shanghai Jiao Tong University, Shanghai, 200240, China}

\begin{abstract}

This paper investigates the global well-posedness and large-time behavior of 3D incompressible active liquid crystals under constant activity, modeled by a coupled system of forced incompressible Navier-Stokes equations for the velocity and a parabolic system for the $Q$-tensor order parameter.
By employing refined commutator estimates, the existence and uniqueness of global strong solutions are proved for small initial data $(Q_0,u_0)\in H^{s+1}\times H^s$ $(s\geq 2)$ with  activity $c>c_\star$, which improves a previous result in \cite{active-limit}.
In addition, if the initial data further belong to $L^1$ and $s\geq 4$, we
obtain a mixing decay estimate on $\|\partial^kQ(t)\|_{L^2}$ that combines both
an extra exponential decay factor at a rate proportional to $(c-c_\star)\Gamma$ and the optimal algebraic decay rate that coincides with that of the heat kernel, where $k\leq s-1$. This result reveals that, in the high activity regime, active nematics become isotropic with an activity-dependent exponential convergence rate, and the estimate is stable in the infinite rotational viscosity limit, as $\Gamma\rightarrow 0$. Meanwhile, the sharp decay estimate on $\|\partial^ku(t)\|_{L^2}$ is also derived for $k\leq s-2$ with an additional initial assumption.
The proof is established via a combination of the Green's function method and the time-weighted energy method. To the best of our knowledge, these results are the first reported for active/passive nematic liquid crystals within the Beris-Edwards framework, and the enhanced decay effect of the orientational field is essentially derived from the free energy. Furthermore, in the passive setting, our result implies the phase transition of thermotropic liquid crystals at high temperatures.



\end{abstract}

\end{frontmatter}
%

%

\section{Introduction}\label{sec.1}
Active matter describes ubiquitous nonequilibrium condensed systems whose constituent elements are capable of converting energy into motion,
and these individual components are commonly known as active or self-propelled particles.
Examples include microtubule bundles \cite{exp-microtubules}, bacteria \cite{exp-bacteral}, actin filaments \cite{exp-actin}, cytoskeletal filaments \cite{exp-cyto} and schools of fish and flocks of birds \cite{exp-fish}, which exist both at the microscopic and macroscopic scale.
The unifying characteristic of such systems is that active particles interacting with each other and their surrounding medium generate collective behaviors, such as nonequilibrium order-disorder transitions and active turbulence.
In addition, many cutting-edge technologies arise along with the rapid development of active systems in the theoretical physics community; for example, active liquid crystals have succeeded in designing autonomous materials systems that can perform logic operations \cite{defect-active}.
For more information, we refer interested readers to \cite{blow2014,active-review,Giomi2011,Giomi2012,Giomi2013,Giomi2014,activematter,Ravnik2013} and the references therein.

\subsection{The PDE system}
Active liquid crystals represent a particular class of active systems, in which active particles typically possess elongated shapes, such as cytoskeletal filaments and a microtubule-kinesin film.
To date, there exist several theoretical and experimental model systems that proposed for studying the collective behaviour of active liquid crystals.
Note that such systems can be classified as polar or nematic according to the head-tail symmetry of their constituent elements (Besides, another criterion is momentum conservation).
Loosely speaking, the polar particles (e.g., bacteria, fish, birds) have distinct heads and tails, are generally self-propelled along their long axis.
They can order in a polar state or a nematic state.
A well-known system for active polar fluids is the Toner-Tu model \cite{TT-model}, where the polar order is described by a vector field $\bm{p}$; see \cite{TT-math1,TT-math2} for mathematical results. On the other hand, the head-tail symmetric particles (e.g., melanocytes) are usually referred to as "apolar", and can order in a nematic state. In addition, self-propelled rod-like entities, which are head-tail symmetric but behave like polar particles, can still form a nematic order and display exclusively apolar interactions.
In what follows, we will refer to the active liquid crystal that exhibits nematic orders as an {\em active nematic system}.
Though numerous theoretical and experimental studies have been developed for active nematics, a comprehensive understanding of their effectiveness and mechanisms is still lacking, especially from the viewpoint of mathematical analysis.

In this study, we are interested in the hydrodynamic equations for 3D incompressible flow of active nematic liquid crystals proposed in \cite{Giomi2012}, where the system is typically described by three fields. That is
\begin{equation}\label{eqiacl}
	\begin{cases}
		\partial_t c + u\cdot\nabla c = \nabla\cdot \left[ \left( D_0 \mathbb{I}_3+ D_1Q \right)\nabla c + \alpha_1 c^2\nabla\cdot Q  \right],\\
		\partial_t Q + u\cdot\nabla Q  = S(\nabla u,Q) +\lambda |Q|D+\Gamma H[Q,c],\\
		\partial_t u + u\cdot\nabla u = \mu\Delta u-\nabla p + \nabla\cdot \left( \tau+\sigma \right),\\
		\nabla\cdot u = 0, \hspace{4.4cm} x\in \R^3,\quad t>0,
	\end{cases}
\end{equation}
where $c >0$ is a scalar function represents the concentration of active particles, $u \in \R^3$ is the flow velocity, and $p$ is the scalar pressure. In addition, $Q\in \mathbb{M}^{3\times 3}$ denotes the nematic tensor order parameter, which describes the primary and secondary directions of nematic alignment along with variations in the degree of nematic order \cite{Q-tensor-book}. As usual, we shall call it "$Q$-tensor". Meanwhile, since $Q$ is a traceless symmetric tensor field, we are able to classify it according to its eigenvalues. Indeed, $Q=0$ corresponds to the isotropic phase, which has three equal zero eigenvalues.
Otherwise, the system is classified as uniaxial for two equal non-zero eigenvalues or biaxial for three distinct eigenvalues.
We also note that, for both uniaxial and biaxial cases, $Q$ can be uniquely represented by the director field $\bm{n}$ and the scalar order parameter $s$ which measures the degree of alignment.
Moreover, $D_0, D_1>0$ are the diffusion coefficients, $\mu>0$ represents the viscosity coefficient, $\Gamma^{-1}>0$ is the rotational viscosity, $\lambda\in\R$ denotes the nematic alignment parameter, $\alpha_1$ is a constant with dimensions of inverse time, and $\mathbb{I}_3$ denotes the $3\times 3$ identity matrix. 
The tensor $H[Q,c]$ represents the molecular interactions that stem from a free energy functional,
which can be written as
\[
\mathcal{F}[Q,c]=\int \left( \dfrac{K}{2}|\nabla Q|^2  + \dfrac{K}{4}(c-c_\star) {\rm Tr}(Q^2) - \dfrac{b}{3} {\rm Tr}(Q^3) + \dfrac{K}{4}c| {\rm Tr}(Q^2) |^2 \right) dA,
\]
where $K>0$ is the elastic constant, $c_\star$ is the critical concentration for the IN (isotropic-nematic) transition, and $b\in\R$ is a material-dependent constant. Note that $\mathcal{F}[Q,c]$ comprises the so-called Landau-de Gennes free energy functional for bulk nematics, augmented with appropriate couplings to $c$ (see also \cite{JFM18,Q-tensor-book}).
In three dimension, $H$ then reads
\[
H[Q,c]:= K\Delta Q -\dfrac{K}{2}(c-c_\star) Q  + b\left[ Q^2-\dfrac{{\rm Tr}(Q^2)}{3}\mathbb{I}_3\right]- Kc Q{\rm Tr}(Q^2),
\]
which obeys that
\[
H=-\frac{\delta \mathcal{F}}{\delta Q}+ \frac{\mathbb{I}_3}{3}{\rm Tr}\left(\frac{\delta \mathcal{F}}{\delta Q}\right).
\]

The term $S(\nabla u, Q)$
is defined by
\begin{equation}
	S(\nabla u, Q)=\xi D\left( Q+\dfrac{1}{3}\mathbb{I}_3 \right) + \xi \left( Q+\dfrac{1}{3}\mathbb{I}_3 \right)D - 2\xi \left( Q+\dfrac{1}{3}\mathbb{I}_3 \right){\rm Tr}(Q\nabla u) + \Omega Q- Q\Omega,\nonumber
\end{equation}
where $\xi\in\R$ is a constant that represents the ratio of tumbling and aligning effects, and $\Omega=\frac{1}{2}\left( \nabla u - \nabla u^{T} \right)$ and $D=\frac{1}{2}\left( \nabla u + \nabla u^{T} \right)$ are the vorticity and the strain-rate tensors, respectively. Note that the relative impact of strain-rate and vorticity on particle alignment with the flow is characterised by $\xi$.
In addition, $\xi=0$ is accounted for the corotational case where the molecules only tumble in a shear flow without aligning.
On the other hand, the symmetric additional stress tensor $\tau$
takes a form similar to that used in passive systems,
\begin{equation}
	\tau=-\xi \left( Q+\dfrac{1}{3}\mathbb{I}_3 \right)H - \xi H\left( Q+\dfrac{1}{3}\mathbb{I}_3 \right) + 2\xi \left( Q+\dfrac{1}{3}\mathbb{I}_3 \right){\rm Tr}(QH) - \nabla Q \odot\nabla Q,\nonumber
\end{equation}
where we denote by $\nabla Q \odot\nabla Q$ the $3\times 3$ matrix whose $(\alpha,\beta)$-entry is $\partial_\beta Q_{\gamma\delta}\partial_\alpha Q_{\gamma\delta}$, $1\leq\alpha,\beta\leq 3$.
In what follows, we will use a partial Einstein summation convention, i.e., the
repeated indices are summed over.

The stress tensor $\sigma$ can be written as a sum of two contributions: $\sigma=\sigma^r + \sigma^a$. The first is the contribution from nematic elasticity, similar to that in passive system,
\[
\sigma^r = -\lambda |Q|H+ QH-HQ.
\]
The second contribution stems from the interaction of active particles, which is typically defined by
\[
\sigma^a = \alpha_2 c^2Q.
\]
Observed that the $c^2$ dependence again arises because of the propulsion, similar to that in the first equation of \eqref{eqiacl}.
Moreover, this active term represents the contractile ($\alpha_2>0$) or extensile ($\alpha_2<0$) stress exerted by the active particles, see \cite{activeterm2,activeterm1}.

From an analytical perspective, the full system \eqref{eqiacl} presents considerable challenges in studying global weak solutions due to the difficulty in performing a priori estimate.
Hence, in the case of non-constant activity, the work in \cite{weak-ALC-comp} first studied the corresponding simplified model of compressible flow under additional parameter restrictions, where the estimate of $\|c\|_{L^\infty}$ is obtained by the maximum principle in order to reach the desired a priori estimate. Alternatively, as shown in \cite{active-limit}, one can construct strong solutions to system \eqref{eqiacl} with sufficiently regular data, where we still have the bound of $\|c\|_{L^\infty}$ by the standard embedding theorems.
However, these studies are far from what we expected, especially for weak solutions or lower regularity solutions which may serve important roles for understanding the complex behavior (e.g., collective motion, pattern formation and active turbulence) of active liquid crystals. Nevertheless, these problems will be further explored in our further studies.

In the rest of this paper, we still focus on the case of constant activity, i.e., $c>0$.
In this setting,
the well-studied simplified corotational model of active system \eqref{eqiacl} in dimension three reads (see \cite{weak-ALC,Giomi2011,Giomi2013})
\begin{equation}\label{eq1.1}
	\begin{cases}
		\partial_tQ+(u\cdot\nabla)Q + Q\Omega-\Omega Q-\lambda |Q|D=\Gamma H[Q];\\
		\partial_t u + (u\cdot\nabla)u +\nabla p-\mu \Delta u =-\nabla\cdot (\nabla Q \odot\nabla Q)+\nabla\cdot (Q\Delta Q-\Delta QQ)\\
		\hspace{5.3cm}-\lambda\nabla\cdot(|Q|H[Q])+\kappa\nabla\cdot Q;\\
		\nabla\cdot u = 0,
	\end{cases}
\end{equation}
where $(t,x)\in\R_+\times \R^3$, $\mu$, $\Gamma>0$, $\lambda$, $b\in \R$, and $a$, $\kappa$, $H[Q]$ obey that
\begin{equation}\label{eq1.2}
	a=\dfrac{1}{2}(c-c_\star),\quad \kappa=\alpha_2c^2
\quad {\rm and} \quad
H[Q]= \Delta Q-aQ+b\left[Q^2-\dfrac{{\rm Tr}(Q^2)}{3}\mathbb{I}_3\right]-cQ{\rm Tr}(Q^2).
\end{equation}
Here we have taken $K=1$ in the free energy without loss of generality. It is obvious that
the velocity field $u$ is modeled by the Navier-Stokes equation with an additional active force, and the evolution of $Q$-tensor is described by the Beris-Edwards framework.
Hence, loosely speaking, the system \eqref{eq1.1} can be viewed as the Beris-Edwards system by incorporating additional hydrodynamic stresses from activity.
Moreover, the above system is supplemented with the initial data
\begin{equation}
	(Q,u)(t,x)\mid_{t=0}=(Q_0,u_0)(x) \quad \text{\rm for}\quad x\in\R^3.
\end{equation}

It is important to emphasize that, even for the above simplified system, 
there still exists the contribution induced by the active stress $\kappa Q$.
In addition, the effect of this active stress on defect dynamics has been observed in \cite{Giomi2013}.
Based on this feature, the dynamics of topological defects under active stresses can be compared to the transport of electrons in the presence of a voltage gradient. Therefore, the system \eqref{eq1.1} has significant application value for performing logic operations and transmitting information within active materials \cite{defect-active}.

\subsection{A brief overview of related literature}

The Beris-Edwards model, which originated in the study of liquid crystals, is now successfully applied in modeling active nematics.
For the simplified system \eqref{eq1.1}, one can also see that it indeed has a similar structure as the well-studied Beris-Edwards system when the activity strength parameter $\kappa$ vanishes.
Therefore, we begin with a short review of some related results for the Beris-Edwards model in order to facilitate the comparison between active nematics and their passive counterparts.

First progress on well-posedness is due to the work of Paicu and Zarnescu \cite{Qtensor12}, in which they proved the existence of global weak solutions in $\R^d$ $(d=2,3)$ for the case $\xi=0$. Meanwhile, they also studied the existence of regular solutions as well as weak-strong uniqueness for dimension two. In \cite{Qtensor11}, they further generalized the mentioned results to the case where $\xi$ is sufficiently small. However, in two dimensional periodic domain, it was shown that such smallness assumption on $\xi$ is unnecessary for the existence of regular solutions \cite{2dwithoutxi}. Additionally, the weak-strong uniqueness result for the corotational and non-corotational settings was investigated in \cite{weak-Q-tensor} and \cite{weak-strong-LCD-xi}, respectively.
For a bounded domain $\mathcal{O}\subset\R^3$ with smooth boundary, Xiao \cite{Q-unique-3d} studied the existence of global strong solutions to the corotational system for small initial data, as well as weak-strong uniqueness.
His approach is based on the maximal regularities of Stokes and parabolic operators.
Later on, Hieber, Hussein and Wrona \cite{xi-all} addressed the existence of global strong solutions for all $\xi\in\R$, where their method is built on the associated quasi-linear system and its maximal regularity property.
Moreover, Wilkinson \cite{Q-tensor-Torus} considered the system with a singular potential introduced by Ball and Majumdar \cite{potential} under the periodic boundary conditions, and proved the existence of strictly physical global weak solutions for $\xi=0$. A nonisothermal variant of the Beris-Edwards model with the Ball-Majumdar potential was also discussed in \cite{nonisothermal1,nonisothermal2}, where the existence of global weak solutions was obtained for periodic boundary conditions.
Furthermore, Wang, Xu and Yu \cite{Q-tensor-comp} established the existence and long-time dynamics of globally defined weak solutions for the coupled compressible Navier-Stokes equations and $Q$-tensor system. For other results, especially concerning modified Beris-Edwards models with general free energy or stress tensor, we refer interested readers to \cite{ables2014,ables2016,2dlowSobolev,A-Z2016-2d,Du-wang-20,Hang-Ding-15,LC-decay-19,LC-decay-22} and references therein.
 
Concerning the active liquid crystals, a few studies exist that focus on the mathematical analysis of the related systems.
As was shown in \cite{weak-ALC}, Chen, Majumdar, Wang and Zhang first established the existence of global weak solutions for the simplified active system \eqref{eq1.1} in dimension two and three.
However, due to technical reasons, they were only able to obtain the existence of regular solutions as well as weak-strong uniqueness for dimension two.
Later on, the first author and Li \cite{weak-strong-3d} studied the weak-strong uniqueness for the Leray-Hopf type weak solutions of system \eqref{eq1.1} in $\R^3$, and a uniqueness criterion was found. Recently, a Koch-Tataru theorem for \eqref{eq1.1} was also presented in \cite{ALC-BMO} by the first author,
where the local well-posedness at critical regularity was established for small initial data $(Q_0,u_0)\in L^\infty\times {\rm BMO}^{-1}$.
Moreover, Lian and Zhang \cite{Lian-zhang20} derived global weak solutions for the inhomogeneous version of system \eqref{eq1.1} in smooth bounded domains.
For the compressible fluid model of active nematics, the existence of global weak solutions of initial-boundary value problems was studied in \cite{weak-ALC-comp}.
As for the full active nematic system \eqref{eqiacl}, Jiang, Tang and Wang \cite{active-limit} first investigated the strong well-posedness in $\mathbb{T}^d$ with $d=2,3$, along with the interesting zero activity limit problem.
For stochastic analysis and other related results, we refer the reader to \cite{active-servy,Jiang-com-s,stochastic-alc} and references therein.

\subsection{Motivation and main results}
As we have mentioned above, 
the 2D simplified system \eqref{eq1.1} admits higher regular solutions that satisfy
\[
	Q\in L^\infty\left(0,T;H^{s+1} \right)\cap L^2\left(0,T;H^{s+2} \right)\,\,\text{and}\,\,u\in L^\infty\left(0,T;H^s \right)\cap L^2\left(0,T;H^{s+1} \right),
\]
where the initial data $(Q_0,u_0)\in H^{s+1}\times H^s$ with $s>0$. 
The existence of such regular solutions in dimension three has been confirmed in a recent work \cite{active-limit}.
Although their existence results are established for the full active system \eqref{eqiacl} in $\mathbb{T}^d$ ($d=2,3$),
it is obvious that the same is true for system \eqref{eq1.1} which is a special case of \eqref{eqiacl}.
However, compared to that known for dimension two, we need more regular initial data in the 3D case to overcome the difficulties caused by the inability to use the logarithmic embedding inequality; see also \cite{Q-unique-3d}. Loosely speaking, similar results may hold in $\R^3$ for the initial data $(Q_0,u_0)\in H^{s+1}\times H^s$ with the integer $s> \frac{d}{2}+1$, but there exists a gap between the Sobolev index $s$. Recall that the essence of these results lies in obtaining the estimates on $\|Q(t,\cdot)\|_{W^{1,\infty}}$ and $\|u(t,\cdot)\|_{L^\infty}$, which allow us to perform the higher regularity estimates of the underlying approximation systems. In view of the standard Sobolev embeddings, we are aware that the integer $s\geq 2$ is enough to get the desired control in $L^\infty$ norm. Therefore, this motivates us to revise the global well-posedness of the 3D simplified system \eqref{eq1.1} in the $H^s$-framework in order to relax the initial requirements, which serves as the first contribution of this paper.

Partially inspired by \cite{active-limit}, we will utilize an approximation system of \eqref{eq1.1} and prove the local well-posedness by the energy method.
Technically, to achieve the key a priori estimate under a slightly lower regularity than that of \cite{active-limit}, the refined commutator estimates have been employed to deal with certain nonlinear terms. Consequently, our first main result regarding the existence and uniqueness of local-in-time solutions to system \eqref{eq1.1} is stated as follows:
\begin{theorem}\label{thm1.1}
	Let $s\geq 2$ be an integer, $\widetilde{C}$ is a positive constant defined as in \eqref{eq2.11} and $E_0:= \|Q_0\|_{H^{s+1}}^2 + \|u_0\|_{H^s}^2$ be the initial energy.
	Then there exists a time $T$ with
	\begin{equation}
		0<T\leq T^\ast:= \dfrac{\min\{E_0, \ln 2\}}{\sum_{l=0}^3 4^{l}\widetilde{C}E_0^{l}}
	\end{equation}
	such that the Cauchy problem of system \eqref{eq1.1}-\eqref{eq1.2} with initial data $(Q_0,u_0)\in H^{s+1}\times H^s$ admits a unique local solution
	\begin{equation}
		Q\in L^\infty\left( 0,T;H^{s+1} \right)\cap L^2\left( 0,T;H^{s+2} \right)\quad\text{and}\quad u\in L^\infty\left( 0,T;H^s \right)\cap L^2\left( 0,T;H^{s+1} \right).
	\end{equation}
	Moreover, $(Q,u)$ satisfies the following energy inequality
	\begin{equation}
		\|u(t)\|_{H^s}^2 + \|Q(t)\|_{H^{s+1}}^2 
		+\int_0^t \mu\|\nabla u(\tau)\|_{H^{s}}^2 + (|a|+1)\Gamma\|Q(\tau)\|_{H^{s+1}}^2 + \Gamma\|\Delta Q(\tau)\|_{H^{s}}^2 d\tau \leq C_0,
	\end{equation}
	for all $t\in [0,T]$, and $C_0$ is a positive constant depending only on $E_0$, $s$, $|a|$, $b$, $c$, $\mu$, $\kappa$, $\lambda$ and $\Gamma$.
\end{theorem}
\begin{remark}
	Note that our result in Theorem~\ref{thm1.1} holds for any large initial data $(Q_0,u_0)$ in $H^{s+1}\times H^s$ with the integer $s\geq 2$ and $a\in\R$, which indeed relaxes the initial requirements of \cite{active-limit}. Moreover, it also differs from the known result regarding the existence of local strong solutions for its passive counterpart in \cite{Q-unique-3d}, where the smallness of $\|Q_0\|_{L^\infty}$ is essential for conducting the a priori estimate for the proposed iteration scheme by utilizing the maximal regularities of Stokes and parabolic operators in Besov spaces.
	%
\end{remark}
\begin{remark}
	We emphasize that $|Q|$ appearing in system \eqref{eq1.1} also induces strong nonlinearity.
	Thus we need more careful treatments on the estimates of the related coupling terms. As shown in \cite{active-limit}, one can see that the $H^s$ norm of $|Q|$ can be controlled by $\|Q\|_{H^s}$ in $\mathbb{T}^d$, which enables us to achieve the desired estimates. In our case, we can still apply this feature because, based on the fact that norms $\|\cdot\|_{\dot{H}^s}$ and $\|\cdot\|_{\dot{B}_{2,2}^s}$ are equivalent, one can generalize it to the whole space case (see Lemma~\ref{lem-A6} for the details).
\end{remark}

Based on the above result, we further show that this local solution is indeed global when the initial data is sufficiently small. However, we need to impose the condition $a>0$ such that the linear term $-aQ$ from the free energy exhibits damping effects. This parameter arises in the definition of dissipation energy functional that used in performing an a priori estimate for system \eqref{eq1.1} via the energy method. Otherwise, we cannot use the dissipation effect of $Q$ to overcome
the difficulty caused by the linear active force $\kappa\nabla\cdot Q$; see \cite{active-limit} and Section~\ref{sec-3} below for the details. Then, together with the standard continuation argument, we can conclude the global existence and uniqueness of solutions to \eqref{eq1.1}-\eqref{eq1.2} in the $H^s$-framework. More precisely, our second main result is stated as follows.
\begin{theorem}\label{thm1.2}
	Let $s\geq 2$ be an integer and $a>0$. If there exists a small real number $\delta_0$ such that $E_0\leq \delta_0 $, then the Cauchy problem of system \eqref{eq1.1}-\eqref{eq1.2} with initial data $(Q_0,u_0)\in H^{s+1}\times H^s$ admits a unique global solution $(Q,u)$ satisfying
	\begin{equation}
		Q\in L^\infty\left(  [0,+\infty);H^{s+1} \right)\cap L^2\left(  [0,+\infty);H^{s+2} \right)~\text{and}~u\in L^\infty\left( [0,+\infty);H^s \right)\cap L^2\left(  [0,+\infty);H^{s+1} \right).
	\end{equation}
	Moreover, $(Q,u)$ satisfies the following energy inequality
	\begin{equation}
		\|u(t)\|_{H^s}^2 + \|Q(t)\|_{H^{s+1}}^2 
		+\int_0^t \|\nabla u(\tau)\|_{H^{s}}^2 + \|  Q(\tau)\|_{H^{s+1}}^2 +\|\Delta
		Q(\tau)\|_{H^{s}}^2 d\tau\leq C,
	\end{equation}
	for all $t\geq 0$,
	and $C$ is a positive constant depending on $\delta_0$, $s$, $a$, $b$, $c$, $\mu$, $\kappa$, $\lambda$ and $\Gamma$.
\end{theorem}
\begin{remark}
		Observed that our global existence result again holds for lower regular initial data than that of \cite{active-limit}.
		Unfortunately,
		we are not sure whether the local solutions established in Theorem~\ref{thm1.1} can globally exist for the case of $a\leq 0$, even if the initial data is sufficiently small. Moreover, compared with the result known for the 2D case, there still is a certain gap. In other words, the well-posedness for initial data $(Q_0,u_0)\in H^2\times H^1$ is unclear. Our future work will explore these directions.
\end{remark}

Now, we turn to the discussion of the large-time behavior of global solutions to the simplified active nematic system \eqref{eq1.1} close to the trivial equilibrium state.
In particular, we will restrict ourselves to the case $a>0$ and the integer $s\geq 4$.
In this setting, we will present the optimal time-decay estimates for the global classical solutions of \eqref{eq1.1}
and their spatial derivatives up to some order.

Indeed, one can find that there still exist several analytic works on the decay rates for the passive version of system \eqref{eq1.1}.
In \cite{Dai2016-decay}, the authors first deduced the following decay properties of weak solutions in the case of $\xi=0$:
\begin{equation}\label{eq1.10}
\|Q(t,\cdot)\|_{H^1} +\|u(t,\cdot)\|_{L^2} \lesssim (1+t)^{-\frac{3}{4}}.
\end{equation}
Later, Schonbek and Shibata \cite{LC-decay-19} obtained optimal decay rates for the $Q$-tensor in the $L^p$ norm $(1\leq p\leq \infty)$. Their result especially showed that the $Q$-tensor tends to zero in the $L^\infty$ norm with a rate $t^{-\frac{3}{2}}$, which was not previously observed in \cite{Dai2016-decay}.
In light of these studies, we hence expect that a similar result holds for the active case.

To achieve this goal, we will employ the Green's function method which is different from that used in \cite{Dai2016-decay,LC-decay-19}.
In particular, 
we investigate the precise behavior of the Green function of the following linearized system of \eqref{eq1.1}:
\begin{equation}\label{eq2.1}
	\begin{cases}
		Q_t -\Gamma \Delta Q+a\Gamma Q =0,\\
		u_t-\mu \Delta u -\kappa \mathbb{P}{\rm div} Q=0,\\
		(Q,u)(0,x)=(Q_0,u_0)(x).
	\end{cases}
\end{equation}
Let $\mathbb{G}_{Q,u}(t,x)$ be the Green function associated with \eqref{eq2.1}. Then by Duhamel's principle, one can conduct the decay estimates via the integral representation of the solution through $\mathbb{G}_{Q,u}$.
However, we do not intend to analyze the system \eqref{eq2.1} directly, because we observe that the structure of this linearized system is similar to that of the incompressible Oldroyd-B model (see \cite{Huang-Old-B2022,Huang-Old-B2024,Zi2014}). 
On the basis of these works, we apply $\mathbb{P}{\rm div}$ to the first equation of \eqref{eq2.1} and study the following auxiliary system instead:
\begin{equation}\label{eq2.2}
	\begin{cases}
		\sigma_t -\Gamma \Delta \sigma +a\Gamma \sigma =0,\\
		u_t-\mu \Delta u -\kappa \sigma=0,\\
		(\sigma,u)(0,x)=(\sigma_0,u_0):= (\mathbb{P} {\rm div}Q_0,u_0)(x),
	\end{cases}
\end{equation}
where $\sigma :=\mathbb{P} {\rm div}Q$. Again, $\mathbb{G}_{\sigma,u}(t,x)$ refers to the Green function of the above system.
In addition, one can see that \eqref{eq2.2} can be solved explicitly using the Fourier transform, rather than solving it directly.
Precisely, one has
\begin{equation}
	\widehat{\mathbb{G}_{\sigma,u}}(t,\xi)=
	\begin{bmatrix}
		\mathcal{A}(t,\xi)\mathbb{I}_3 & 0 \\
		\mathcal{B}(t,\xi)\mathbb{I}_3 &  \mathcal{C}(t,\xi)\mathbb{I}_3
	\end{bmatrix}.
\end{equation}
For the details please refer to
Section~\ref{Sec.4}.
Then, to understand the decay properties of the linear system \eqref{eq2.2},
we establish a rigorous analysis of $\widehat{\mathbb{G}_{\sigma,u}}(t,\xi)$ by utilizing the Fourier splitting method.
Together with the relation between $\sigma$ and $Q$, we eventually find
\begin{align}
	\left\| \partial_x^k Q_L(t)\right\|_{L^2} &\lesssim  e^{-a\Gamma t}(1+t)^{-\frac{3}{4}-\frac{k}{2}},\quad 0\leq k\leq s+1,\label{eq1.13}\\
	\left\| \partial_x^k u_L(t)\right\|_{L^2} &\lesssim (1+t)^{-\frac{3}{4}-\frac{k}{2}},\hspace{1.32cm} 0\leq k\leq s,\label{eq1.14}
\end{align}
where $(Q_L,u_L)$ denotes the corresponding solution of linear system \eqref{eq2.1}. The intriguing fact about the above decay rates is that a partial contribution from the free energy produces an extra exponential decay of $Q$-tensor field, unlike in \eqref{eq1.10}.
Note that \eqref{eq1.13} can also be referred to as a {\em stable mixing estimate} in the infinite rotational viscosity limit with a fixed constant activity $c>c_\star$, because there is no harm in setting $\Gamma=0$; this still indicates the algebraic time-decay rates for the $Q$-tensor. See also \cite{mixing-decay} for this concept. Needless to say, 
we then aim to extend such decay estimates to the full nonlinear system \eqref{eq1.1}.
%
%

However, when performing the optimal time-decay bounds of the solution to \eqref{eq1.1} by the Green's function method, the major difficulty lies in certain nonlinear terms, for example, $u\cdot\nabla Q$ and $Q\Delta Q$.
Basically, having the explicit solution by Duhamel's principle, such nonlinear terms involving the highest derivatives will induce a problem of loss of derivatives in nonlinear decay estimates
(see \eqref{eq5.1}-\eqref{eq-5.2}).
Hence, we cannot expect the decay rates of the highest derivatives of the solution to be optimal. To overcome this obstacle, we present the time-weighted estimates to exhibit some decay properties of the highest derivatives. Although their rates are not optimal, they are enough for us to address the expected time-decay estimates of solutions to \eqref{eq1.1}. In addition to this, we also found that deriving a stable mixing estimate akin to \eqref{eq1.13} involves a more complicated analysis, which again caused by some higher order nonlinear terms.
Nevertheless, 
we will discuss these specific difficulties after stating our main results.
%
More precisely, we have the following upper time-decay estimates of the solutions.
\begin{theorem}\label{Thm1.3}
	Let $s\geq 4$ and $(Q,u)$ be the global solution to the Cauchy problem of system \eqref{eq1.1}-\eqref{eq1.2} constructed in Theorem~\ref{thm1.2}.
	Suppose that $(Q_0,u_0)\in L^1$ and
	\[
	\widebar{E}_0:=\|Q_0\|_{H^{s+1}} + \|u_0\|_{H^s} + \|(Q_0,u_0)\|_{L^1}.
	\]
	Then we have the following time-decay estimates of the solution:
	\begin{equation}\label{eq:1.11}
			\begin{split}
				\|\partial^kQ(t) \|_{L^2}  \leq& C\widebar{E}_0(1+t)^{-\frac{3}{4}-\frac{k}{2}}e^{-\frac{(c-c_\star)\Gamma t}{4}}, \quad \text{\rm for all $k\leq s-1$},\\
				\|\partial^ku(t) \|_{L^2}  \leq& C\widebar{E}_0(1+t)^{-\frac{3}{4}-\frac{k}{2}}, \hspace{1.99cm} \text{\rm for all $k\leq s-2$},
			\end{split}
	\end{equation}
	and
	\begin{equation}\label{eq:1.12}
	\begin{split}
		\|\partial^k u(t)\|_{L^2} \leq& C\widebar{E}_0(1+t)^{-\frac{k}{2}},\quad \text{\rm when $k=s-1$},\\
		\|\partial^{k} Q(t)\|_{H^1}+ \|\partial^k u(t)\|_{L^2} \leq& C\widebar{E}_0(1+t)^{-\frac{k}{2}},\quad \text{\rm when $k= s$},
	\end{split}
	\end{equation}
	for any $t\geq 0$ and generic positive constant $C$ which depends only on $s$, $b$, $c$, $c_\star$, $\mu$, $\kappa$, $\lambda$ and $\Gamma$.
	Meanwhile, it holds that
	\begin{equation}\label{eq:1.13}
		\begin{split}
			\|\partial^kQ(t) \|_{L^\infty}  \leq& C\widebar{E}_0(1+t)^{-\frac{3}{2}-\frac{k}{2}}e^{-\frac{(c-c_\star)\Gamma t}{4}}, \quad \text{\rm for all $k\leq s-3$},\\
			\|\partial^ku(t) \|_{L^\infty}  \leq& C\widebar{E}_0(1+t)^{-\frac{3}{2}-\frac{k}{2}}, \hspace{1.99cm} \text{\rm for all $k\leq s-4$}.
		\end{split}
	\end{equation}
	Moreover, for all $p\in [2,\infty]$, one has
	\begin{equation}\label{eq:1.14}
		\begin{split}
			\|\partial^\beta Q(t) \|_{L^p}  \leq& C\widebar{E}_0(1+t)^{-\frac{3}{2}+\frac{3}{2p}-\frac{\beta}{2}}e^{-\frac{(c-c_\star)\Gamma t}{4}},\quad \text{\rm for $0\leq \beta \leq s+ \tfrac{3}{p}-\tfrac{5}{2}$},\\
			\|\partial^\beta u(t) \|_{L^p}  \leq& C\widebar{E}_0(1+t)^{-\frac{3}{2}+\frac{3}{2p}-\frac{\beta}{2}}, \hspace{1.99cm}  \text{\rm for $0\leq \beta \leq s+ \tfrac{3}{p}-\tfrac{7}{2}$}.
		\end{split}
	\end{equation}
\end{theorem}
\begin{remark}
From the above stable mixing estimates, it can be seen that the threshold of the effect on the exponential decay of the orientational field $Q$ is indeed related to the Landau-de Gennes free energy.
In particular, due to the contribution of the free energy, active fluids with nematic phase will converge to the isotropic phase ($Q=0$) with an exponential decay rate that depends on the constant activity number $c$ and rotational viscosity $\Gamma^{-1}$. And this rate can be enhanced by increasing the value of the active constant. Moreover,
optimal algebraic time-decay rates compared with those of the heat equation can still be obtained for the orientational field $Q$ in the absence of the free energy, where the rotational viscosity $\Gamma^{-1}$ tends to infinity in the system \eqref{eq1.1}.
To the best of our knowledge, such theoretical observations are reported for the first time in the literature of active nematics. However, it is worth noting that our main result is proved for a simplified active system of \eqref{eqiacl} for small initial perturbations around the equilibrium state; the active forcing induced by constant activity is seen to be irrelevant and the system behaves like its passive model. Therefore, we believe that the result in Theorem~\ref{Thm1.3} may not be tenable once we allow the active concentration to change; it will largely depend on the active parameters $\alpha_1$, $\alpha_2$ in system \eqref{eqiacl} and the size of initial data. See also \cite{Giomi2012,active-limit} and references therein.
\end{remark}

\begin{remark}
It is important to mention that Theorem~\ref{Thm1.3} is also applicable to the passive system of \eqref{eq1.1} discussed in \cite{Qtensor12}.
This is because one can essentially think of the associated linearized system as the system \eqref{eq2.2} with $\kappa=0$. In this case, the corresponding kernel that we deal with has a diagonal form which is equal to $\widehat{\mathbb{G}_{\sigma,u}}(t,\xi)$ with $\mathcal{B}=0$. Since the exponential time-decay is related to the kernel $\mathcal{A}$. Together with the fact that the nonlinear structure of the passive system is similar to that of \eqref{eq1.1}, the desired result can be achieved through our method with minor modifications.
Recall that the Landau-de Gennes free energy in the original theory of liquid crystals is a function that depends on temperature rather than activity. Hence, Theorem~\ref{Thm1.3} corresponding to the passive case will reveal that the nematic phase transitions into the isotropic phase at an exponential decay rate that depends on temperature. In addition, it is obvious that the orientational field will decay fast at high temperatures. These observations again seem to be new and agree with the facts that we know for thermotropic liquid crystals whose phase transitions occur in a certain temperature range.
In other words, thermotropic liquid crystals undergo a phase transition as temperature is increased, from crystalline to layered structure, then nematic, and finally the isotropic state; see \cite{LCs1992,Q-tensor-book,Khoo-book,review-thermal}. 
Nevertheless, we omit the proof and leave it to the interested readers.
\end{remark}
Furthermore, by imposing an additional condition on the initial data $u_0$, we also obtain the following result regarding the sharp decay characterization for the velocity field $u$ in the $L^2$ framework.

\begin{theorem}\label{Thm1.4}
	Let the assumptions of Theorem~\ref{Thm1.3} be satisfied and
	\begin{equation}\label{eq1.20}
		\int_{\R^3} u_0(x) dx \ne 0. 
	\end{equation}
	Then there exists some positive constant $C$ such that
	\begin{equation}
		C^{-1} (1+t)^{-\frac{3}{4}-\frac{k}{2}}\leq \left\|\partial^k u(t) \right\|_{L^2}
		\leq  C (1+t)^{-\frac{3}{4}-\frac{k}{2}},\quad t>0,
	\end{equation}
	holds for all $0\leq k\leq s-2$.
\end{theorem}
\begin{remark}
	This result indicates that the upper bounds of algebraic time-decay rates of the velocity field obtained in Theorem~\ref{Thm1.3} are indeed optimal.
	In particular, inspired by Kagei and Kobayashi \cite{Kagei02,Kagei05}, we investigated the special case where the initial data $u_0\in L^1(\R^3)$ satisfy the additional condition \eqref{eq1.20}. Note that, by the continuity of $\widehat{u_0}(\xi)$ near $\xi=0$, there exists a small constant $r>0$ such that $\widehat{u_0}(\xi)>0$ for $|\xi|<r$. This fact enables us to obtain the lower bounds for decay rates; see Section~\ref{sec5.4} for the details.
\end{remark}

\begin{remark}
	 It seems not easy to achieve the lower bounds for the decay rates of the orientational field $Q$, because the leading contribution stems from the nonlinear parts.
\end{remark}

\subsection{Difficulties and ideas}
We now explain the aforementioned technical challenges in proving Theorem~\ref{Thm1.3}. First, let us admit that \eqref{eq:1.12} can be shown by time-weighted estimates. Then, when performing the stable mixing estimate of $\partial^{s-1}Q$,
two difficulties arise from certain nonlinear coupling terms,
which we outline below.
\begin{itemize}
	\item Heuristically, the first one is caused by ${\rm div} \left(u\otimes Q\right)$. Indeed, we have
	\begin{align*}
		\left\|\partial^{s-1}Q\right\|_{L^2}\lesssim & \left\|\int_{\frac{t}{2}}^t \check{\mathcal{A}}(t-s) \ast \partial^{s-1}{\rm div}(u\otimes Q)(s) ds \right\|_{L^2} + \mathcal{R}_{\text{remainder}}\\
		\lesssim &  \int_{\frac{t}{2}}^t e^{-a\Gamma(t-s)} (1+t-s)^{-\frac{3}{4}}\| \partial^{s} (u\otimes Q)(s) \|_{L^1\cap L^2} ds + \mathcal{R}_{\text{remainder}} \\
		\lesssim &  \int_{\frac{t}{2}}^t e^{-a\Gamma(t-s)} (1+t-s)^{-\frac{3}{4}}\|u\|_{L^\infty}\| \partial^{s} Q \|_{L^2} ds + \widetilde{\mathcal{R}}_{\text{remainder}}\\
		\lesssim &  \int_{\frac{t}{2}}^t e^{-a\Gamma(t-s)} (1+t-s)^{-\frac{3}{4}}(1+s)^{-\frac{3}{2}-\frac{s}{2}} ds + \widetilde{\mathcal{R}}_{\text{remainder}}\\
		\lesssim  &  (1+s)^{-\frac{3}{4}-\frac{s}{2}} + \widetilde{\mathcal{R}}_{\text{remainder}}
	\end{align*}
	Here we used \eqref{eq:1.12}, i.e., $\|\partial^{s} Q(t) \|_{L^2}\lesssim (1+t)^{-\frac{s}{2}}$ and that $u$ has algebraic time-decay rates in the $L^\infty$ norm.
	Note that we cannot expect the optimal time-decay rate for $\partial^{s}Q$ because of the loss of derivatives.
	Thus, with the above estimate, getting a stable mixing estimate of $\partial^{s-1}Q$ in $L^2$ cannot be expected. To overcome this obstacle, we evaluate the following one instead:
	\begin{align*}
		\Bigg\|\int_{\frac{t}{2}}^t \nabla\check{\mathcal{A}}(t-s) \ast &\partial^k (u\otimes Q)(s) ds \Bigg\|_{L^2}\\
		\lesssim &  \int_{\frac{t}{2}}^t e^{-a\Gamma(t-s)}\left( (1+t-s)^{-\frac{5}{4}}\| \partial^{k} (u\otimes Q)(s) \|_{L^1} + \frac{\| \partial^{k} (u\otimes Q)(s) \|_{L^2}}{\sqrt{t-s}} \right) ds \\
		\lesssim & (1+t)^{-\frac{3}{4}-\frac{k}{2}}e^{-\frac{a\Gamma t}{2}},
	\end{align*}
	which holds for all $k\leq s-1$. It is obvious that stable mixing estimates can be achieved in this way; see Section~\ref{Sec.5} below for the details.
	\item Secondly, there is an additional difficulty caused by the remainder term like $Q\nabla u$. For example, one may have
	\begin{align*}
		\left\| \int_0^{\frac{t}{2}}\partial^k \check{\mathcal{A}}(t-s) \ast (Q\nabla u)(s) ds\right\|_{L^2}
		\lesssim  &  \int_0^{\frac{t}{2}}  e^{-a\Gamma(t-s)} (1+t-s)^{-\frac{3}{4}-\frac{k}{2}} \| (Q\nabla u)(s) \|_{H^k \cap L^1} ds \\
		\lesssim  & \int_0^{\frac{t}{2}}  e^{-\frac{a\Gamma t}{2}} (1+t-s)^{-\frac{3}{4}-\frac{k}{2}} (1+s)^{-\frac{3}{4}}  ds \\
		\lesssim&  (1+t)^{-\frac{1}{2}-\frac{k}{2}}e^{-\frac{a\Gamma t}{2}}.
	\end{align*}
	Obviously, with the current way of performing the $L^2$ decay estimate of $\partial^kQ$, we cannot obtain the desired result similar to \eqref{eq:1.11}. In fact, one way to overcome this is to get the gain of the decay factor $(1+t)^{-\frac{1}{2}}$  
	by exchanging half of the exponential decay, but the stable mixing estimates will be slightly worse than those of the corresponding linearized system. In addition, when building decay estimates in such a way, it should be noted that we only need the initial data bounded in the $L^1$ norm.
	It is enough for us to get
	the energy inequality shown in Proposition~\ref{prop5.2},
	which allows us to justify our result without the smallness of the low-frequency part of $\|(Q_0,u_0)\|_{L^1}$. This is why we are willing to compromise on the decay results in Theorem~\ref{Thm1.3}.
\end{itemize}
	\begin{remark}
		As already emphasized,
		since we do not want to impose the low-frequency assumption on the initial perturbation, it is shown herein that
		the time-decay rates of the global solution $(Q,u)$ are not entirely consistent with that observed from the solution $(Q_L,u_L)$ to linear system \eqref{eq2.1}.
		However, we would like to mention that
		our approach
		is flexible enough, as
		the stable mixing estimates consistent with \eqref{eq1.13}
		can be achieved by a new energy inequality provided that
		the low-frequency assumption on the initial data with respect to $L^1$ is suitably small; see Remark~\ref{rk5.2} for the details.
		%
		
	\end{remark}

\subsection{Organization of this paper}

The rest of this paper is organized as follows. In Section~\ref{Sec.2}, we establish the existence and uniqueness of local-in-time solutions for the active system \eqref{eq1.1} in the framework of Sobolev spaces.
Then, Section~\ref{sec-3} is used to study the global well-posedness of \eqref{eq1.1} for small initial data.
Thereafter, 
we will give the explicit expressions for the Fourier transform of the Green function of the associated linearized systems in Section~\ref{Sec.4}.
In addition, the decay properties of linear system \eqref{eq2.1} are characterized by a delicate analysis on the Green functions.
In Section~\ref{Sec.5}, we present the time-weighted energy estimates for the global classical solution, and derive the upper bounds for the decay rates of $(Q,u)$ as stated in Theorem~\ref{Thm1.3}.
Moreover, Section~\ref{sec5.4} is devoted to the sharp time-decay characterization for the velocity field.
Finally,
we also recall some technical inequalities and provide some important preliminary estimates in Appendix~\ref{AP}.

\subsection{Notations}
For brevity, we use intensively the notation $X\lesssim Y$, which means $X\leq CY$ with some universal positive constant $C$.
Similarly, the notation $X\sim Y$ is used to indicate that there exists a positive constant $C$ such that $C^{-1}Y\leq X\leq CY$.
For a real-valued function $f(x)$ defined in $\R^d$, we denote the partial derivative of order $\alpha$ by
\[
\partial^{\alpha}f(x) = \dfrac{\partial^{|\alpha|} f(x) }{\partial_{x_1}^{\alpha_1}\partial_{x_2}^{\alpha_2}\cdots \partial_{x_d}^{\alpha_d}},
\]
where the multi-index $\alpha=\left(\alpha_1,\alpha_2,\cdots,\alpha_d \right)\in \mathbb{N}^d$ with $|\alpha|=\sum_{j=1}^d \alpha_j$.
For $1\leq p\leq \infty$, we denote by $L^p$ the usual Lebesgue spaces on $\R^3$ with norm $\|\cdot\|_{L^p}$. In addition, $\|\cdot\|_{H^s}$ and $\|\cdot\|_{\dot{H}^s}$ stand
for the norms on the usual Sobolev spaces $H^s(\R^3)$ and homogenous Sobolev spaces $\dot{H}^s(\R^3)$. Moreover,
$\dot{B}_{p,q}^s(\R^3)$ denotes the homogenous Besov spaces with norm $\|\cdot\|_{\dot{B}_{p,q}^s}$.
For the sake of conciseness, we also do not distinguish scalar, vector or tensor valued functions.

Let $\mathbb{M}^{d\times d}$ be the space of all $d\times d$ matrix-valued functions,
the space of symmetric traceless $Q$-tensors in $d$-dimension is defined by
\[
S_0^d:=\left\lbrace  Q\in\mathbb{M}^{d\times d}: Q_{\alpha\beta}=Q_{\beta\alpha},\quad {\rm Tr}(Q)=0,\quad \alpha,\beta=1,2,\cdots,d  \right\rbrace.
\]
For $Q\in \mathcal{S}_0^d$, we denote $|Q|$ by
\[
|Q|:=\sqrt{{\rm Tr}(Q^2)}=\sqrt{Q_{\alpha\beta}Q_{\alpha\beta}}.
\]
As usual, $\left\langle\cdot,\cdot\right\rangle$ denotes the inner product for vector-valued functions in $L^2$. Similarly, the inner product of two matrix-valued functions $A$ and $B$ is given by
\[
\left\langle A:B\right\rangle\overset{\Delta}{=}\int_{\R^d}{\rm Tr}(AB)dx.
\]

For a function $h$ in the Schwartz space $\mathcal{S}(\R^3)$, we can define the Fourier transform of $h$ and its inverse on $\R^3$ as
\[
\hat{h}(\xi)=\mathcal{F}[h]=\int_{\R^3} h(x) e^{-ix\cdot\xi}dx,\quad
\check{h}(x)=\mathcal{F}^{-1}[h]= \frac{1}{(2\pi)^3} \int_{\R^3} h(\xi) e^{ix\cdot\xi}d\xi.
\]
The symbol $\mathbb{P}$ stands for the well-known Leray projector, whose Fourier transform is known as
\[
\widehat{\mathbb{P}f}(\xi)=\left( \mathbb{I}_3-\dfrac{\xi\otimes \xi}{|\xi|^2} \right)\hat{f}(\xi).
\]
Moreover, we also use the following commutators of $A$ and $B$ for any functions $f, g$:
\[
\begin{split}
	[A, B]f &= ABf -BAf,\\
	[A, B]_{\rm div}f &= AB\cdot\nabla f -B\cdot\nabla (A f),\\
	[A, (f,g)]_{-} &= A(fg)-(Af)g - [A,g]f.
\end{split}
\]

\section{Local existence and uniqueness}\label{Sec.2}
The purpose of this section is to show the local well-posedness of system \eqref{eq1.1} with initial data $(Q_0,u_0)\in H^{s+1}\times H^s$ for integer $s\geq 2$.
Similar to \cite{active-limit}, we can construct the following approximation system:
\begin{equation}\label{eq-ap}
	\begin{cases}
		\partial_t Q^{n+1}+ u^n\cdot\nabla Q^{n+1} + Q^n\Omega^{n+1}-\Omega^{n+1} Q^n-\lambda |Q^n|D^{n+1}=\Gamma H^{n+1};\\
		\partial_t u^{n+1} + u^n \cdot\nabla u^{n+1}-\mu \Delta u^{n+1} =-\mathbb{P}\nabla\cdot (\nabla Q^n \odot\nabla Q^n)+\mathbb{P}\nabla\cdot (Q^n\Delta Q^{n+1}-\Delta Q^{n+1}Q^n)\\
		\hspace{5.3cm}-\lambda\mathbb{P}\nabla\cdot(|Q^n|H^{n+1})+\kappa \mathbb{P}\nabla\cdot Q^{n+1};\\
		\nabla\cdot u^{n+1}= {\rm Tr}\,Q^{n+1} = 0,\\
		(Q^{n+1},u^{n+1})(t,x)\mid_{t=0}=(Q_0,u_0)(x),
	\end{cases}
\end{equation}
where
\[
D^{n+1}:=\dfrac{\nabla u^{n+1} + (\nabla u^{n+1})^{T}}{2},\quad \Omega^{n+1}:=\dfrac{\nabla u^{n+1} - (\nabla u^{n+1})^{T}}{2}
\]
and
\[
H^{n+1}:= \Delta Q^{n+1}-aQ^{n+1}+b\left[(Q^n)^2-\dfrac{{\rm Tr}((Q^n)^2)}{3}\mathbb{I}_3\right]-cQ^{n+1}{\rm Tr}((Q^n)^2).
\]

Basically, the above iterative system is a linear system, whose solution can be obtained by the standard method. So, the key point is to find some a priori estimate for system \eqref{eq-ap},
which enables us to get the limit of the solution sequence through standard compactness arguments. Then, the limit is expected to be the solution of system \eqref{eq1.1}.
However,
we have to analyze it in a situation where the requirement of initial data
is weaker than that of \cite{active-limit}. Indeed, this improvement can be achieved by deriving the new commutator estimates for some specific nonlinear terms,
which are included in the following Lemma.


\begin{lemma}\label{new-com}
	Let $\psi\in H^s(\R^3)$ and $\phi,\Phi\in H^{s+1}(\R^3)$ with $s\geq 2$. For all $|k|\leq s$, we have
	\begin{align}
		&\left\|\partial^k(\psi \nabla \phi) -\psi \partial^k\nabla \phi \right\|_{L^2}\leq C \|\psi\|_{H^s}\|\phi\|_{H^{s+1}},\label{eq2.3}\\
		&\left\|\partial^k(\phi \nabla\psi ) -\phi \partial^k\nabla \psi \right\|_{L^2}\leq C \|\psi\|_{H^s}\|\phi\|_{H^{s+1}},\label{eq2.4}\\
		&\left\|\partial^k(\phi \Delta \Phi) -\phi \partial^k\Delta \Phi \right\|_{L^2} \leq C \|\phi\|_{H^{s+1}}\|\Phi\|_{H^{s+1}},\label{eq:2.5}
	\end{align}
	where $C$ is a positive constant depending on $k$ and $s$.
\end{lemma}
\begin{proof}
	We will prove these three commutator estimates, respectively.
	Actually, it is sufficient to show the above inequalities hold for all $2\leq |k|\leq s$.
	Hence,
	one can apply the H\"{o}lder inequality and Sobolev embedding to get
	\begin{align*}
		\left\|\partial^k(\psi \nabla \phi) -\psi \partial^k\nabla \phi \right\|_{L^2}&\leq C\left(\sum_{l=1}^{k-1}\left\|  \partial^l \psi \cdot\partial^{k-l}\nabla\phi  \right\|_{L^2} + \left\| \partial^k\psi\cdot \nabla \phi\right\|_{L^2}\right)\\
		& \leq C \left(\sum_{l=1}^{k-1}\left\|  \partial^l \psi \right\|_{L^6}\cdot \left\| \partial^{k-l}\nabla\phi  \right\|_{L^3} + \left\| \partial^k\psi \right\|_{L^2}\cdot \left\|  \nabla \phi\right\|_{L^\infty}\right)\\
		& \leq C \left(\sum_{l=1}^{k-1}\left\|  \partial^{l+1} \psi \right\|_{L^2}\cdot \left\| \partial^{k-l}\nabla\phi  \right\|_{H^1} + \left\| \partial^k\psi \right\|_{L^2}\cdot \left\|  \nabla \phi\right\|_{H^2}\right)\\
		& \leq C \left( \left\|  \psi \right\|_{H^k}\cdot \left\| \nabla\phi  \right\|_{H^k} + \left\| \psi \right\|_{H^k}\cdot \left\|  \nabla \phi\right\|_{H^s}\right),
	\end{align*}
	which is precisely \eqref{eq2.3}.
	
	To prove \eqref{eq2.4}, we only need to make the following modifications:
	\begin{align*}
		\left\|\partial^k(\phi \nabla \psi) -\phi \partial^k\nabla \psi \right\|_{L^2}& \leq C \left(\sum_{l=1}^{k-1}\left\|  \partial^l \phi \right\|_{L^\infty}\cdot \left\| \partial^{k-l}\nabla\psi  \right\|_{L^2}+ \left\|  \partial^k \phi \right\|_{L^6}\cdot \left\| \nabla\psi  \right\|_{L^3} \right)\\
		& \leq C \left(\sum_{l=1}^{k-1}\left\|  \partial^{l} \phi \right\|_{H^2}\cdot \left\| \partial^{k-l}\nabla\psi  \right\|_{L^2} + \left\|  \partial^{k+1} \phi \right\|_{L^2}\cdot \left\| \nabla\psi  \right\|_{H^1} \right).
	\end{align*}
	
	For the last statement \eqref{eq:2.5}, we can proceed as before to get
	\begin{align*}
		\left\|\partial^k(\phi \Delta \Phi) -\phi \partial^k\Delta \Phi \right\|_{L^2}
		& \leq C \left(\sum_{l=1}^{k-1}\left\|  \partial^{l} \phi \right\|_{H^2}\cdot \left\| \partial^{k-l}\Delta \Phi \right\|_{L^2} + \left\|  \partial^{k+1} \phi \right\|_{L^2}\cdot \left\| \Delta \Phi  \right\|_{H^1} \right), 
	\end{align*}
	which completes the proof.
\end{proof}
%
Now, we are in a position to derive some a priori estimates for solutions to the approximation system \eqref{eq-ap}. To do this, we introduce the following energy functionals for any integer $n\geq 0$:
\begin{align*}
	\mathbb{E}_n(t):= &\|u^n(t)\|_{H^s}^2 + \|Q^n(t)\|_{H^{s}}^2 + \|\nabla Q^n(t)\|_{H^{s}}^2,\\
	\mathbb{D}_n(t):= &\mu\|\nabla u^n(t)\|_{H^{s}}^2 + (|a|+1)\Gamma\|Q^n(t)\|_{H^{s+1}}^2 + \Gamma\|\Delta
	Q^n(t)\|_{H^{s}}^2.
\end{align*}
Then, together with Lemma~\ref{new-com}, one can obtain the desired a priori estimate in the following proposition.

\begin{prop}\label{prop2.1}
	Let $s\geq 2$ be an integer and $(Q^{n+1},u^{n+1})(t,x)$ be the solution to the Cauchy problem of system \eqref{eq-ap} with initial data $(Q_0,u_0)\in  H^{s+1} \times H^{s} $ in time interval $[0,T]$. Then, for all $t\in [0,T]$, one has
	\begin{equation}\label{en-loc}
		\dfrac{1}{2}\dfrac{d}{dt}\mathbb{E}_{n+1}(t)+ \mathbb{D}_{n+1}(t)\leq C\mathbb{E}_{n+1}(t) + C\sum_{l=2}^3\mathbb{E}_{n}^{\frac{l}{2}}(t)\mathbb{D}_{n+1}^{\frac{1}{2}}(t) +C\sum_{l=0}^3\mathbb{E}_{n}^{\frac{l}{2}}(t)\mathbb{E}_{n+1}^{\frac{1}{2}}(t)\mathbb{D}_{n+1}^{\frac{1}{2}}(t),
	\end{equation}
	where $C$ is a constant depending only on $s$, $|a|$, $b$, $c$, $\mu$, $\kappa$, $\lambda$ and $\Gamma$.
\end{prop}
\begin{proof}
	\begin{enumerate}[(1).]
		\item The first thing is to show the basic $L^2$-energy estimate, which will indicate the possible structure of the higher order estimate.
		As usual, we take the summation of the first equation in \eqref{eq-ap} multiplied by $Q^{n+1}-\Delta Q^{n+1}$ and the second equation in \eqref{eq-ap} multiplied by $u^{n+1}$, take the trace, and then integrate by parts over $\R^3$ to have
		\begin{align*}\label{eq2.4}
			&\dfrac{1}{2}\dfrac{d}{dt}\left( \|u^{n+1}\|_{L^2}^2 + \|Q^{n+1}\|_{L^2}^2 + \|\nabla Q^{n+1}\|_{L^2}^2 \right) +\mu\|\nabla u^{n+1}\|_{L^2}^2  + (|a|+1)\Gamma\|Q^{n+1}\|_{H^{1}}^2  + \Gamma\|\Delta
			Q^{n+1}\|_{L^2}^2\\
			=&(\eta+1)\Gamma\|Q^{n+1}\|_{L^2}^2 + \eta\Gamma\|\nabla Q^{n+1}\|_{L^2}^2 \underbrace{-\left\langle  u^n\cdot\nabla Q^{n+1}:Q^{n+1}  \right\rangle  -\left\langle  u^n\cdot\nabla u^{n+1},u^{n+1}  \right\rangle}_{I_1} \\
			&+\left\langle  u^n\cdot\nabla Q^{n+1}:\Delta Q^{n+1}  \right\rangle + \left\langle  \Omega^{n+1} Q^n - Q^n\Omega^{n+1}:  Q^{n+1}  \right\rangle  + \left\langle  \nabla Q^n \odot\nabla Q^n:\nabla u^{n+1}  \right\rangle\\
			&\underbrace{-\left\langle  \Omega^{n+1} Q^n - Q^n\Omega^{n+1}:\Delta Q^{n+1}  \right\rangle}_{I_2}  \underbrace{- \lambda \left\langle  |Q^n|D^{n+1}:\Delta Q^{n+1}  \right\rangle +\lambda \left\langle  |Q^n|\Delta Q^{n+1}:\nabla u^{n+1}  \right\rangle}_{I_3} \\
			& + \lambda \left\langle  |Q^n|D^{n+1}: Q^{n+1}  \right\rangle-a\lambda \left\langle  |Q^n| Q^{n+1}:\nabla u^{n+1}   \right\rangle \underbrace{-  \left\langle Q^n\Delta Q^{n+1}-\Delta Q^{n+1}Q^n :\nabla u^{n+1}   \right\rangle}_{I_4} \\
			&- \kappa \left\langle  Q^{n+1}:\nabla u^{n+1}   \right\rangle +\Gamma  \left\langle  \left\lbrace b\left[(Q^n)^2-\dfrac{{\rm Tr}((Q^n)^2)}{3}\mathbb{I}_3\right]-cQ^{n+1}{\rm Tr}((Q^n)^2) \right\rbrace:Q^{n+1} -\Delta Q^{n+1}  \right\rangle\\
			&+ \lambda\left\langle |Q^n| \left\lbrace b\left[(Q^n)^2-\dfrac{{\rm Tr}((Q^n)^2)}{3}\mathbb{I}_3\right]-cQ^{n+1}{\rm Tr}((Q^n)^2) \right\rbrace:\nabla u^{n+1}  \right\rangle,
		\end{align*}
		where the constant $\eta$ is given by
		\[
		\eta=
		\begin{cases}
			0, & \text{if $a\geq 0$,}\\
			2|a|, & \text{if $a< 0$.}\\
		\end{cases}
		\]
		Recall that the cancellation rules stated in Lemma~\ref{lem.cancel} indicate that $I_1=I_2+I_4=I_3=0$. Thus, the above equation can be simplified as
		\begin{align}\label{eq2.7}
			&\dfrac{1}{2}\dfrac{d}{dt}\left( \|u^{n+1}\|_{L^2}^2 + \|Q^{n+1}\|_{L^2}^2 + \|\nabla Q^{n+1}\|_{L^2}^2 \right) +\mu\|\nabla u^{n+1}\|_{L^2}^2  + (|a|+1)\Gamma\|Q^{n+1}\|_{H^{1}}^2  + \Gamma\|\Delta
			Q^{n+1}\|_{L^2}^2\nonumber\\
			=&(\eta+1)\Gamma\|Q^{n+1}\|_{L^2}^2 + \eta\Gamma\|\nabla Q^{n+1}\|_{L^2}^2  +\left\langle  u^n\cdot\nabla Q^{n+1}:\Delta Q^{n+1}  \right\rangle + \lambda \left\langle  |Q^n|D^{n+1}: Q^{n+1}  \right\rangle\nonumber\\
			& + \left\langle  \nabla Q^n \odot\nabla Q^n:\nabla u^{n+1}  \right\rangle -a\lambda \left\langle  |Q^n| Q^{n+1}:\nabla u^{n+1}   \right\rangle + \left\langle  \Omega^{n+1} Q^n - Q^n\Omega^{n+1}:  Q^{n+1}  \right\rangle \\
			& - \kappa \left\langle  Q^{n+1}:\nabla u^{n+1}   \right\rangle  + \lambda\left\langle |Q^n| \left\lbrace b\left[(Q^n)^2-\dfrac{{\rm Tr}((Q^n)^2)}{3}\mathbb{I}_3\right]-cQ^{n+1}{\rm Tr}((Q^n)^2) \right\rbrace:\nabla u^{n+1}  \right\rangle\nonumber \\
			& +\Gamma  \left\langle  \left\lbrace b\left[(Q^n)^2-\dfrac{{\rm Tr}((Q^n)^2)}{3}\mathbb{I}_3\right]-cQ^{n+1}{\rm Tr}((Q^n)^2) \right\rbrace:Q^{n+1} -\Delta Q^{n+1}  \right\rangle\nonumber\\
			\overset{\bigtriangleup}{=}& (\eta+1)\Gamma\|Q^{n+1}\|_{L^2}^2 + \eta\Gamma\|\nabla Q^{n+1}\|_{L^2}^2 + \sum_{i=1}^{8} J_i.\nonumber
		\end{align}
		Next, we estimate the right-hand side of \eqref{eq2.7} term by term. By the use of H\"{o}lder inequality, Gagliardo-Nirenberg inequality and Sobolev embeddings, we have
		\begin{equation}
			\begin{split}
				|J_1|+|J_2| +|J_4|+ |J_5|&\lesssim \|u^n\|_{L^6}\|\nabla Q^{n+1}\|_{L^3}\|\Delta Q^{n+1}\|_{L^2} + \| \nabla u^{n+1}\|_{L^2}\| Q^{n}\|_{L^3}\| Q^{n+1}\|_{L^6}\\
				&\lesssim \|\nabla u^n\|_{L^2}\|\nabla Q^{n+1}\|_{H^1}\|\Delta Q^{n+1}\|_{L^2} + \|\nabla u^{n+1}\|_{L^2}\| Q^{n}\|_{H^1}\| \nabla Q^{n+1}\|_{L^2}\\
				&\leq C\mathbb{E}_n^{\frac{1}{2}}(t) \mathbb{E}_{n+1}^{\frac{1}{2}}(t)\mathbb{D}_{n+1}^{\frac{1}{2}}(t)\\
				|J_3|+ |J_6|&\lesssim \|\nabla Q^{n}\|_{L^4}^2\|\nabla u^{n+1}\|_{L^2} + \| Q^{n+1}\|_{L^2}\|\nabla u^{n+1}\|_{L^2}\\
				&\lesssim \|\nabla Q^{n}\|_{H^1}^2\|\nabla u^{n+1}\|_{L^2} + \mathbb{E}_{n+1}^{\frac{1}{2}}(t)\mathbb{D}_{n+1}^{\frac{1}{2}}(t)\\
				&\leq C\left( \mathbb{E}_{n}(t) + \mathbb{E}_{n+1}^{\frac{1}{2}}(t)\right) \mathbb{D}_{n+1}^{\frac{1}{2}}(t)\\
				|J_{7}|&\leq C \|\nabla u^{n+1}\|_{L^2}\|Q^n\|_{L^\infty} \left(\|Q^n\|_{L^4}^2+ \|Q^n\|_{L^6}^2 \|Q^{n+1}\|_{L^6} \right)\\
				&\leq C\|Q^n\|_{H^2}\left( \|Q^n\|_{H^1}^2+ \|\nabla Q^n\|_{L^2}^2\|\nabla Q^{n+1}\|_{L^2} \right)\mathbb{D}_{n+1}^{\frac{1}{2}}(t)\\
				&\leq C\left( 1 +\mathbb{E}_{n+1}^{\frac{1}{2}}(t)\right)\mathbb{E}_{n}^{\frac{3}{2}}(t)\mathbb{D}_{n+1}^{\frac{1}{2}}(t)\\
				|J_8|&\leq C  \left( \|Q^n\|_{L^4}^2+ \|Q^n\|_{L^6}^2 \|Q^{n+1}\|_{L^6}  \right)\cdot \left( \|Q^{n+1}\|_{L^2}+ \|\Delta Q^{n+1}\|_{L^2}  \right) \\
				&\leq C\left( \|Q^n\|_{H^1}^2+ \|\nabla Q^n\|_{L^2}^2\|\nabla Q^{n+1}\|_{L^2}  \right)\mathbb{D}_{n+1}^{\frac{1}{2}}(t)\\
				&\leq C\left( 1 +\mathbb{E}_{n+1}^{\frac{1}{2}}(t)\right)\mathbb{E}_{n}(t)\mathbb{D}_{n+1}^{\frac{1}{2}}(t)
			\end{split}
		\end{equation}
		Substituting these results into \eqref{eq2.7}, we infer that
		\begin{equation}\label{eq:0-th}
			\begin{split}
				\dfrac{1}{2}\dfrac{d}{dt}\Big( \|u^{n+1}\|_{L^2}^2 +& \|Q^{n+1}\|_{L^2}^2 + \|\nabla Q^{n+1}\|_{L^2}^2 \Big)\\
				&+\mu\|\nabla u^{n+1}\|_{L^2}^2  + (|a|+1)\Gamma\|Q^{n+1}(t)\|_{H^{1}}^2  + \Gamma\|\Delta
				Q^{n+1}\|_{L^2}^2\\
				\leq& C\mathbb{E}_{n+1}(t) + C\sum_{l=2}^3\mathbb{E}_{n}^{\frac{l}{2}}(t)\mathbb{D}_{n+1}^{\frac{1}{2}}(t) +C\sum_{l=0}^3\mathbb{E}_{n}^{\frac{l}{2}}(t)\mathbb{E}_{n+1}^{\frac{1}{2}}(t)\mathbb{D}_{n+1}^{\frac{1}{2}}(t).
			\end{split}
		\end{equation}
		\item In this step, we utilize the commutator estimates established in Lemma~\ref{new-com} to show the higher order estimate of system \eqref{eq-ap}.
		For any $1\leq |k|\leq s$,
		applying $\partial^k$ to the \eqref{eq-ap}$_2$ and multiplying the resulting equation by $\partial^ku^{n+1}$, meanwhile,
		applying $\partial^k$ to the \eqref{eq-ap}$_1$, and multiplying the resulting equation by $\partial^kQ^{n+1}-\partial^k\Delta Q^{n+1}$,
		taking the trace, and then summing up all the results to get
		\begin{align*}
			\dfrac{1}{2}\dfrac{d}{dt}\Big( &\|\partial^ku^{n+1}\|_{L^2}^2 + \|\partial^kQ^{n+1}\|_{L^2}^2 + \|\partial^k\nabla Q^{n+1}\|_{L^2}^2 \Big)\\
			&+\mu\|\partial^{k}\nabla u^{n+1}\|_{L^2}^2 + (|a|+1)\Gamma \|\partial^{k} Q^{n+1}\|_{H^1}^2+ \Gamma\|\partial^k\Delta
			Q^{n+1}\|_{L^2}^2\\
			=& (\eta+1)\Gamma\|\partial^k Q^{n+1}\|_{L^2}^2 + \eta\Gamma\|\partial^k\nabla Q^{n+1}\|_{L^2}^2 +\left\langle u^n\cdot\nabla\partial^k Q^{n+1}:\partial^k\Delta
			Q^{n+1}  \right\rangle   \\
			& + \left\langle [\partial^k,u^n]_{\rm div} Q^{n+1}:\partial^k\Delta
			Q^{n+1}-\partial^kQ^{n+1}  \right\rangle +\left\langle \partial^k (u^n\otimes u^{n+1}):\partial^k\nabla u^{n+1} \right\rangle \\
			&+\left\langle \partial^k\Omega^{n+1} Q^n -  Q^n\partial^k\Omega^{n+1}:\partial^k Q^{n+1} \right\rangle +\left\langle  [\partial^k,( \Omega^{n+1}, Q^n)]_{-}:\partial^k Q^{n+1}-\partial^k\Delta
			Q^{n+1} \right\rangle\\
			& +  \left\langle  \partial^k\left( \nabla Q^n \odot\nabla Q^n\right):\partial^{k}\nabla u^{n+1}  \right\rangle  +\left\langle  [\partial^k,( \Delta Q^{n+1}, Q^n)]_{-}:\partial^{k}\nabla u^{n+1} \right\rangle\\
			& - \kappa \left\langle  \partial^kQ^{n+1}:\partial^k\nabla u^{n+1}   \right\rangle  + \lambda \left\langle \partial^k\left(  |Q^n|D^{n+1}\right):\partial^k Q^{n+1}  \right\rangle  - \lambda \left\langle [\partial^k, |Q^n|]D^{n+1}:\partial^k\Delta Q^{n+1}  \right\rangle \\
			&+\lambda \left\langle [\partial^k, |Q^n|]\Delta Q^{n+1} :\partial^k\nabla u^{n+1}  \right\rangle -a\lambda \left\langle \partial^k\left(  |Q^n|Q^{n+1}\right):\partial^k \nabla u^{n+1}  \right\rangle \\
			& +\Gamma  \left\langle  \partial^k\left\lbrace b\left[(Q^n)^2-\dfrac{{\rm Tr}((Q^n)^2)}{3}\mathbb{I}_3\right]-cQ^{n+1}{\rm Tr}((Q^n)^2) \right\rbrace:\partial^kQ^{n+1} -\partial^k\Delta Q^{n+1}  \right\rangle \\
			&+ \lambda\left\langle \partial^k\left(|Q^n| \left\lbrace b\left[(Q^n)^2-\dfrac{{\rm Tr}((Q^n)^2)}{3}\mathbb{I}_3\right]-cQ^{n+1}{\rm Tr}((Q^n)^2) \right\rbrace\right):\partial^k\nabla u^{n+1}  \right\rangle\\
			\overset{\bigtriangleup}{=}& (\eta+1)\Gamma\|\partial^k Q^{n+1}\|_{L^2}^2 + \eta\Gamma\|\partial^k\nabla Q^{n+1}\|_{L^2}^2 + \sum_{j=1}^{14} K_j,
		\end{align*}
		where we have used the following cancellation rules:
		\[
		\begin{split}
			\left\langle Q^n \partial^k\Omega^{n+1} - \partial^k\Omega^{n+1} Q^n:\partial^k\Delta Q^{n+1}\right\rangle -\left\langle Q^n\partial^k \Delta Q^{n+1} - \partial^k \Delta Q^{n+1} Q^n: \partial^k\nabla u^{n+1} \right\rangle =&0,\\
			\left\langle | Q^n| \partial^k\Delta Q^{n+1}: \partial^k\nabla u^{n+1}\right\rangle-\left\langle|Q^n|\partial^k D^{n+1}: \partial^k\Delta Q^{n+1}\right\rangle=&0.
		\end{split}
		\]
		Similarly, we will estimate $K_j$ term by term. Indeed, one can get
		\begin{align*}
			|K_1| + |K_4| \lesssim &  \|u^{n}\|_{L^\infty} \| \nabla\partial^{k}  Q^{n+1}\|_{L^2} \|\partial^{k} Q^{n+1}\|_{H^2} +  \|Q^{n}\|_{L^\infty} \| \nabla\partial^{k}  u^{n+1}\|_{L^2} \|\partial^{k}   Q^{n+1}\|_{L^2} \\
			\lesssim & \mathbb{E}_n^{\frac{1}{2}}(t) \mathbb{E}_{n+1}^{\frac{1}{2}}(t)\mathbb{D}_{n+1}^{\frac{1}{2}}(t),\\
			|K_2|+ |K_3|\lesssim&   \| u^{n}\|_{H^s}\| Q^{n+1}\|_{H^{s+1}}\| \partial^k Q^{n+1}\|_{H^2} + \| u^{n}\|_{L^\infty} \|\partial^{k} u^{n+1}\|_{L^2}\|\partial^k\nabla u^{n+1}\|_{L^2}   \\
			&  + \| u^{n+1}\|_{L^\infty} \|\partial^{k} u^{n}\|_{L^2}\|\partial^k\nabla u^{n+1}\|_{L^2} \\
			\lesssim&  \mathbb{E}_n^{\frac{1}{2}}(t) \mathbb{E}_{n+1}^{\frac{1}{2}}(t)\mathbb{D}_{n+1}^{\frac{1}{2}}(t),\\
			|K_5|+ |K_7|+|K_{10}|+ |K_{11}| \lesssim&  \| Q^{n}\|_{H^{s+1}}\| u^{n+1}\|_{H^s}\| \partial^k Q^{n+1}\|_{H^2} + \| Q^{n}\|_{H^{s+1}}\| Q^{n+1}\|_{H^{s+1}}\|\partial^k\nabla u^{n+1}\|_{L^2}  \\
			\lesssim&  \mathbb{E}_n^{\frac{1}{2}}(t) \mathbb{E}_{n+1}^{\frac{1}{2}}(t)\mathbb{D}_{n+1}^{\frac{1}{2}}(t),\\
			|K_6|+ |K_8|\lesssim &   \|\partial^{k}\nabla Q^n \|_{L^2} \|\nabla Q^n \|_{L^\infty} \|\partial^{k}\nabla u^{n+1}  \|_{L^2} +\|  \partial^kQ^{n+1}\|_{L^2}\|\partial^k\nabla u^{n+1}  \|_{L^2}\\
			\lesssim & \|  \nabla Q^n \|_{H^{s}} \|\nabla Q^n \|_{H^{2}} \|\nabla u^{n+1}  \|_{H^s} + \| Q^{n+1}\|_{H^s}\| \nabla u^{n+1}  \|_{H^s}\\
			\lesssim&  \mathbb{E}_n(t)\mathbb{D}_{n+1}^{\frac{1}{2}}(t)+  \mathbb{E}_{n+1}^{\frac{1}{2}}(t)\mathbb{D}_{n+1}^{\frac{1}{2}}(t),\\
			|K_9| \lesssim  & \left(  \| Q^n \|_{L^\infty}  \|\partial^{k}\nabla u^{n+1} \|_{L^2} +  \| \nabla u^{n+1} \|_{L^\infty}  \|\partial^{k}  |Q^n| \|_{L^2}  \right) \|\partial^k Q^{n+1} \|_{L^2} \\
			\lesssim &  \left(  \| Q^n \|_{H^{2}}  \| \nabla u^{n+1} \|_{H^s} +  \| \nabla u^{n+1} \|_{H^2}  \| Q^n \|_{H^s}  \right)  \| Q^{n+1} \|_{H^s} \\
			\lesssim & \mathbb{E}_n^{\frac{1}{2}}(t) \mathbb{E}_{n+1}^{\frac{1}{2}}(t)\mathbb{D}_{n+1}^{\frac{1}{2}}(t),\\
			|K_{13}|  \lesssim& \mathbb{D}_{n+1}^{\frac{1}{2}}(t)\cdot \| Q^n\|_{L^\infty}\left( \left(1+  \| Q^{n+1}\|_{L^\infty}\right)\|\partial^k Q^n \|_{L^2} + \| Q^n\|_{L^\infty}\|\partial^k Q^{n+1} \|_{L^2} \right)\\
			\lesssim & \left( \left(1+   \| Q^{n+1}\|_{H^2}\right) \| Q^n\|_{H^s} +\| Q^{n}\|_{H^2} \| Q^{n+1}\|_{H^s}\right)\mathbb{E}_n^{\frac{1}{2}}(t) \mathbb{D}_{n+1}^{\frac{1}{2}}(t) \\
			\lesssim& \left( 1+ \mathbb{E}_{n+1}^{\frac{1}{2}}(t)\right)\mathbb{E}_n(t) \mathbb{D}_{n+1}^{\frac{1}{2}}(t),\\
			|K_{12}| + |K_{14}| \lesssim&\|\partial^k\nabla u^{n+1}  \|_{L^2}\Big(\|Q^n\|_{L^\infty}\|  \partial^k Q^{n+1}\|_{L^2} + \|Q^{n+1}\|_{L^\infty}\|  \partial^k |Q^n|\|_{L^2}\\
			&+ \|Q^n\|_{L^\infty}^2\| \left( \partial^k Q^n\|_{L^2} +  \partial^k |Q^n|\|_{L^2}\right)  + \|Q^{n}\|_{L^\infty}^2\|Q^{n+1}\|_{L^\infty}\|  \partial^k |Q^n|\|_{L^2}\\
			& + \|Q^n\|_{L^\infty}^2\|Q^{n+1}\|_{L^\infty}\|  \partial^k Q^n\|_{L^2}  + \|Q^{n}\|_{L^\infty}^3 \|  \partial^k Q^{n+1}\|_{L^2} \Big)\\
			\lesssim & \mathbb{D}_{n+1}^{\frac{1}{2}}(t)\cdot \left( \| Q^n\|_{H^s} \| Q^{n+1}\|_{H^s}+  \| Q^n\|_{H^s}^3 + \| Q^n\|_{H^s}^3 \|Q^{n+1}\|_{H^s}\right) \\
			\lesssim& \mathbb{E}_n^{\frac{1}{2}}(t) \mathbb{E}_{n+1}^{\frac{1}{2}}(t)\mathbb{D}_{n+1}^{\frac{1}{2}}(t) + \left( 1+\mathbb{E}_{n+1}^{\frac{1}{2}}(t)\right) \mathbb{E}_{n}^{\frac{3}{2}}(t)\mathbb{D}_{n+1}^{\frac{1}{2}}(t),
		\end{align*}
		where we have also applied Lemma~\ref{new-com}, Lemma~\ref{lem-a1}, Lemma~\ref{lem-A6} and $\|f\|_{L^\infty(\R^3)}\lesssim \|f\|_{H^{s}(\R^3)}$ ($s\geq 2$).
		In conclusion, we get
		\begin{equation}\label{eq:k-th}
			\begin{split}
				\dfrac{1}{2}\dfrac{d}{dt}\Big( &\|\partial^ku^{n+1}\|_{L^2}^2 + \|\partial^kQ^{n+1}\|_{L^2}^2 + \|\partial^k\nabla Q^{n+1}\|_{L^2}^2 \Big)\\
				&+\mu\|\partial^{k}\nabla u^{n+1}\|_{L^2}^2 + (|a|+1)\Gamma \|\partial^{k} Q^{n+1}\|_{H^1}^2+ \Gamma\|\partial^k\Delta
				Q^{n+1}\|_{L^2}^2\\
				\lesssim&  \mathbb{E}_{n+1}(t) + \sum_{l=2}^3\mathbb{E}_{n}^{\frac{l}{2}}(t)\mathbb{D}_{n+1}^{\frac{1}{2}}(t) +\sum_{l=0}^3\mathbb{E}_{n}^{\frac{l}{2}}(t)\mathbb{E}_{n+1}^{\frac{1}{2}}(t)\mathbb{D}_{n+1}^{\frac{1}{2}}(t).
			\end{split}
		\end{equation}
		\item Finally, summing up \eqref{eq:0-th} and \eqref{eq:k-th} for all $1\leq |k|\leq s$, we deduce that
		\[
		\dfrac{1}{2}\dfrac{d}{dt}\mathbb{E}_{n+1}(t)+ \mathbb{D}_{n+1}(t)\leq C\mathbb{E}_{n+1}(t) + C\sum_{l=2}^3\mathbb{E}_{n}^{\frac{l}{2}}(t)\mathbb{D}_{n+1}^{\frac{1}{2}}(t) +C\sum_{l=0}^3\mathbb{E}_{n}^{\frac{l}{2}}(t)\mathbb{E}_{n+1}^{\frac{1}{2}}(t)\mathbb{D}_{n+1}^{\frac{1}{2}}(t).
		\]
	\end{enumerate}
	This completes the proof.
\end{proof}
\begin{remark}\label{re2.2}
	It should be noted that we applied new commutator estimates ‌instead‌ of the Moser-type calculus inequalities, while performing the high order energy estimates for $K_2$, $K_5$, $K_7$, $K_{10}$ and $K_{11}$. This differs from our global analysis below and allows us to address the above a priori estimate for \eqref{eq-ap} under a weaker assumption on initial data than that in \cite{active-limit}.
\end{remark}

Thanks to the a priori estimate \eqref{en-loc},
we are able to prove Theorem~\ref{thm1.1} concerning
the local well-posedness of system \eqref{eq1.1} for arbitrary initial data $(Q_0,u_0)\in H^{s+1} \times H^{s}$ with $s\geq 2$.

\noindent\textbf{Proof of Theorem\ref{thm1.1}.}

As usual, the proof consists of three steps:
\begin{enumerate}[(1).]
	\item In the first step, we will show the local-in-time existence of solutions to system \eqref{eq1.1}.
	In fact, it is easy to check that
	\begin{align*}
		\sum_{l=0}^3\mathbb{E}_n^{\frac{l}{2}}(t) \mathbb{E}_{n+1}^{\frac{1}{2}}(t)\mathbb{D}_{n+1}^{\frac{1}{2}}(t) \leq& 12C \sum_{l=0}^3\mathbb{E}_n^{l}(t)\mathbb{E}_{n+1}(t)+\dfrac{1}{3C} \mathbb{D}_{n+1}(t),\\
		\sum_{l=2}^3\mathbb{E}_n^{\frac{l}{2}}(t) \mathbb{D}_{n+1}^{\frac{1}{2}}(t)\leq& 12C \sum_{l=2}^3\mathbb{E}_n^{l}(t)+\dfrac{1}{6C} \mathbb{D}_{n+1}(t),
	\end{align*}
	where $C$ is chosen to be the same as in Proposition~\ref{prop2.1}.
	Then, substituting these results into \eqref{en-loc}, one can get
	\begin{equation}\label{eq2.11}
		\dfrac{d}{dt}\mathbb{E}_{n+1}(t)+ \mathbb{D}_{n+1}(t)
		\leq \widetilde{C}\sum_{l=0}^3\mathbb{E}_n^{l}(t)\mathbb{E}_{n+1}(t) + \widetilde{C} \sum_{l=2}^3\mathbb{E}_n^{l}(t),
	\end{equation}
	where $\widetilde{C}:= 2C(12C+1)$.
	Next, we set
	\[
	(Q^0,u^0)(t,x)=(Q_0,u_0)(x), \quad \forall t\in [0,T],
	\]
	and prove the following statement by induction
	\begin{equation}\label{eq2.12}
		\sup_{t\in [0,T]} \mathbb{E}_{n}(t)\leq 4E_0,\quad \forall n\geq 0.
	\end{equation}
	Observe that for $n=0$, the claim follows from the definition directly. Hence, by assuming that \eqref{eq2.12} holds for $n$, we aim to show that it is again true for $n+1$.
	%
	Indeed, one can apply the Gr\"{o}nwall's inequality to get
	\[
		\begin{split}
			\mathbb{E}_{n+1}(t)+ \int_0^t\mathbb{D}_{n+1}(\tau)d\tau &\leq \left( \mathbb{E}_{n+1}(0) + \widetilde{C} \sum_{l=2}^3\int_0^t\mathbb{E}_n^{l}(\tau) d\tau \right)
			\exp\left\lbrace \widetilde{C}\sum_{l=0}^3\int_0^t\mathbb{E}_n^{l}(\tau) d\tau  \right\rbrace\\
			&\leq \left( E_0 + \sum_{l=2}^3 4^{l}\widetilde{C}E_0^{l}t  \right)
			\exp\left\lbrace \sum_{l=0}^3 4^{l}\widetilde{C}E_0^{l}t \right\rbrace
		\end{split}
	\]
	for all $0\leq t\leq T$.
	From this, it can be seen that it is sufficient to choose the time $t$ satisfying
	\[
		\sum_{l=2}^3 4^{l}\widetilde{C}E_0^{l}t \leq E_0 \quad {\rm and}\quad
		\exp\left\lbrace \sum_{l=0}^3 4^{l}\widetilde{C}E_0^{l}t \right\rbrace  \leq 2,
	\]
	such that
	\[
	\mathbb{E}_{n+1}(t)+ \int_0^t\mathbb{D}_{n+1}(t)dt \leq 4E_0.
	\]
	This finishes the induction step,
	where the time $T$ can be specified by
	\begin{equation}\label{eq:Tmax}
		0<T\leq T^\ast:=\dfrac{\min\{E_0, \ln 2\}}{\sum_{l=0}^3 4^{l}\widetilde{C}E_0^{l}}.
	\end{equation}
    Now, with the above induction argument,
	one can conclude from the standard compactness argument that there exists a local-in-time solution $(Q,u)$ to system \eqref{eq1.1} with initial data $(Q_0,u_0)\in H^{s+1}\times H^s$ for $s\geq 2$. More precisely,
	\[
	Q\in L^\infty\left( 0,T;H^{s+1} \right)\cap L^2\left( 0,T;H^{s+2} \right),\quad u\in L^\infty\left( 0,T;H^s \right)\cap L^2\left( 0,T;H^{s+1} \right)
	\]
	and
	\[
	\sup_{t\in [0,T]} \|u(t)\|_{H^s}^2 + \|Q(t)\|_{H^{s+1}}^2 
	+\int_0^T \mu\|\nabla u(t)\|_{H^{s}}^2 + (|a|+1)\Gamma\|Q(t)\|_{H^{s+1}}^2 + \Gamma\|\Delta Q(t)\|_{H^{s}}^2 dt \leq C_0,
	\]
	where $C_0$ is a positive constant depending only on $E_0$, $s$, $|a|$, $b$, $c$, $\mu$, $\kappa$, $\lambda$ and $\Gamma$.
	\item As the second step,
	we take the difference of any two solutions $(Q_1,u_1)$ and $(Q_2,u_2)$ belonging to the class specified above, and prove the uniqueness by using the energy method. Namely, we shall define
	\[
	\left(\delta Q,\delta u \right):=\left(Q_1-Q_2, u_1-u_2\right),
	\]
	and study the following system:
	\begin{equation}
		\begin{cases}
			\partial_t \delta Q + \delta u\cdot\nabla \delta Q-\delta \Omega\delta Q +\delta Q\delta \Omega + \delta u\cdot\nabla Q_2 + u_2\cdot\nabla \delta Q\\
			\hspace{1.6cm}+ Q_2\delta\Omega + \delta Q\Omega_2 - \delta\Omega Q_2 - \Omega_2\delta Q\\
			\hspace{1.2cm}= \lambda|Q_1|\delta D +\lambda\left( |Q_1|-|Q_2| \right)D_2 + b\Gamma\left[ Q_1\delta Q +\delta QQ_2 - \dfrac{{\rm Tr}\left( Q_1\delta Q + \delta Q Q_2 \right)}{3}\mathbb{I}_3 \right]\\
			\hspace{1.6cm}+ \Gamma\left( \Delta\delta Q  - a\delta Q- c\left( \delta Q{\rm Tr}(Q_1^2) + Q_2 {\rm Tr}\left( Q_1\delta Q + \delta Q Q_2 \right) \right) \right),\\
			\partial_t \delta u + \delta u\cdot\nabla \delta u -\mu \Delta\delta u\\
			\hspace{1.2cm}= -\mathbb{P}\nabla\cdot\left( \nabla\delta Q\odot \nabla\delta Q - \delta Q\Delta \delta Q + \Delta \delta Q\delta Q \right)- \mathbb{P} \left( u_2\cdot\nabla \delta u +\delta u\cdot\nabla u_2 \right)\\
			\hspace{1.6cm}+ \mathbb{P}\nabla\cdot\left(  \delta Q\Delta Q_2 + Q_2\Delta\delta Q -\Delta \delta QQ_2 -\Delta Q_2\delta Q \right) + \kappa\mathbb{P}\nabla\cdot\delta Q\\
			\hspace{1.6cm}-\mathbb{P}\nabla\cdot\left( \nabla\delta Q\odot Q_2 + Q_2\odot\nabla\delta Q  \right) -\lambda\mathbb{P}\nabla\cdot \left( |Q_1| \left(\Delta\delta Q  - a\delta Q\right) \right)\\
			\hspace{1.6cm}- \lambda\mathbb{P}\nabla\cdot \left( \left(|Q_1|-|Q_2|\right) \left(Q_2  - aQ_2\right) \right) + c\lambda\mathbb{P}\nabla\cdot \left( \left(|Q_1|-|Q_2|\right)Q_2{\rm Tr}(Q_2^2) \right)\\
			\hspace{1.6cm} + c\lambda\mathbb{P}\nabla\cdot \left( |Q_1|\delta Q{\rm Tr}(Q_1^2) \right) + c\lambda\mathbb{P}\nabla\cdot \left( |Q_1|Q_2{\rm Tr}\left(Q_1\delta Q + \delta QQ_2\right) \right)\\
			\hspace{1.6cm} -b\lambda\mathbb{P}\nabla\cdot \left[ |Q_1|Q_1\delta Q + \left( |Q_1|-|Q_2| \right)Q_1Q_2 + |Q_2|\delta QQ_2 \right]\\
			\hspace{1.6cm} +b\lambda\mathbb{P}\nabla\cdot \left[ \dfrac{|Q_1|{\rm Tr}\left(Q_1\delta Q\right) + \left( |Q_1|-|Q_2| \right){\rm Tr}\left(Q_1Q_2\right) + |Q_2|{\rm Tr}\left(\delta QQ_2\right)}{3}\mathbb{I}_3 \right].
		\end{cases}
	\end{equation}
	In the same way as the proof of Proposition~\ref{prop2.1}, one can obtain that
	\[
	\begin{split}
		\dfrac{1}{2}\dfrac{d}{dt}\Big(\|\delta u\|_{H^s}^2 + \|\delta Q\|_{H^{s+1}}^2 \Big)
		\lesssim&\Big[ 1+ \|u_1\|_{H^s}^2+ \|u_2\|_{H^s}^2 + \|\nabla u_2\|_{H^s}^2 +\| Q_1\|_{H^{s+1}}^2 +\| Q_2\|_{H^{s+1}}^2 \\
		&+ \|Q_1\|_{H^s}^4 + \|Q_2\|_{H^s}^4 + \|Q_1\|_{H^s}^6 + \|Q_2\|_{H^s}^6  \Big]\cdot\left(\|\delta u\|_{H^s}^2 + \|\delta Q\|_{H^{s+1}}^2\right).
	\end{split}
	\]
	For simplicity, we omit the details here and leave it to the interested readers.
	Then, the desired uniqueness result follows from the standard Gr\"{o}nwall's inequality.
	\item In the last step,
	we will claim that ${\rm Tr}Q=0$ by a similar argument as in \cite{weak-ALC}. Hence, by 
	taking the trace on both sides of the first equation in system \eqref{eq1.1}, we aim to study the following Cauchy problem:
	\[
	\begin{cases}
		\partial_t {\rm Tr}Q+ u \cdot\nabla{\rm Tr} Q  =\Gamma \left(\Delta {\rm Tr}Q -a{\rm Tr}Q -c{\rm Tr}Q {\rm Tr}(Q^2)\right),\\
		{\rm Tr}Q|_{t=0}= {\rm Tr}Q_0=0.
	\end{cases}
	\]
	Then, a standard energy argument implies that
	\begin{equation}
		\dfrac{d}{dt}\|{\rm Tr}Q\|_{L^2}^2 + \Gamma\|\nabla {\rm Tr}Q\|_{L^2}^2 \leq C\|{\rm Tr}Q\|_{L^2}^2,
	\end{equation}
	where $C$ is a positive constant depending on $E_0$, $|a|$, $c$ and $\Gamma$. Therefore, applying the Gr\"{o}nwall's inequality again, one can see that $Q$ is traceless and this completes the proof of Theorem~\ref{thm1.1}.$\hfill\Box$
\end{enumerate}

\section{Global well-posedness}\label{sec-3}
In this section, we are going to show that the local solution previously constructed is indeed global when the initial data is sufficiently small.
To do this, one may prove it by using the standard continuation argument along with suitable a priori estimate of system \eqref{eq1.1}.
Hence, we define the following kinetic energy functional
\[
\mathcal{E}(t):= \|u(t)\|_{H^s}^2 + M\|Q(t)\|_{H^{s}}^2 +\|\nabla Q(t)\|_{H^{s}}^2
\]
and dissipation energy functional
\[
\mathcal{D}(t):= \mu\|\nabla u(t)\|_{H^{s}}^2 + aM\Gamma\|  Q(t)\|_{H^{s}}^2 + (a+M)\Gamma\|\nabla Q(t)\|_{H^{s}}^2 + \Gamma\|\Delta
Q(t)\|_{H^{s}}^2,																	
\]
where $M$ is a sufficiently large positive constant, to be determined later in the proof.

Let us explain why we use these new energy functionals here; see also \cite{active-limit}.
In fact,
by observing the structure of system \eqref{eq1.1}, the difficulty is how to deal with the linear active term $\kappa\nabla\cdot Q$ in performing the a priori estimate.
Heuristically, we will multiply $\kappa\partial^k\nabla\cdot Q$ by $\partial^k u$ and integrate by parts to get the resulting quadratic term $\kappa\left\langle \partial^kQ:\partial^k\nabla u\right\rangle$, which can be estimated as follows
\[
\left| \kappa\left\langle \partial^k Q:\partial^k\nabla u\right\rangle \right|\leq |\kappa| \|Q\|_{H^s}\|\nabla u\|_{H^s}.
\]
Thus, we expect that the right-hand side of the above inequality can be absorbed by the dissipative energy terms $aM\Gamma\|Q\|_{H^s}^2$ and $\mu\|\nabla u\|_{H^s}^2$. However, when $a\leq 0$,
one can see that the dissipation energy functional that we used is not well-defined.
Hence, to demonstrate the global regularity for system \eqref{eq1.1} by employing the energy argument,
we restrict ourselves to the case of $a>0$.
In particular, we will apply the Cauchy-Schwarz inequality to get
\begin{equation}\label{linear-term}
|\kappa| \|Q\|_{H^s}\|\nabla u\|_{H^s} \leq \dfrac{\mu}{2}\|\nabla u\|_{H^s}^2 + \dfrac{2\kappa^2}{\mu} \|Q\|_{H^s}^2,
\end{equation}
and these two terms on the right-hand side of \eqref{linear-term} can be absorbed by the dissipative energy functional $\mathcal{D}(t)$ with a well-chosen constant $M$.

Based on this idea, we are now ready to present the a priori estimate for system \eqref{eq1.1} in the following.
\begin{prop}\label{prop3.1}
	Let $s\geq 2$ be an integer and $a>0$. Then the local solution $(Q,u)$ established in Theorem~\ref{thm1.1} satisfies
	\begin{equation}\label{es-global}
			 \dfrac{d}{dt}\mathcal{E}(t)+ \mathcal{D}(t)
			 \leq C \sum_{l=1}^3\mathcal{E}^{\frac{l}{2}}(t) \mathcal{D}(t),\quad \forall t\in [0,T],
	\end{equation}
	where $C$ is a positive constant depending only on $s$, $a$, $b$, $c$, $\mu$, $\kappa$, $\lambda$ and $\Gamma$.
\end{prop}
\begin{proof}
Similarly to Proposition~\ref{prop2.1}, we divide the proof into two steps.

\noindent\text{\em Step~1: The basic energy estimate.}
We take the summation of the first equation in \eqref{eq1.1} multiplied by $MQ-\Delta Q$ and the second equation in \eqref{eq1.1} multiplied by $u$, take the trace, and integrate by parts over $\R^3$ to get
\begin{align*}\label{eq3.2}
	&\dfrac{1}{2}\dfrac{d}{dt}\left( \|u\|_{L^2}^2 + M\|Q\|_{L^2}^2 + \|\nabla Q\|_{L^2}^2 \right) +\mu\|\nabla u\|_{L^2}^2  + aM\Gamma\| Q\|_{L^2}^2 + (a+M)\Gamma\|\nabla Q\|_{L^2}^2 + \Gamma\|\Delta
	Q\|_{L^2}^2\\
	=& \lambda M \left\langle  |Q|D: Q  \right\rangle   -a\lambda \left\langle  |Q| Q:\nabla u   \right\rangle - \kappa \left\langle  Q:\nabla u  \right\rangle + \lambda\left\langle |Q| \left\lbrace b\left[Q^2-\dfrac{{\rm Tr}(Q^2)}{3}\mathbb{I}_3\right]-cQ{\rm Tr}(Q^2) \right\rbrace:\nabla u  \right\rangle\\
	& +\Gamma  \left\langle  \left\lbrace b\left[Q^2-\dfrac{{\rm Tr}(Q^2)}{3}\mathbb{I}_3\right]-cQ{\rm Tr}(Q^2) \right\rbrace:MQ -\Delta Q  \right\rangle\\
	\leq & C\|Q\|_{L^\infty} \|Q\|_{L^2}\|\nabla u\|_{L^2} +\dfrac{\mu}{2}\|\nabla u\|_{L^2}^2 + \dfrac{2\kappa^2}{\mu} \|Q\|_{L^2}^2 + C  \left( \|Q\|_{L^4}^2+ \|Q\|_{L^6}^3  \right)\cdot \left( \|Q\|_{H^2}+ \|\nabla u\|_{L^2}\|Q\|_{L^\infty} \right) \\
	\leq & C\sum_{l=1}^{3}\mathcal{E}^{\frac{l}{2}}(t)\mathcal{D}(t) +\dfrac{\mu}{2}\|\nabla u\|_{L^2}^2 + \dfrac{2\kappa^2}{\mu} \|Q\|_{L^2}^2.
\end{align*}
Let
\begin{equation}\label{def-M}
	M:= \max\left\lbrace 1, \dfrac{4\kappa^2}{a\mu\Gamma}\right\rbrace,
\end{equation}
we immediately get
\begin{equation}\label{eq3.6}
	\begin{split}
		\dfrac{1}{2}\dfrac{d}{dt}\Big( \|u\|_{L^2}^2 + M\|Q\|_{L^2}^2 + &\|\nabla Q\|_{L^2}^2 \Big)
		+\dfrac{\mu}{2}\|\nabla u\|_{L^2}^2 + \dfrac{aM\Gamma}{2}\| Q\|_{L^2}^2\\
		   &+ (a+M)\Gamma\|\nabla Q\|_{L^2}^2 + \Gamma\|\Delta
		Q\|_{L^2}^2\\
		\leq &  C\sum_{l=1}^{3}\mathcal{E}^{\frac{l}{2}}(t)\mathcal{D}(t).
	\end{split}
\end{equation}
\noindent\text{\em Step~2: Higher-order energy estimate.}
In this step, we similarly apply $\partial^k$ ($1\leq |k|\leq s$) to system \eqref{eq1.1}, and sum up the second equation of resulting system multiplied by $\partial^ku $ and the first equation of resulting system multiplied by $M\partial^kQ -\partial^k\Delta Q $, take the trace, and then integrate by parts to find
\begin{align}\label{eq3.7}
		\dfrac{1}{2}\dfrac{d}{dt}\Big( \|\partial^ku\|_{L^2}^2& +  M\|\partial^kQ\|_{L^2}^2 +  \|\partial^k\nabla Q\|_{L^2}^2 \Big)
		+\mu\|\partial^{k}\nabla u\|_{L^2}^2 \nonumber\\
		&+ aM\Gamma \|\partial^{k} Q\|_{L^2}^2+ (a+M)\Gamma\|\partial^{k}\nabla Q\|_{L^2}^2 + \Gamma\|\partial^k\Delta
		Q\|_{L^2}^2\nonumber\\
		=&  \left\langle u\cdot\nabla \partial^kQ:\partial^k\Delta
		Q  \right\rangle + \left\langle [\partial^k,u]_{\rm div} Q:\partial^k\Delta
		Q-M\partial^kQ  \right\rangle -\left\langle [\partial^k,u]_{\rm div} u:\partial^k u \right\rangle\nonumber\\
		& +M\left\langle \partial^k\Omega Q -  Q\partial^k\Omega:\partial^k Q\right\rangle +\left\langle[\partial^k,(\Omega,Q)]_-:M\partial^k Q-\partial^k\Delta
		Q \right\rangle\nonumber\\
		& +  \left\langle  \partial^k\left( \nabla Q \odot\nabla Q\right):\partial^{k}\nabla u  \right\rangle + \left\langle[\partial^k,(\Delta Q,Q)]_-: \partial^k\nabla u \right\rangle - \kappa \left\langle  \partial^kQ:\partial^k\nabla u   \right\rangle\nonumber\\
		&+ \lambda M\left\langle \partial^k\left(  |Q|D\right):\partial^k Q  \right\rangle -a\lambda \left\langle \partial^k\left(  |Q|Q\right):\partial^k \nabla u  \right\rangle  \nonumber\\
		& - \lambda \left\langle [\partial^k,|Q|] D:\partial^k\Delta Q  \right\rangle +\lambda \left\langle [\partial^k,|Q|]\Delta Q:\partial^k\nabla u  \right\rangle \\
		& +\Gamma  \left\langle  \partial^k\left\lbrace b\left[Q^2-\dfrac{{\rm Tr}(Q^2)}{3}\mathbb{I}_3\right]-cQ{\rm Tr}(Q^2) \right\rbrace:M\partial^kQ -\partial^k\Delta Q  \right\rangle\nonumber\\
		&+ \lambda\left\langle \partial^k\left(|Q| \left\lbrace b\left[Q^2-\dfrac{{\rm Tr}(Q^2)}{3}\mathbb{I}_3\right]-cQ {\rm Tr}(Q^2) \right\rbrace\right):\partial^k\nabla u \right\rangle\nonumber\\
		\overset{\bigtriangleup}{=}& \sum_{j=1}^{14} L_j.\nonumber
\end{align}
By a similar procedure as the proof of Proposition~\ref{prop2.1},
the nonlinear terms on the right-hand side of \eqref{eq3.7}
can be estimated as follows:
\begin{align*}
		|L_1|+|L_4|\leq &  \| u\|_{L^\infty}\|\nabla \partial^kQ\|_{L^2}\|\partial^k\Delta
		Q\|_{L^2}+ M\| Q\|_{L^\infty}\|\partial^{k}\nabla u\|_{L^2}\|\partial^k Q  \|_{L^2}\\
		\leq & C\mathcal{E}^{\frac{1}{2}}(t)\mathcal{D}(t)\\
		|L_2|\leq &  \left( \|\nabla u\|_{L^\infty} \|\partial^{k-1} \nabla  Q\|_{L^2}+ \|\nabla  Q\|_{L^\infty}\|\partial^{k} u\|_{L^2}\right)\left( M\|\partial^{k}  Q \|_{L^2} + \|\partial^{k} \Delta  Q \|_{L^2}\right)\\
		\leq & C\left( \|\nabla u\|_{H^2} \|  Q\|_{H^{s}}+ \|\nabla  Q\|_{H^2}\|  u\|_{H^s}\right)\left(M \|  Q\|_{H^s} + \| \Delta  Q\|_{H^s}\right)\\
		\leq & C\mathcal{E}^{\frac{1}{2}}(t)\mathcal{D}(t)\\
		|L_3|\leq & \left( \|\nabla u \|_{L^\infty}\|\partial^{k-1}\nabla u \|_{L^2}+\|\nabla u \|_{L^\infty} \|\partial^{k} u \|_{L^2}\right)  \|\partial^{k}  u \|_{L^2}\\
		\leq & C \|\nabla u \|_{H^2}\| u \|_{H^{s}}\mathcal{D}^{\frac{1}{2}}(t)   \qquad (|k|\geq 1)\\
		\leq & C\mathcal{E}^{\frac{1}{2}}(t)\mathcal{D}(t)\\
		|L_5|\leq & \left( \|\nabla Q\|_{L^\infty}\|\partial^{k-1}\nabla u \|_{L^2}+ \|\nabla u \|_{L^\infty}\|\partial^{k}Q \|_{L^2} \right)\left( M\|\partial^{k}  Q \|_{L^2} + \|\partial^{k} \Delta  Q \|_{L^2}\right)\\
		\leq & C\left( \|\nabla Q \|_{H^2}\| u \|_{H^{s}}+ \|\nabla u \|_{H^2}\| Q \|_{H^s} \right)\mathcal{D}^{\frac{1}{2}}(t)\\
		\leq & C\mathcal{E}^{\frac{1}{2}}(t)\mathcal{D}(t)\\
		|L_6|+ |L_8|\leq &   C\|\partial^{k}\nabla Q  \|_{L^2} \|\nabla Q  \|_{L^\infty} \|\partial^{k}\nabla u   \|_{L^2} + |\kappa|\|\partial^k \nabla u\|_{L^2}\|\partial^k Q\|_{L^2}\\
		\leq & C\mathcal{E}^{\frac{1}{2}}(t)\mathcal{D}(t) + \dfrac{\mu}{2}\|\partial^k\nabla u\|_{L^2}^2 + \dfrac{2\kappa^2}{\mu} \|\partial^kQ\|_{L^2}^2\\
		|L_7|\leq &   \left(  \|\nabla  Q \|_{L^\infty}  \|\partial^{k-1}\Delta Q \|_{L^2} +  \| \Delta Q \|_{L^\infty}  \|\partial^{k}  Q \|_{L^2}  \right) \|\partial^{k}\nabla u  \|_{L^2}\\
		\leq & C \left(  \|\nabla  Q  \|_{H^2}  \|\nabla Q  \|_{H^{s}} +  \| \Delta Q \|_{H^2}  \| Q \|_{H^s}  \right) \mathcal{D}^{\frac{1}{2}}(t)\\
		\leq &  C\mathcal{E}^{\frac{1}{2}}(t)\mathcal{D}(t)\\
		|L_9|\leq &  C\left(  \|Q \|_{L^\infty}  \|\partial^{k}\nabla u \|_{L^2} +  \| \nabla u \|_{L^\infty}  \|\partial^{k}  |Q| \|_{L^2}  \right) \|\partial^{k}  Q \|_{L^2}  \\
		\leq &  C\left(  \| Q\|_{H^{2}}  \|\nabla u  \|_{H^{s}} +  \| \nabla u\|_{H^2}  \| Q\|_{H^s}  \right)\mathcal{D}^{\frac{1}{2}}(t)  \\
		\leq & C\mathcal{E}^\frac{1}{2}(t)\mathcal{D}(t)\\
		|L_{11}| \leq & C\left(  \|\nabla  |Q| \|_{L^\infty}  \|\partial^{k-1}\nabla u \|_{L^2} +  \| \nabla u \|_{L^\infty}  \|\partial^{k}  |Q| \|_{L^2}  \right) \|\partial^k\Delta Q \|_{L^2}\\
		\leq &   C\left(  \|\nabla Q \|_{H^{2}}  \|u  \|_{H^{s}} +  \| \nabla u \|_{H^2}  \| Q  \|_{H^s}  \right) \|\Delta Q  \|_{H^s}\\
		\leq &  C\mathcal{E}^{\frac{1}{2}}(t)\mathcal{D}(t)\\
		|L_{12}| \leq &  C\left(  \|\nabla  |Q| \|_{L^\infty}  \|\partial^{k-1}\Delta Q \|_{L^2} +  \| \Delta Q \|_{L^\infty}  \|\partial^{k}  |Q| \|_{L^2}  \right) \|\partial^k\nabla u \|_{L^2}\\
		\leq &   C\left(  \|\nabla Q \|_{H^{2}}  \|\nabla Q  \|_{H^{s}} +  \| \Delta Q  \|_{H^2}  \| Q  \|_{H^s}  \right) \|\nabla u  \|_{H^s}\\
		\leq &  C\mathcal{E}^{\frac{1}{2}}(t)\mathcal{D}(t)\\
		|L_{10}|+ |L_{13}| + |L_{14}|\lesssim & \mathcal{D}^{\frac{1}{2}}(t)\Big( \|Q\|_{L^\infty}\|\partial^k Q\|_{L^2} + \|Q\|_{L^\infty}\|\partial^k |Q|\|_{L^2} + \|Q\|_{L^\infty}^2\|\partial^k Q\|_{L^2}\\
		& + \|Q\|_{L^\infty}^2\|\partial^k |Q|\|_{L^2} + \|Q\|_{L^\infty}^3\|\partial^k Q\|_{L^2} + \|Q\|_{L^\infty}^3\|\partial^k |Q|\|_{L^2}\Big)\\
		\leq & C\sum_{l=1}^{3}\mathcal{E}^{\frac{l}{2}}(t)\mathcal{D}(t).
\end{align*}
Together with \eqref{def-M}, the above results imply that
\begin{equation}\label{eq3.23}
	\begin{split}
		\dfrac{1}{2}\dfrac{d}{dt}\Big( \|\partial^ku\|_{L^2}^2 +  M\|\partial^kQ\|_{L^2}^2 + & \|\partial^k\nabla Q\|_{L^2}^2 \Big)
		+\dfrac{\mu}{2}\|\partial^{k}\nabla u\|_{L^2}^2 + \dfrac{aM\Gamma}{2} \|\partial^{k} Q\|_{L^2}^2\\
		&+ (a+M)\Gamma\|\partial^{k}\nabla Q\|_{L^2}^2 + \Gamma\|\partial^k\Delta
		Q\|_{L^2}^2\\
		\leq& C \sum_{l=1}^{3}\mathcal{E}^{\frac{l}{2}}(t)  \mathcal{D}(t) .
	\end{split}
\end{equation}
Finally, we get \eqref{es-global} by summing \eqref{eq3.6} and \eqref{eq3.23} for all $1\leq |k|\leq s$. This completes the proof.
\end{proof}
\noindent\textbf{Proof of Theorem\ref{thm1.2}}~~
For small initial data, together with the a priori estimate \eqref{es-global}, the global existence result for system \eqref{eq1.1}
follows from the standard continuation argument \cite{nishida1978}.
This completes the proof.$\hfill\Box$



\section{Analysis of the linearized system}\label{Sec.4}
As already pointed out in Section~\ref{sec.1}, in order to obtain the decay rates of the global solution stated in Theorem~\ref{thm1.2}, the idea is to combine the energy method with the Fourier analysis of the corresponding linear system.
Thus, we reformulate the system \eqref{eq1.1} at equilibrium state $(0,0)$ as follows:
\begin{equation}\label{eq4.1}
	\begin{cases}
		Q_t -\Gamma \Delta Q+a\Gamma Q =-{\rm  div}f_1 + f_2,\\
		u_t-\mu \Delta u -\kappa \mathbb{P}{\rm div} Q=-\mathbb{P} {\rm  div}f_3,
	\end{cases}
\end{equation}
where
\begin{equation}\label{eq4.2}
	\begin{split}
		f_1=& u\otimes Q,\\
		f_2= & \Omega Q-Q\Omega+\lambda |Q|D + b\Gamma\left[Q^2-\dfrac{{\rm Tr}(Q^2)}{3}\mathbb{I}_3\right]-c\Gamma Q{\rm Tr}(Q^2),\\
		f_3=& u\otimes u + \nabla Q \odot\nabla Q- Q\Delta Q+\Delta QQ +\lambda |Q|H[Q].
	\end{split}
\end{equation}
Observe that the linearized system of \eqref{eq4.1} is indeed \eqref{eq2.1}.
Hence, as we explained above, we shall analyze the following auxiliary linearized system:
\begin{equation}\label{lin-sys}
	\begin{cases}
		\sigma_t -\Gamma \Delta \sigma +a\Gamma \sigma =0,\\
		u_t-\mu \Delta u -\kappa \sigma=0,\\
		(\sigma,u)(0,x)=(\sigma_0,u_0):= (\mathbb{P} {\rm div}Q_0,u_0)(x).
	\end{cases}
\end{equation}
In general, although obtaining the explicit formulation of the Green function $\mathbb{G}_{\sigma,u}(t,x)$ for this system is difficult, one may seek the associated Fourier transform of the Green function.
More precisely,
we have an explicit expression of $\widehat{\mathbb{G}_{\sigma,u}}(t,\xi)$ in the following Lemma.
\begin{lemma}\label{lem4.1}
	Fourier transform of the solution to the linearized system \eqref{lin-sys} is
	\begin{equation}
		\begin{bmatrix}
			\hat{\sigma}(t,\xi)\\
			\hat{u}(t,\xi)
		\end{bmatrix}
		=
		\widehat{\mathbb{G}_{\sigma,u}}(t,\xi)
		\begin{bmatrix}
			\hat{\sigma}_0\\
			\hat{u}	_0
		\end{bmatrix},
	\end{equation}
	where
	\begin{equation}
		\widehat{\mathbb{G}_{\sigma,u}}(t,\xi):=
		\begin{bmatrix}
			\mathcal{A}(t,\xi)\mathbb{I}_3 & 0 \\
			\mathcal{B}(t,\xi)\mathbb{I}_3 &  \mathcal{C}(t,\xi)\mathbb{I}_3
		\end{bmatrix},
	\end{equation}
	with
	\begin{equation}
	\mathcal{A}(t,\xi):= e^{-\Gamma (|\xi|^2+a)t},\quad  \mathcal{C}(t,\xi):= e^{-\mu |\xi|^2t},
	\end{equation}
	and $\mathcal{B}(t,\xi)$ is defined by the following two cases:
	\begin{enumerate}[(i).]
		\item $\mu\leq \Gamma$ and $a> 0$:
		\begin{equation}\label{eqB-1}
			\mathcal{B}(t,\xi):=\dfrac{ \kappa}{\mu|\xi|^2 -\Gamma |\xi|^2-a\Gamma} \left(e^{ -\Gamma\left( |\xi|^2+a  \right)t} -e^{-\mu|\xi|^2t}  \right)
		\end{equation}
		\item $\mu> \Gamma$ and $a> 0$:
		\begin{equation}\label{eqB-2}
			\mathcal{B}(t,\xi):=
			\begin{cases}
				\dfrac{ \kappa}{\mu|\xi|^2 -\Gamma |\xi|^2-a\Gamma} \left(e^{ -\Gamma\left( |\xi|^2+a  \right)t} -e^{-\mu|\xi|^2t}  \right), & \text{if $|\xi|\ne\sqrt{\dfrac{a\Gamma}{\mu-\Gamma}}$}\\[2ex]
				\kappa e^{-\mu|\xi|^2t} t, & \text{if $|\xi|=\sqrt{\dfrac{a\Gamma}{\mu-\Gamma}}$}\\
			\end{cases}
		\end{equation}
	\end{enumerate}
\end{lemma}
\begin{proof}
	By applying Fourier transform to the auxiliary system \eqref{lin-sys}, we aim to study the following initial value problem
	\begin{equation}\label{eq1.17}
		\begin{cases}
			\hat{\sigma}_t +\Gamma |\xi|^2 \hat{\sigma}+a\Gamma \hat{\sigma} =0,\\
			\hat{u}_t+\mu |\xi|^2 \hat{u} -\kappa \hat{\sigma}=0,\\
			(\hat{\sigma},\hat{u})\mid_{t=0}=(\widehat{\sigma_0},\widehat{u_0}),
		\end{cases}
	\end{equation}
	where
	\begin{equation}\label{eq4.8}
	(\widehat{\sigma_0})_l = i \sum_{j=1}^3 \left(\left( \delta_{lj} -\dfrac{\xi_l \xi_j}{|\xi|^2} \right) \cdot\sum_{k=1}^3 (\widehat{Q_0})_{jk}\xi_k \right),\quad  1\leq l\leq 3.
	\end{equation}
	Obviously, it follows from \eqref{eq1.17}$_{1}$ that
	\begin{equation}\label{eq:4.12}
		\hat{\sigma}= e^{-\Gamma (|\xi|^2+a)t}\widehat{\sigma_0}.
	\end{equation}
	On the other hand, multiplying the second equation of \eqref{eq1.17} by $e^{\mu|\xi|^2t}$, one has
	\begin{equation}\label{eq4.7}
		\left[  e^{\mu|\xi|^2t} \hat{u} \right]_t = \kappa e^{\mu|\xi|^2t} \hat{\sigma}
		=\kappa e^{ \left(\mu|\xi|^2 -\Gamma |\xi|^2-a\Gamma  \right)t} \widehat{\sigma_0}.
	\end{equation}
	For the case $(i)$, it is not difficult to see that
	\[
	\mu|\xi|^2 -\Gamma |\xi|^2-a\Gamma\ne 0.
	\]
	Thus, one can solve \eqref{eq4.7} to get
	\begin{equation*}
		e^{\mu|\xi|^2t} \hat{u} =  \dfrac{\kappa}{\mu|\xi|^2 -\Gamma |\xi|^2-a\Gamma}  e^{ \left(\mu|\xi|^2 -\Gamma |\xi|^2-a\Gamma  \right)t} \widehat{\sigma_0} +C,
	\end{equation*}
	where $C$ is determined by
	\begin{equation*}
		C:= \widehat{u_0} - \dfrac{\kappa}{\mu|\xi|^2 -\Gamma |\xi|^2-a\Gamma} \widehat{\sigma_0}.
	\end{equation*}
	Hence,
	\begin{equation}\label{eq4.12}
		\hat{u}= \dfrac{\kappa}{\mu|\xi|^2 -\Gamma |\xi|^2-a\Gamma} \left(e^{ -\Gamma\left( |\xi|^2+a  \right)t} -e^{-\mu|\xi|^2t}  \right)\widehat{\sigma_0} + e^{-\mu|\xi|^2t}  \widehat{u_0}.
	\end{equation}
	As for the case $(ii)$, one should note that
	\[
	\mu|\xi|^2 -\Gamma |\xi|^2-a\Gamma = 0
	\]
	when $|\xi|=\sqrt{\dfrac{a\Gamma}{\mu-\Gamma}}$. In this setting, \eqref{eq4.7} reduces to
	\begin{equation}
		\left[  e^{\mu|\xi|^2t} \hat{u} \right]_t = \kappa  \widehat{\sigma_0}.
	\end{equation}
	Analogously, the solution of the above equation is
	\begin{equation}\label{eq4.15}
		\hat{u}= \kappa e^{-\mu|\xi|^2t} \widehat{\sigma_0}t + e^{-\mu|\xi|^2t}  \widehat{u_0}.
	\end{equation}
	In conclusion, from \eqref{eq:4.12}, \eqref{eq4.12}, \eqref{eq4.15}, we obtain the explicit Fourier transform of the Green function for system \eqref{lin-sys}. This completes the proof of Lemma~\ref{lem4.1}.
\end{proof}
As a corollary to this Lemma, we have the following result immediately.
\begin{cor}\label{cor4.1}
	The explicit expression of $\widehat{\mathbb{G}_{Q,u}}(t,\xi)$ is determined by
	\begin{equation}
		\widehat{Q}(t,\xi)= \mathcal{A}(t,\xi)\widehat{Q_0}
	\end{equation}
	and
	\begin{equation}
		\hat{u}_l(t,\xi)= i\mathcal{B}(t,\xi)   \sum_{j=1}^3 \left(\left( \delta_{lj} -\dfrac{\xi_l \xi_j}{|\xi|^2} \right) \cdot\sum_{k=1}^3 (\widehat{Q_0})_{jk}\xi_k \right)  + \mathcal{C}(t,\xi)  (\widehat{u_0})_l,\quad 1\leq l\leq 3,
	\end{equation}
	where $\mathcal{A}(t,\xi)$, $\mathcal{B}(t,\xi)$ and $\mathcal{C}(t,\xi)$ are the same as in Lemma~\ref{lem4.1}.
\end{cor}
Now, we are in a position to give the $L^2$ estimates of $\check{\mathcal{A}}(t,\xi)$, $\check{\mathcal{B}}(t,\xi)$, $\check{\mathcal{C}}(t,\xi)$, which play a crucial role in studying the decay properties of solutions to system \eqref{eq1.1}.
\begin{lemma}\label{lem4.2}
	Suppose that $f\in L^1\cap H^k$ $(k\geq 0)$, we have
	\begin{enumerate}[(i).]
		\item $\left\| \partial_x^k \left(  \check{\mathcal{A}}(t,x)\ast f \right)\right\|_{L^2}\leq Ce^{-a\Gamma t}(1+t)^{-\frac{3}{4}-\frac{k}{2}}\|f\|_{L^1\cap H^k}$,
		\item $\left\|\partial_x^k \left(  \check{\mathcal{B}}(t,x)\ast f \right)\right\|_{L^2}\leq C(1+t)^{-\frac{3}{4}-\frac{k}{2}}\|f\|_{L^1\cap H^k}$,
		\item $\left\| \partial_x^k \left( \check{\mathcal{C}}(t,x)\ast f \right)\right\|_{L^2}\leq C(1+t)^{-\frac{3}{4}-\frac{k}{2}}\|f\|_{L^1\cap H^k}$,
	\end{enumerate}
	where $C$ is a positive constant depending on $s$, $a$, $\mu$, $\kappa$ and $\Gamma$. Moreover, it holds that
	\begin{equation}\label{eq417}
		\left\| \partial_x^k \left( \nabla \check{\mathcal{A}}(t,x)\ast f \right)\right\|_{L^2}\lesssim e^{-a\Gamma t}\left((1+t)^{-\frac{5}{4}-\frac{k}{2}}\|f\|_{L^1} +\frac{1}{\sqrt{t}}\|f\|_{H^k}\right).
	\end{equation}
\end{lemma}
\begin{proof}
First, we deal with the estimate of $\check{\mathcal{A}}(t,x)$. Following \cite{Liu1987,Matsumura76}, by virtue of the Hausdorff-Young inequality, we have
\begin{equation}
	\begin{split}
		\Big\| \partial_x^k  (\check{\mathcal{A}}\ast f) \Big\|_{L^2}^2 =&  \int_{\R^3}|\xi|^{2k} \left| \mathcal{A}(t,\xi)\right|^2  |\hat{f}(\xi)|^2d\xi	\overset{\bigtriangleup}{=} \underbrace{\int_{|\xi|\geq R}}_{\text{\uppercase\expandafter{\romannumeral1}}_1} + \underbrace{\int_{|\xi|< R}}_{\text{\uppercase\expandafter{\romannumeral1}}_2}\\
		\leq & e^{-2a\Gamma t}\int_{|\xi|< R} |\xi|^{2k} e^{-2\Gamma |\xi|^2 t}  |\hat{f}(\xi)|^2d\xi + e^{-2a\Gamma t}\int_{|\xi|\geq R} |\xi|^{2k}   |\hat{f}(\xi)|^2d\xi\\
		\leq & Ce^{-2a\Gamma t}(1+t)^{-\frac{3}{2}-k}\|\hat{f}(\xi)\|_{L^\infty}^2 + e^{-2a\Gamma t} \|f\|_{H^k}^2\\
		\leq & Ce^{-2a\Gamma t}(1+t)^{-\frac{3}{2}-k}\|f\|_{L^1\cap H^k}^2.
	\end{split}
\end{equation}
Similarly, one can see that
\begin{equation}\label{eq4.18}
	\begin{split}
		\Big\| \partial_x^k  (\check{\mathcal{C}}\ast f) \Big\|_{L^2}^2 =&  \int_{\R^3}|\xi|^{2k} \left| \mathcal{C}(t,\xi)\right|^2  |\hat{f}(\xi)|^2d\xi	\overset{\bigtriangleup}{=} \underbrace{\int_{|\xi|\geq R}}_{\text{\uppercase\expandafter{\romannumeral2}}_1} 
		+ \underbrace{\int_{|\xi|< R}}_{\text{\uppercase\expandafter{\romannumeral2}}_2}\\
		\leq & \int_{|\xi|< R} |\xi|^{2k} e^{-2\mu |\xi|^2 t}  |\hat{f}(\xi)|^2d\xi + e^{-2\mu R^2 t}\int_{|\xi|\geq R} |\xi|^{2k}   |\hat{f}(\xi)|^2d\xi\\
		\leq & C(1+t)^{-\frac{3}{2}-k}\|\hat{f}(\xi)\|_{L^\infty}^2 + e^{-\mu R^2 t} \|f\|_{H^k}^2\\
		\leq & C(1+t)^{-\frac{3}{2}-k}\|f\|_{L^1\cap H^k}^2,
	\end{split}
\end{equation}
which completes the estimate of $\check{\mathcal{C}}(t,x)$.

Next, we show the estimate of $\check{\mathcal{B}}(t,x)$. From \eqref{eqB-1} and \eqref{eqB-2}, it can be seen that the expression of $\mathcal{B}(t,\xi)$ depends largely on the viscosity coefficients $\mu$ and $\Gamma^{-1}$. However, one can also see that we only need to prove that for $\mu> \Gamma$ and $a> 0$.
Hence, let $\epsilon>0$ be a small fixed number and $R:= \sqrt{\frac{a\Gamma}{\mu-\Gamma}}$, we infer that
\begin{align}\label{eq4.19}
	\Big\| \partial_x^k  (\check{\mathcal{B}}\ast f) \Big\|_{L^2}^2 =&  \int_{\R^3}|\xi|^{2k} \left| \mathcal{B}(t,\xi)\right|^2  |\hat{f}(\xi)|^2d\xi\nonumber\\	
	\overset{\bigtriangleup}{=}&  \underbrace{\int_{|\xi|\geq R+\epsilon}}_{\text{\uppercase\expandafter{\romannumeral3}}_1}+   \underbrace{\int_{R-\epsilon\leq |\xi|\leq  R+\epsilon}}_{\text{\uppercase\expandafter{\romannumeral3}}_2} 
	+ \underbrace{\int_{|\xi|< R-\epsilon}}_{\text{\uppercase\expandafter{\romannumeral3}}_3}\nonumber\\
	\leq & C\|\hat{f}(\xi)\|_{L^\infty}^2\int_{|\xi|< R-\epsilon} |\xi|^{2k} e^{-2\min\{\mu,\Gamma\} |\xi|^2 t} d\xi\nonumber\\
	& + Ce^{-2\min\{\mu(R+\epsilon)^2, \Gamma[(R+\epsilon)^2 +a]\}t}\int_{|\xi|> R+\epsilon} |\xi|^{2k}   |\hat{f}(\xi)|^2d\xi\\
	& +  \kappa^2\int_{R-\epsilon\leq|\xi|\leq R+\epsilon} |\xi|^{2k}  \left(\dfrac{e^{( \mu|\xi|^2-\Gamma  |\xi|^2-a\Gamma )t} -1}{\mu|\xi|^2 -\Gamma |\xi|^2-a\Gamma}\right)^2 e^{-2\mu|\xi|^2t} |\hat{f}(\xi)|^2d\xi  \nonumber\\
	\leq & C(1+t)^{-\frac{3}{2}-k}\|f\|_{L^1\cap H^k}^2 + \kappa^2t^2e^{-\frac{1}{2}\mu R^2t} \int_{R-\epsilon\leq|\xi|\leq R+\epsilon} |\xi|^{2k} |\hat{f}(\xi)|^2d\xi\nonumber\\
	\leq & C(1+t)^{-\frac{3}{2}-k}\|f\|_{L^1\cap H^k}^2 + Ce^{-\frac{1}{4}\mu R^2t} \|f\|_{H^k}^2\nonumber\\
	\leq & C(1+t)^{-\frac{3}{2}-k}\|f\|_{L^1\cap H^k}^2,\nonumber
\end{align}
where we have used the fact that
\begin{equation}
	\left| \dfrac{\kappa}{\mu|\xi|^2 -\Gamma |\xi|^2-a\Gamma} \right|\leq
	\begin{cases}
		\dfrac{\kappa}{\epsilon\sqrt{\mu-\Gamma}\sqrt{a\Gamma}}, & \text{$|\xi|<R-\epsilon$,}\\[2ex]
		\dfrac{\kappa}{\epsilon^2(\mu-\Gamma)}, & \text{$|\xi|>R+\epsilon$.}\\
	\end{cases}
\end{equation}

Finally, we estimate $\nabla \check{\mathcal{A}}(t,x)$ in the following:
\begin{align*}
	\Big\| \partial_x^k  (\nabla \check{\mathcal{A}}\ast f) \Big\|_{L^2}^2 =&  \int_{\R^3}|\xi|^{2k+2} \left| \mathcal{A}(t,\xi)\right|^2  |\hat{f}(\xi)|^2d\xi\\
	\leq & e^{-2a\Gamma t}\int_{|\xi|< R} |\xi|^{2k+2} e^{-2\Gamma |\xi|^2 t}  |\hat{f}(\xi)|^2d\xi + Ce^{-2a\Gamma t}\int_{|\xi|\geq R} |\xi|^{2k}\cdot t^{-1}  |\hat{f}(\xi)|^2d\xi\\
	\leq & Ce^{-2a\Gamma t}(1+t)^{-\frac{5}{2}-k}\|\hat{f}(\xi)\|_{L^\infty}^2 + Ct^{-1}e^{-2a\Gamma t} \|f\|_{H^k}^2\\
	\leq & Ce^{-2a\Gamma t}\left((1+t)^{-\frac{5}{2}-k}\|f\|_{L^1}^2 +t^{-1}\|f\|_{H^k}^2\right),
\end{align*}
which yields \eqref{eq417} and this completes the proof.
\end{proof}
Based on Lemma~\ref{lem4.1}, Corollary~\ref{cor4.1} and Lemma~\ref{lem4.2},
we then conclude this section with the following result regarding decay estimates for the linearized system~\eqref{eq2.1}, which indicates that possibly stable mixing estimates can be expected for the simplified active system \eqref{eq1.1}.
\begin{theorem}\label{thm4.1}
Let $(Q_L,u_L)$ be the solution of system \eqref{eq2.1}, then it can be expressed via the Green function $\mathbb{G}_{Q,u}(t,x)$, i.e.,
\[
\begin{pmatrix}
	Q_L(t,x)\\
	u_L(t,x)
\end{pmatrix}
=
\mathbb{G}_{Q,u}(t,x) \ast
\begin{pmatrix}
	Q_0(x)\\
	u_0(x)
\end{pmatrix}.
\]
In addition, for any $s \in \mathbb{N}$, one has
\begin{equation}\label{eq4.23}
	\begin{split}
		\left\| \partial_x^k Q_L(t,x)\right\|_{L^2} &\lesssim\widebar{E}_0 e^{-a\Gamma t}(1+t)^{-\frac{3}{4}-\frac{k}{2}},\quad 0\leq k\leq s+1,\\
		\left\| \partial_x^k u_L(t,x)\right\|_{L^2} &\lesssim\widebar{E}_0(1+t)^{-\frac{3}{4}-\frac{k}{2}},\hspace{0.9cm}\quad 0\leq k\leq s.
	\end{split}
\end{equation}
\end{theorem}

\section{Decay estimates for the nonlinear system}\label{Sec.5}
In this section, we shall investigate the time-decay properties of the global solution $(Q,u)$ established in Theorem~\ref{thm1.2} for the case of $s\geq 4$.
In particular, we consider that the solution to \eqref{eq4.1} satisfies the following Duhamel formulation:
\begin{align}
	Q(t) =& Q_L(t) +\int_0^t  \check{\mathcal{A}}(t-s) \ast \left( f_2(s)-{\rm div}f_1(s) \right)  ds,\label{eq5.1}
	\\
	u(t) =& u_L(t) +\int_0^t  \check{\mathcal{B}}(t-s)\ast \mathbb{P}{\rm  div}\left( f_2(s)-{\rm div}f_1(s) \right) ds -\int_0^t  \check{\mathcal{C}}(t-s) \ast \mathbb{P}{\rm  div}f_3(s)  ds.\label{eq-5.2}
\end{align}
Hence, one can perform the $L^2$ time-decay estimates for $\partial^k Q$ and $\partial^k u$ via these integral representations of the solution.
However, we cannot expect optimal time-decay rates for all the derivatives of the solution in general,
because the nonlinear coupling terms in \eqref{eq4.2} involve certain highest derivatives in $Q$ and $u$ (such as $\nabla u$, ${\rm div}\Delta Q$), which eventually bring some difficulties in deriving the desired decay properties via the Green's function method. 
To overcome these obstacles, we first present the time-weighted estimates for the global solutions of system \eqref{eq1.1}, which will provide some decay properties for the highest derivatives of solutions (but not optimal). This enables us to solve the problem regarding loss of derivatives in nonlinear decay estimates.

\subsection{Time-weighted energy estimates}

Similar to our recent work \cite{Yang-blood}, we introduce the following time-weighted quantities:
\begin{equation*}
	\begin{split}
		\mathcal{N}(t):=&\sum_{k=0}^{s} (1+t)^k\left( \|\partial^k u\|_{H^{s-k}}^2 + M\|\partial^k Q\|_{H^{s-k}}^2 + \|\partial^k \nabla Q\|_{H^{s-k}}^2 \right),\\
		\mathcal{M}(t):=&\sum_{k=0}^{s} (1+t)^k\left( \dfrac{\mu}{2}\|\partial^{k}\nabla u\|_{H^{s-k}}^2 + \dfrac{aM\Gamma}{2} \|\partial^{k} Q\|_{H^{s-k}}^2+ (a+M)\Gamma\|\partial^{k}\nabla Q\|_{H^{s-k}}^2 + \Gamma\|\partial^k\Delta
		Q\|_{H^{s-k}}^2 \right),
	\end{split}
\end{equation*}
where $M$ is chosen to be the same as in \eqref{def-M}. Then, we employ an iteration argument to obtain 
the following time-weighted energy estimates for the global solution of system \eqref{eq1.1}.

\begin{prop}\label{prop5.1}
	Let $s\geq 4$ be an integer and $a>0$, $E_0<+\infty$.
	Suppose that $(Q,u)$ is the unique global solution established in Theorem \ref{thm1.2}. Then, it holds
	\begin{equation}\label{eq5.2}
		\dfrac{d}{dt}\mathcal{N}(t)+ \mathcal{M}(t)
		\leq C \sum_{m=1}^3 \mathcal{N}^\frac{m}{2}(t)\mathcal{M}(t),
	\end{equation}
	where $C$ is a positive constant depending only on $s$, $a$, $b$, $c$, $\mu$, $\kappa$, $\lambda$ and $\Gamma$.
\end{prop}
\begin{proof}
First of all, for any fixed $k\in [1,s]$, we aim to prove that
\begin{equation}\label{eq5.3}
	\begin{split}
		\dfrac{1}{2}\dfrac{d}{dt}\Big[ &(1+t)^k\Big( \|\partial^lu\|_{L^2}^2 +  M\|\partial^lQ\|_{L^2}^2 +  \|\partial^l\nabla Q\|_{L^2}^2 \Big)\Big]\\
		&+ (1+t)^k\left(\dfrac{\mu}{2}\|\partial^{l}\nabla u\|_{L^2}^2 + \dfrac{aM\Gamma}{2} \|\partial^{l} Q\|_{L^2}^2+ (a+M)\Gamma\|\partial^{l}\nabla Q\|_{L^2}^2 + \Gamma\|\partial^l\Delta
		Q\|_{L^2}^2 \right) \\
		\leq & \frac{k}{2}(1+t)^{k-1}\Big( \|\partial^lu\|_{L^2}^2 +  M\|\partial^lQ\|_{L^2}^2 +  \|\partial^l\nabla Q\|_{L^2}^2 \Big) + C\sum_{m=1}^3 \mathcal{N}^\frac{m}{2}(t)\mathcal{M}(t)
	\end{split}
\end{equation}
holds for all $k\leq |l|\leq s$.
Indeed, it follows from \eqref{linear-term}, \eqref{def-M}, \eqref{eq3.7} that
\begin{equation}\label{eq:5.4}
	\begin{split}
		\dfrac{1}{2}\dfrac{d}{dt}\Big[ &(1+t)^k\Big( \|\partial^lu\|_{L^2}^2 +  M\|\partial^lQ\|_{L^2}^2 +  \|\partial^l\nabla Q\|_{L^2}^2 \Big)\Big]\\
		&+ (1+t)^k\left(\dfrac{\mu}{2}\|\partial^{l}\nabla u\|_{L^2}^2 + \dfrac{aM\Gamma}{2} \|\partial^{l} Q\|_{L^2}^2+ (a+M)\Gamma\|\partial^{l}\nabla Q\|_{L^2}^2 + \Gamma\|\partial^l\Delta
		Q\|_{L^2}^2 \right) \\
		\leq & \frac{k}{2}(1+t)^{k-1}\Big( \|\partial^lu\|_{L^2}^2 +  M\|\partial^lQ\|_{L^2}^2 +  \|\partial^l\nabla Q\|_{L^2}^2 \Big)  + (1+t)^k \left\langle u\cdot\nabla \partial^lQ,\partial^l\Delta
		Q  \right\rangle \\
		&+ (1+t)^k\left\langle [\partial^l,u]_{\rm div} Q,\partial^l\Delta
		Q-M\partial^lQ  \right\rangle -(1+t)^k \left\langle [\partial^l,u]_{\rm div} u,\partial^l u \right\rangle\\
		& +M(1+t)^k\left\langle \partial^l\Omega Q -  Q\partial^l\Omega,\partial^l Q\right\rangle +(1+t)^k\left\langle[\partial^l,(\Omega,Q)]_-,M\partial^l Q-\partial^l\Delta
		Q \right\rangle \\
		& +  (1+t)^k\left\langle  \partial^l\left( \nabla Q \odot\nabla Q\right),\partial^{l}\nabla u  \right\rangle + (1+t)^k\left\langle[\partial^l,(\Delta Q,Q)]_-, \partial^l\nabla u \right\rangle \\
		& + \lambda M(1+t)^k\left\langle \partial^l\left(  |Q|D\right),\partial^l Q  \right\rangle -a\lambda(1+t)^k \left\langle \partial^l\left(  |Q|Q\right),\partial^l \nabla u  \right\rangle  \\
		& - \lambda (1+t)^k\left\langle [\partial^l,|Q|] D,\partial^l\Delta Q  \right\rangle +\lambda (1+t)^k\left\langle [\partial^l,|Q|]\Delta Q,\partial^l\nabla u  \right\rangle \\
		& +\Gamma (1+t)^k \left\langle  \partial^l\left\lbrace b\left[Q^2-\dfrac{{\rm Tr}(Q^2)}{3}\mathbb{I}_3\right]-cQ{\rm Tr}(Q^2) \right\rbrace,M\partial^lQ -\partial^l\Delta Q  \right\rangle \\
		&+ \lambda(1+t)^k\left\langle \partial^l\left(|Q| \left\lbrace b\left[Q^2-\dfrac{{\rm Tr}(Q^2)}{3}\mathbb{I}_3\right]-cQ {\rm Tr}(Q^2) \right\rbrace\right),\partial^l\nabla u \right\rangle \\
		\overset{\bigtriangleup}{=}& \frac{k}{2}(1+t)^{k-1}\Big( \|\partial^lu\|_{L^2}^2 +  M\|\partial^lQ\|_{L^2}^2 +  \|\partial^l\nabla Q\|_{L^2}^2 \Big) +\sum_{j=1}^{13} \mathcal{Z}_j.
	\end{split}
\end{equation}
According to the definition of $\mathcal{N}(t)$, $\mathcal{M}(t)$ with $s\geq 4$, it is easy to check that
\begin{align}
	\|(Q,u)\|_{L^\infty}\lesssim& \|(Q,u)\|_{H^s}\leq  \mathcal{N}^{\frac{1}{2}}(t),\label{eq5.5}\\
	\|\nabla(|Q|,Q,u)\|_{L^\infty}\lesssim& \|\nabla(Q,u)\|_{H^s}\leq  \mathcal{M}^{\frac{1}{2}}(t),\label{eq5.6}\\
	\|\nabla (Q,u)\|_{L^\infty}\lesssim &\|\partial^3 Q\|_{L^2}^{\frac{3}{4}}\| \nabla Q\|_{L^2}^{\frac{1}{4}} + \|\partial^3 u\|_{L^2}^{\frac{3}{4}}\| \nabla u\|_{L^2}^{\frac{1}{4}}\leq \mathcal{N}^{\frac{1}{2}}(t)(1+t)^{-\frac{5}{4}},\\
	\|\Delta Q\|_{L^\infty}\lesssim &\|\partial^4 Q\|_{L^2}^{\frac{3}{4}}\| \partial^2 Q\|_{L^2}^{\frac{1}{4}}\leq \mathcal{N}^{\frac{1}{2}}(t)(1+t)^{-\frac{11}{8}}.\label{eq5.8}
\end{align}
Similarly to the proof of \eqref{eq3.23}, the remainder terms in the right-hand side of \eqref{eq:5.4} are estimated as follows:
\begin{align*}
		|\mathcal{Z}_1|+|\mathcal{Z}_4|\leq &  C(1+t)^k\left( \| u\|_{L^\infty}\|\partial^l\nabla Q\|_{L^2}\|\partial^l\Delta
		Q\|_{L^2}+ \| Q\|_{L^\infty}\|\partial^{l}\nabla u\|_{L^2}\|\partial^l Q  \|_{L^2}\right)\\
		\leq &  C\mathcal{N}^{\frac{1}{2}}(t)(1+t)^k\left( \|\partial^k \nabla Q\|_{H^{s-k}}\|\partial^k \Delta Q\|_{H^{s-k}}+ \|\partial^{k}\nabla u\|_{H^{s-k}}\|\partial^k Q  \|_{H^{s-k}}\right)\\
		\leq & C\mathcal{N}^{\frac{1}{2}}(t)\mathcal{M}(t),\\
		|\mathcal{Z}_2|\leq & C(1+t)^k  \left( \|\nabla u\|_{L^\infty} \|\partial^{l-1} \nabla  Q\|_{L^2}+ \|\nabla  Q\|_{L^\infty}\|\partial^{l} u\|_{L^2}\right)\left( \|\partial^{l}  Q \|_{L^2} + \|\partial^{l} \Delta  Q \|_{L^2} \right)\\
		\leq & C\mathcal{M}^{\frac{1}{2}}(t) \left[(1+t)^k\left( \|\partial^{k} Q\|_{H^{s-k}}+  \|\partial^{k} u \|_{H^{s-k}}\right) \left( \|\partial^{k}  Q \|_{H^{s-k}} + \|\partial^{k} \Delta  Q \|_{H^{s-k}} \right)\right]\\
		\leq & C\mathcal{N}^{\frac{1}{2}}(t)\mathcal{M}(t),\\
		|\mathcal{Z}_3|\leq & C(1+t)^k\left( \|\nabla u \|_{L^\infty}\|\partial^{l-1}\nabla u \|_{L^2}+\|\nabla u \|_{L^\infty} \|\partial^{l} u \|_{L^2}\right)  \|\partial^{l}  u \|_{L^2}\\
		\leq &  C(1+t)\cdot \mathcal{N}^{\frac{1}{2}}(t)(1+t)^{-\frac{5}{4}} \cdot (1+t)^{k-1} \|\partial^{k-1}\nabla u \|_{H^{s-k+1}}^2\\
		\leq & C\mathcal{N}^{\frac{1}{2}}(t)\mathcal{M}(t),\\
		|\mathcal{Z}_5|\leq & C(1+t)^k\left( \|\nabla Q\|_{L^\infty}\|\partial^{l-1}\nabla u \|_{L^2}+ \|\nabla u \|_{L^\infty}\|\partial^{l}Q \|_{L^2} \right)\left( \|\partial^{l}  Q \|_{L^2} + \|\partial^{l} \Delta  Q \|_{L^2}\right)\\
		\leq & C\mathcal{M}^{\frac{1}{2}}(t)  \left[(1+t)^k\left(   \|\partial^{k} u \|_{H^{s-k}} + \|\partial^{k} Q\|_{H^{s-k}}\right)\left( \|\partial^{k}  Q \|_{H^{s-k}} + \|\partial^{k} \Delta  Q \|_{H^{s-k}} \right)\right]\\
		\leq & C\mathcal{N}^{\frac{1}{2}}(t)\mathcal{M}(t),\\
		|\mathcal{Z}_6| + |\mathcal{Z}_7| \leq &  C(1+t)^k \Big[\|\nabla Q  \|_{L^\infty} \left(\|\partial^{l}\nabla Q  \|_{L^2} +  \|\partial^{l-1}\Delta Q \|_{L^2}\right) +  \| \Delta Q \|_{L^\infty}  \|\partial^{l}  Q \|_{L^2}  \Big] \|\partial^{l}\nabla u  \|_{L^2}\\
		\leq &  C\left(\|\nabla Q  \|_{L^\infty}+ \|\Delta Q  \|_{L^\infty} \right)\cdot \Big[ (1+t)^k\left(\|\partial^{k}\nabla Q  \|_{H^{s-k}} + \|\partial^{k} Q  \|_{H^{s-k}}\right)\|\partial^{k}\nabla u  \|_{H^{s-k}} \Big] \\
		\leq & C\mathcal{N}^{\frac{1}{2}}(t)\mathcal{M}(t),\\
		|\mathcal{Z}_8| + |\mathcal{Z}_{10}|\leq &  C(1+t)^{\frac{k}{2}}\mathcal{M}^{\frac{1}{2}}(t)\left( \|\nabla  |Q| \|_{L^\infty}  \|\partial^{l-1}\nabla u \|_{L^2} +  \|Q \|_{L^\infty}  \|\partial^{l}\nabla u \|_{L^2} +  \| \nabla u \|_{L^\infty}  \|\partial^{l}  |Q| \|_{L^2}  \right)   \\
		\leq & C\left[  \mathcal{M}(t) (1+t)^{\frac{k}{2}}\left(\| \partial^k u \|_{H^{s-k}} + \| \partial^k Q \|_{H^{s-k}} \right)+ \mathcal{M}^{\frac{1}{2}}(t)\mathcal{N}^{\frac{1}{2}}(t)(1+t)^{\frac{k}{2}}\| \partial^k \nabla u \|_{H^{s-k}} \right]  \\
		\leq & C\mathcal{N}^{\frac{1}{2}}(t)\mathcal{M}(t),\\
		|\mathcal{Z}_{11}| \leq &  C(1+t)^k\left(  \|\nabla  |Q| \|_{L^\infty}  \|\partial^{l-1}\Delta Q \|_{L^2} +  \| \Delta Q \|_{L^\infty}  \|\partial^{l}  |Q| \|_{L^2}  \right) \|\partial^l\nabla u \|_{L^2}\\
		\leq &   C (1+t)^{\frac{k}{2}}\| \partial^k\nabla u \|_{H^{s-k}}\cdot (1+t)^{\frac{k}{2}}\left(\mathcal{M}^{\frac{1}{2}}(t)\| \partial^k\nabla Q \|_{H^{s-k}} + \mathcal{N}^{\frac{1}{2}}(t)\| \partial^k Q \|_{H^{s-k}}  \right)  \\
		\leq & C\mathcal{N}^{\frac{1}{2}}(t)\mathcal{M}(t),\\
		|\mathcal{Z}_{9}| + |\mathcal{Z}_{12}| + |\mathcal{Z}_{13}|\leq & C(1+t)^{\frac{k}{2}}\mathcal{M}^{\frac{1}{2}}(t)\Big( \|\partial^l Q\|_{L^2} +\|\partial^l |Q|\|_{L^2} \Big)\sum_{m=1}^3\|Q\|_{L^\infty}^m\\
		\leq & C\sum_{m=1}^{3}\mathcal{N}^{\frac{m}{2}}(t)\mathcal{M}^{\frac{1}{2}}(t)\cdot(1+t)^{\frac{k}{2}}\| \partial^k Q \|_{H^{s-k}}\\
		\leq & C\sum_{m=1}^{3}\mathcal{N}^{\frac{m}{2}}(t)\mathcal{M}(t),
\end{align*}
where we have used the Cauchy-Schwarz inequality, the Moser-type calculus inequalities and Lemma~\ref{lem-A6}.

Combining all the above estimates, we immediately get \eqref{eq5.3}.
By summing up \eqref{eq5.3} for all $k\leq |l|\leq s$ again, we infer that
\begin{equation}\label{eq5.11}
	\begin{split}
		\dfrac{1}{2}\dfrac{d}{dt}&\Big[ (1+t)^k\Big( \|\partial^ku\|_{H^{s-k}}^2 +  M\|\partial^kQ\|_{H^{s-k}}^2 +  \|\partial^k\nabla Q\|_{H^{s-k}}^2 \Big)\Big]\\
		&+ (1+t)^k\left(\dfrac{\mu}{2}\|\partial^{k}\nabla u\|_{H^{s-k}}^2 + \dfrac{aM\Gamma}{2} \|\partial^{k} Q\|_{H^{s-k}}^2+ (a+M)\Gamma\|\partial^{k}\nabla Q\|_{H^{s-k}}^2 + \Gamma\|\partial^k\Delta
		Q\|_{H^{s-k}}^2 \right) \\
		\leq & \frac{k}{2}\underbrace{(1+t)^{k-1}\Big( \|\partial^ku\|_{H^{s-k}}^2 +  M\|\partial^kQ\|_{H^{s-k}}^2 +  \|\partial^k\nabla Q\|_{H^{s-k}}^2 \Big)}_{:=\Xi_{k}} + C\sum_{m=1}^3 \mathcal{N}^\frac{m}{2}(t)\mathcal{M}(t),
	\end{split}
\end{equation}
holds for all $1\leq k\leq s$.
Moreover, let us recall that, it follows from Proposition~\ref{prop3.1} that
\begin{equation}\label{eq5.18}
	\begin{split}
		\dfrac{1}{2}\dfrac{d}{dt}\Big( \|u\|_{H^s}^2 &+  M\|Q\|_{L^2}^2 +  \|\nabla Q\|_{H^s}^2 \Big)\\
		&+ \left(\dfrac{\mu}{2}\|\nabla u\|_{H^s}^2 + \dfrac{aM\Gamma}{2} \| Q\|_{H^s}^2+ (a+M)\Gamma\|\nabla Q\|_{H^s}^2 + \Gamma\|\Delta
		Q\|_{H^s}^2 \right) \\
		\leq &  C\sum_{m=1}^3 \mathcal{N}^\frac{m}{2}(t)\mathcal{M}(t),
	\end{split}
\end{equation}
which yields a similar result as expected for the case $k=0$.

Now, we turn to proving \eqref{eq5.2}.
It is obvious that if we multiply \eqref{eq5.11}$_{s-1}$ by a sufficiently large positive constant $\widetilde{C}$ and then add \eqref{eq5.11}$_{s}$, one has
\begin{equation}\label{eq5.19}
	\begin{split}
		\dfrac{1}{2}\dfrac{d}{dt}\Big[ &\sum_{k=s-1}^s(1+t)^k\Big( \|\partial^k u\|_{H^{s-k}}^2 +  M\|\partial^k Q\|_{H^{s-k}}^2 +  \|\partial^k\nabla Q\|_{H^{s-k}}^2 \Big)\Big]\\
		&+\sum_{k=s-1}^s\left[(1+t)^k\left(\dfrac{\mu}{2}\|\partial^{k}\nabla u\|_{H^{s-k}}^2 + \dfrac{aM\Gamma}{2} \|\partial^{k} Q\|_{H^{s-k}}^2+ (a+M)\Gamma\|\partial^{k}\nabla Q\|_{H^{s-k}}^2 + \Gamma\|\partial^k\Delta
		Q\|_{H^{s-k}}^2 \right)\right]\\[2.5ex]
		\leq & \widetilde{C}\frac{s-1}{2}(1+t)^{s-2}\Big( \|\partial^{s-1}u\|_{H^1}^2 +  M\|\partial^{s-1}Q\|_{H^1}^2 +  \|\partial^{s-1}\nabla Q\|_{H^1}^2 \Big) + C\sum_{m=1}^3 \mathcal{N}^\frac{m}{2}(t)\mathcal{M}(t).
	\end{split}
\end{equation}
Here the remainder $\Xi_{s}$ on the right-hand side of \eqref{eq5.11}$_{s}$ has been absorbed by \eqref{eq5.11}$_{s-1}$.
Thus, by repeating this process for \eqref{eq5.11}$_k$ ($1\leq k\leq s-2$) and \eqref{eq5.18}, we eventually obtain that
\[
	\begin{split}
		\dfrac{1}{2}\dfrac{d}{dt}&\Big[ \sum_{k=0}^s(1+t)^k\Big( \|\partial^k u\|_{H^{s-k}}^2 +  M\|\partial^k Q\|_{H^{s-k}}^2 +  \|\partial^k\nabla Q\|_{H^{s-k}}^2 \Big)\Big]\\
		&+\sum_{k=0}^s\left[(1+t)^k\left(\dfrac{\mu}{2}\|\partial^{k}\nabla u\|_{H^{s-k}}^2 + \dfrac{aM\Gamma}{2} \|\partial^{k} Q\|_{H^{s-k}}^2+ (a+M)\Gamma\|\partial^{k}\nabla Q\|_{H^{s-k}}^2 + \Gamma\|\partial^k\Delta
		Q\|_{H^{s-k}}^2 \right)\right]\\
		\leq &  C\sum_{m=1}^3 \mathcal{N}^\frac{m}{2}(t)\mathcal{M}(t).
	\end{split}
\]
This completes the proof of Proposition~\ref{prop5.1}.
\end{proof}
\begin{remark}\label{rmk5.1}
	 Let us mention that
	 the optimal time-decay estimates established in Theorem~\ref{Thm1.3} require that $\|u\|_{L^\infty}$, $\|Q\|_{L^\infty}$, $\|\nabla Q\|_{L^\infty}$ also enjoy the optimal time-decay properties, and then one can close the global a priori estimates
	 (see Lemma~\ref{lem5.1} and Proposition~\ref{prop5.2}).
	 To this end, we need to take $s\geq 4$ in \eqref{eq511}. 
	 However, the aforementioned estimates can still be achieved for some slightly lower regularity, and we leave it to the interested readers.
\end{remark}

\subsection{$L^2$ decay estimates}
In this subsection, for any $s\geq 4$, we define
\begin{equation}\label{eq518}
	\mathcal{H}(t):=\sum_{k=0}^{s-2} (1+t)^{\frac{3}{4}+\frac{k}{2}}  \|\partial^k u(t)\|_{L^2}  + \sum_{k=0}^{s-1} (1+t)^{\frac{3}{4}+\frac{k}{2}} e^{\frac{a\Gamma t}{2}} \|\partial^k Q(t)\|_{L^2}.
\end{equation}
Obviously, the boundedness of $\mathcal{H}(t)$ yields the desired $L^2$ decay properties of the global solution stated in Theorem~\ref{Thm1.3}. Hence, the key ingredient in our analysis is to control the time-weighted quantity $\mathcal{H}(t)$. 
In addition,
it is obvious that one can combine the Gagliardo-Nirenberg inequality and \eqref{eq518}
to get the time-decay estimate of the global solution in the $L^\infty$ norm, i.e.,
\begin{equation}\label{eq511}
	\begin{split}
		\|\partial^k u(t)\|_{L^\infty}\leq& \|\partial^{k+2} u(t)\|_{L^2}^{\frac{3}{4}}\|\partial^k u(t)\|_{L^2}^{\frac{1}{4}}\leq \mathcal{H}(t)(1+t)^{-\frac{3}{2}-\frac{k}{2}},\hspace{1.53cm} \forall~0\leq k\leq s-4,\\
		\|\partial^k Q(t)\|_{L^\infty}\leq &\|\partial^{k+2} Q(t)\|_{L^2}^{\frac{3}{4}}\|\partial^k Q(t)\|_{L^2}^{\frac{1}{4}}\leq \mathcal{H}(t)(1+t)^{-\frac{3}{2}-\frac{k}{2}}e^{-\frac{a\Gamma t}{2}},\quad \forall~0\leq k\leq s-3.
	\end{split}
\end{equation}

Before we prove the main result of this subsection, we present a lemma which will be needed in the proof of Proposition~\ref{prop5.2}. In other words, we have the following time-decay estimates for the source terms defined in \eqref{eq4.2}.
\begin{lemma}\label{lem5.1}
Let $s\geq 4$ be an integer and $a>0$, $E_0<+\infty$. For all $0\leq k\leq s-1$, we have
\begin{align}\label{eq5.12}
	\|\partial^k f_1\|_{L^2\cap L^1}\leq& C  \mathcal{N}^{\frac{1}{2}}(t)\mathcal{H}(t) (1+t)^{-\frac{3}{4}-\frac{k}{2}}e^{-\frac{a\Gamma t}{2}},\\
	\|\partial^k f_2\|_{L^2\cap L^1}\leq& C  \sum_{m=1}^{2} \mathcal{N}^{\frac{m}{2}}(t)\mathcal{H}(t) (1+t)^{-\frac{3}{4}-\frac{k}{2}}e^{-\frac{a\Gamma t}{2}},\\
	\|\partial^k f_3\|_{L^2\cap L^1} \leq& C  \sum_{m=1}^{3} \mathcal{N}^{\frac{m}{2}}(t)\mathcal{H}(t) (1+t)^{-\frac{3}{4}-\frac{k}{2}},
\end{align}
where $C$ is a positive constant depending on $s$, $a$, $b$, $c$, $\mu$, $\kappa$, $\lambda$ and $\Gamma$. Moreover, it holds that
\begin{equation}\label{eq5.17}
	\|\partial^k f_1\|_{L^2\cap L^1}\lesssim \mathcal{N}^{\frac{1}{2}}(t)\mathcal{H}(t) (1+t)^{-\frac{3}{4}-\frac{k}{2}},\quad \forall~0\leq k\leq s.
\end{equation}
\end{lemma}
\begin{proof}
From
\eqref{eq518} and \eqref{eq511}, it immediately follows that
\begin{equation}\label{eq5.13}
	\begin{split}
		\|u\|_{L^\infty\cap L^2}\lesssim& (1+t)^{-\frac{3}{4}} \mathcal{H}(t),\\
		\|Q\|_{L^\infty\cap L^2}\lesssim& (1+t)^{-\frac{3}{4}}e^{-\frac{a\Gamma t}{2}} \mathcal{H}(t),\\
		\|\nabla Q\|_{L^\infty\cap L^2}\lesssim& (1+t)^{-\frac{5}{4}}e^{-\frac{a\Gamma t}{2}} \mathcal{H}(t).
	\end{split}
\end{equation}
Then, together with the Moser-type calculus inequalities and \eqref{eq5.5}, we are able to obtain
\begin{align*}
	\|\partial^k f_1\|_{L^2\cap L^1}\leq&C\|u\|_{L^\infty\cap L^2}\|\partial^k Q\|_{L^2} + C\| Q\|_{L^\infty\cap L^2}\|\partial^ku\|_{L^2}\\
	\leq& C\mathcal{N}^{\frac{1}{2}}(t)\cdot (1+t)^{-\frac{3}{4}-\frac{k}{2}}e^{-\frac{a\Gamma t}{2}}  \mathcal{H}(t) +  C(1+t)^{-\frac{3}{4}}e^{-\frac{a\Gamma t}{2}}  \mathcal{H}(t) \cdot (1+t)^{-\frac{k}{2}}\mathcal{N}^{\frac{1}{2}}(t) \\
	\leq & C  \mathcal{N}^{\frac{1}{2}}(t)\mathcal{H}(t)(1+t)^{-\frac{3}{4}-\frac{k}{2}}e^{-\frac{a\Gamma t}{2}},
\end{align*}
for all $0\leq k\leq s-1$. Moreover, one can get \eqref{eq5.17} by making some modifications, i.e.,
\[
\|\partial^k f_1\|_{L^2\cap L^1}\lesssim \left[(1+t)^{-\frac{3}{4}} \mathcal{H}(t) +  (1+t)^{-\frac{3}{4}}e^{-\frac{a\Gamma t}{2}}\mathcal{H}(t) \right] (1+t)^{-\frac{k}{2}}\mathcal{N}^{\frac{1}{2}}(t) 
\lesssim  \mathcal{N}^{\frac{1}{2}}(t)\mathcal{H}(t)(1+t)^{-\frac{3}{4}-\frac{k}{2}}.
\]
Next, using \eqref{eq5.6}-\eqref{eq5.8} and a similar argument as above, we have
\begin{align}\label{eq5.20}
	\|\partial^k f_2\|_{L^2\cap L^1}\sim & \| Q\|_{L^\infty\cap L^2}\|\partial^k\nabla u\|_{L^2} + \|\nabla u\|_{L^\infty\cap L^2}\left(\|\partial^k Q\|_{L^2} + \|\partial^k |Q|\|_{L^2}\right) \nonumber\\
	& + \| Q\|_{L^\infty\cap L^2}\|\partial^kQ\|_{L^2} +  \| Q\|_{L^\infty}\| Q\|_{L^\infty\cap L^2}\|\partial^kQ\|_{L^2}\nonumber\\
	\lesssim&  (1+t)^{-\frac{3}{4}} e^{-\frac{a\Gamma t}{2}} \mathcal{H}(t) \cdot (1+t)^{-\frac{k+1}{2}}\mathcal{N}^{\frac{1}{2}}(t)  +  (1+t)^{-\frac{1}{2}}\mathcal{N}^{\frac{1}{2}}(t) \cdot  (1+t)^{-\frac{3}{4}-\frac{k}{2}}e^{-\frac{a\Gamma t}{2}}  \mathcal{H}(t) \nonumber\\
	& +  \mathcal{N}^{\frac{1}{2}}(t)\cdot (1+t)^{-\frac{3}{4}-\frac{k}{2}}e^{-\frac{a\Gamma t}{2}}  \mathcal{H}(t) + \mathcal{N}(t)\cdot (1+t)^{-\frac{3}{4}-\frac{k}{2}}e^{-\frac{a\Gamma t}{2}}  \mathcal{H}(t) \\
	\leq & C \sum_{m=1}^{2} \mathcal{N}^{\frac{m}{2}}(t)\mathcal{H}(t)(1+t)^{-\frac{3}{4}-\frac{k}{2}}e^{-\frac{a\Gamma t}{2}},\nonumber
\end{align}
and
\begin{align*}
		\|\partial^k f_3\|_{L^2\cap L^1}\sim &\|u\|_{L^\infty\cap L^2}\|\partial^k u\|_{L^2} + \|\nabla Q\|_{L^\infty\cap L^2}\|\partial^k\nabla Q\|_{L^2} + \|Q\|_{L^\infty\cap L^2}\|\partial^k \Delta Q\|_{L^2}\\
		&+ \Big(\| \Delta Q\|_{L^\infty\cap L^2} + \| Q\|_{L^\infty\cap L^2} \sum_{m=0}^2\| Q\|_{L^\infty}^m \Big)\left(\|\partial^kQ\|_{L^2}+ \|\partial^k|Q|\|_{L^2}\right)\\
		\lesssim&  (1+t)^{-\frac{3}{4}} \mathcal{H}(t) \cdot (1+t)^{-\frac{k}{2}}\mathcal{N}^{\frac{1}{2}}(t) +  (1+t)^{-\frac{5}{4}}e^{-\frac{a\Gamma t}{2}} \mathcal{H}(t) \cdot (1+t)^{-\frac{k+1}{2}}\mathcal{N}^{\frac{1}{2}}(t)\\
		& +  (1+t)^{-\frac{3}{4}}e^{-\frac{a\Gamma t}{2}}\mathcal{H}(t)\cdot (1+t)^{-\frac{k+1}{2}}\mathcal{N}^{\frac{1}{2}}(t) +  (1+t)^{-1} \mathcal{N}^{\frac{1}{2}}(t) \cdot (1+t)^{-\frac{3}{4}-\frac{k}{2}}e^{-\frac{a\Gamma t}{2}}  \mathcal{H}(t)  \\
		& +  (1+t)^{-\frac{3}{4}} e^{-\frac{a\Gamma t}{2}}\mathcal{H}(t) \cdot (1+t)^{ -\frac{k}{2}}\mathcal{N}^{\frac{1}{2}}(t)\sum_{m=0}^2 \mathcal{N}^{\frac{m}{2}}(t)\\
		\leq & C \sum_{m=1}^{3} \mathcal{N}^{\frac{m}{2}}(t)\mathcal{H}(t) (1+t)^{-\frac{3}{4}-\frac{k}{2}}.
\end{align*}
This completes the proof.
\end{proof}
Now, we are ready to establish the $L^2$ decay estimates of the global solution to \eqref{eq1.1} by using Duhamel formulation.
To begin with, we define the following new quantities:
\begin{equation}
	\mathcal{W}(t):=\sup_{0\leq \tau\leq t} \sqrt{\mathcal{N}(\tau)}\quad \text{and}\quad \mathcal{V}(t):=\sup_{0\leq \tau\leq t} \mathcal{H}(\tau).
\end{equation}
Then, we have the following result.
\begin{prop}\label{prop5.2}
Let $s\geq 4$ be an integer and $a>0$, $\widebar{E}_0<+\infty$. It then holds that
\begin{equation}\label{eq5.33}
	\mathcal{V}(t)\leq 	C  \left( \widebar{E}_0 + \sum_{m=1}^{3} \mathcal{W}^m(t)\mathcal{V}(t) \right),
\end{equation}
where $C$ is a positive constant depending on $s$, $a$, $b$, $c$, $\mu$, $\kappa$, $\lambda$ and $\Gamma$.
\end{prop}
\begin{proof}
Let $0\leq |k| \leq s-1$.
Using the integral representation \eqref{eq5.1},
it follows from Lemma~\ref{lem4.2}, Lemma~\ref{lem5.1}, Lemma~\ref{lem-a2} and Theorem~\ref{thm4.1} that
\begin{align}\label{eq5.34}
		\|\partial^k Q(t)\|_{L^2}\lesssim & \|\partial^kQ_L(t)\|_{L^2} + \left\| \int_0^{\frac{t}{2}}\partial^k \check{\mathcal{A}}(t-s) \ast f_2(s) ds\right\|_{L^2} + \left\| \int_0^{\frac{t}{2}}\partial^{k+1} \check{\mathcal{A}}(t-s) \ast f_1(s) ds\right\|_{L^2} \nonumber\\
		& + \left\|\int_{\frac{t}{2}}^t \check{\mathcal{A}}(t-s) \ast \partial^kf_2(s) ds \right\|_{L^2} + \left\|\int_{\frac{t}{2}}^t \nabla\check{\mathcal{A}}(t-s) \ast \partial^kf_1(s) ds \right\|_{L^2}\nonumber\\
		\lesssim  & \widebar{E}_0 (1+t)^{-\frac{3}{4}-\frac{k}{2}}e^{-a\Gamma t} + \int_0^{\frac{t}{2}}  e^{-a\Gamma(t-s)} (1+t-s)^{-\frac{3}{4}-\frac{k}{2}} \| f_2(s) \|_{H^k \cap L^1} ds\nonumber\\
		& + \int_0^{\frac{t}{2}} e^{-a\Gamma(t-s)} (1+t-s)^{-\frac{5}{4}-\frac{k}{2}}  \| f_1(s) \|_{H^{k+1} \cap L^1}  ds\nonumber\\
		& +\int_{\frac{t}{2}}^t e^{-a\Gamma(t-s)}(1+t-s)^{-\frac{3}{4}} \| \partial^k f_2(s) \|_{L^2 \cap L^1}  ds\nonumber\\
		&+ \int_{\frac{t}{2}}^t e^{-a\Gamma(t-s)}\left( (1+t-s)^{-\frac{5}{4}}\| \partial^{k} f_1(s) \|_{L^1} + \frac{1}{\sqrt{t-s}}\| \partial^{k} f_1(s) \|_{L^2} \right) ds \\
		\lesssim  & \widebar{E}_0 (1+t)^{-\frac{3}{4}-\frac{k}{2}}e^{-\frac{a\Gamma t}{2} } + \sum_{m=1}^{2} \mathcal{W}^m(t)\mathcal{V}(t)\int_0^{\frac{t}{2}}  e^{-\frac{a\Gamma t}{2}} (1+t-s)^{-\frac{5}{4}-\frac{k}{2}} (1+s)^{-\frac{3}{4}}  ds\nonumber\\
		& + e^{-\frac{a\Gamma t}{2}} \mathcal{W}(t)\mathcal{V}(t)\int_0^{\frac{t}{2}}  (1+t-s)^{-\frac{5}{4}-\frac{k}{2}} (1+s)^{-\frac{3}{4}}  ds\nonumber\\
		& +\sum_{m=1}^{2} \mathcal{W}^m(t)\mathcal{V}(t)\int_{\frac{t}{2}}^t e^{-\frac{a\Gamma t}{2}} (1+t-s)^{-\frac{5}{4}} (1+s)^{-\frac{3}{4}-\frac{k}{2}}  ds\nonumber\\
		&+ e^{-\frac{a\Gamma t}{2}}\mathcal{W}(t)\mathcal{V}(t)\int_{\frac{t}{2}}^t e^{-\frac{a\Gamma(t-s)}{2}} \frac{1}{\sqrt{t-s}} (1+s)^{-\frac{3}{4}-\frac{k}{2}}   ds\nonumber\\
		\lesssim& \left(\widebar{E}_0 +\sum_{m=1}^{2} \mathcal{W}^m(t)\mathcal{V}(t) \right) (1+t)^{-\frac{3}{4}-\frac{k}{2}}  e^{-\frac{a\Gamma t}{2}}  + \mathcal{W}(t)\mathcal{V}(t)  (1+t)^{-\frac{3}{4}}e^{-\frac{a\Gamma t}{2}}\Gamma\left(\frac{1}{2}\right)\nonumber\\
		\leq &C (1+t)^{-\frac{3}{4}-\frac{k}{2}}e^{-\frac{a\Gamma t}{2}} \left( \widebar{E}_0+\sum_{m=1}^{2} \mathcal{W}^m(t)\mathcal{V}(t) \right).\nonumber
\end{align}
Again, via a similar argument, we apply $\partial^k$ $(0\leq |k|\leq s-2)$ to \eqref{eq-5.2} and get
\begin{align}
	\|\partial^k u(t)\|_{L^2}\lesssim & \|\partial^ku_L(t)\|_{L^2} + \left\| \int_0^{\frac{t}{2}}\partial^k \check{\mathcal{B}}(t-s)\ast \nabla^2 f_1(s)  ds\right\|_{L^2} + \left\| \int_0^{\frac{t}{2}}\partial^k \check{\mathcal{B}}(t-s)\ast \nabla f_2(s)  ds\right\|_{L^2} \nonumber\\
	&+ \left\| \int_0^{\frac{t}{2}}\partial^k   \check{\mathcal{C}}(t-s) \ast\nabla f_3(s) ds\right\|_{L^2} + \left\| \int_{\frac{t}{2}}^t \check{\mathcal{B}}(t-s)\ast \partial^k\nabla^2 f_1(s)  ds\right\|_{L^2} \nonumber\\
	& + \left\| \int_{\frac{t}{2}}^t \check{\mathcal{B}}(t-s)\ast \partial^k\nabla f_2(s)  ds\right\|_{L^2} + \left\| \int_{\frac{t}{2}}^t   \check{\mathcal{C}}(t-s) \ast \partial^k\nabla f_3(s) ds\right\|_{L^2}\nonumber\\
	\lesssim & \widebar{E}_0 (1+t)^{-\frac{3}{4}-\frac{k}{2}} + \int_0^{\frac{t}{2}} (1+t-s)^{-\frac{3}{4}-\frac{k}{2}} \| \nabla^2 f_1(s) \|_{H^{k} \cap L^1}  ds\nonumber\\
	&+   \int_0^{\frac{t}{2}} (1+t-s)^{-\frac{3}{4}-\frac{k}{2}} \|  \nabla f_2(s) \|_{H^{k} \cap L^1}  ds +   \int_0^{\frac{t}{2}} (1+t-s)^{-\frac{3}{4}-\frac{k}{2}} \|  \nabla f_3(s) \|_{H^{k} \cap L^1}  ds\nonumber\\
	& + \int_{\frac{t}{2}}^t (1+t-s)^{-\frac{3}{4}} \| \partial^k \nabla^2 f_1(s) \|_{L^2 \cap L^1}  ds + \int_{\frac{t}{2}}^t (1+t-s)^{-\frac{3}{4}} \| \partial^k \nabla f_2(s) \|_{L^2 \cap L^1}  ds \nonumber\\
	& + \int_{\frac{t}{2}}^t (1+t-s)^{-\frac{3}{4}} \| \partial^k \nabla f_3(s) \|_{L^2 \cap L^1}  ds \label{eq5.35}\\
	\lesssim & \widebar{E}_0 (1+t)^{-\frac{3}{4}-\frac{k}{2}} +  \mathcal{W}(t)\mathcal{V}(t) \int_0^{\frac{t}{2}} (1+t-s)^{-\frac{3}{4}-\frac{k}{2}} (1+s)^{-\frac{7}{4}} ds  \nonumber\\
	&+  \sum_{m=1}^{3} \mathcal{W}^m(t)\mathcal{V}(t)   \int_0^{\frac{t}{2}} (1+t-s)^{-\frac{3}{4}-\frac{k}{2}} (1+s)^{-\frac{5}{4}} ds\nonumber\\
	&+  \mathcal{W}(t)\mathcal{V}(t)   \int_{\frac{t}{2}}^t (1+t-s)^{-\frac{3}{4}} (1+s)^{-\frac{7}{4}-\frac{k}{2}} ds \nonumber\\
	&+ \sum_{m=1}^{3} \mathcal{W}^m(t)\mathcal{V}(t)  \int_{\frac{t}{2}}^t (1+t-s)^{-\frac{3}{4}} (1+s)^{-\frac{5}{4}-\frac{k}{2}} ds \nonumber\\
	\leq& C (1+t)^{-\frac{3}{4}-\frac{k}{2}} \left( \widebar{E}_0 + \sum_{m=1}^{3} \mathcal{W}^m(t)\mathcal{V}(t) \right).\nonumber
\end{align}
Note that we excluded the Leray projector $\mathbb{P}$ for simplicity, as it is
bounded in $L^2$.

Finally, by summing up \eqref{eq5.34} and \eqref{eq5.35} for all $k$, we get \eqref{eq5.33} and this completes the proof.
\end{proof}

\subsection{Proof of Theorem~\ref{Thm1.3}}
Now, we are in a position to present the proof of Theorem~\ref{Thm1.3}.
To achieve this, we make the a priori assumption that
\[
\mathcal{N}(t)\leq \delta^2 \ll 1,\quad \forall \,t\in [0,T],
\]
where time $T>0$ is fixed.
It then follows from \eqref{eq5.2} that
\[
	\mathcal{N}(t)\leq E_0,\quad \forall \,t\in [0,T].
\]
Thus, we can conclude that $\sup_{0\leq t\leq T}\mathcal{N}(t)\leq \delta^2$ for some sufficiently small $E_0$, and hence
\[
\mathcal{W}(T)\leq \delta.
\]
Together with \eqref{eq5.33}, we see that
\[
\mathcal{V}(T)\leq 	C \widebar{E}_0.
\]
Since $T$ is ‌arbitrary, the above results yield \eqref{eq:1.11}-\eqref{eq:1.12} hold for all $t>0$. By \eqref{eq511}, we immediately get \eqref{eq:1.13}.
Moreover, the $L^p$ decay rates \eqref{eq:1.14} follows from the following Gagliardo-Nirenberg inequality
\begin{equation}
	\|\partial_x^kf\|_{L^p} \leq C\|\partial_x^mf\|_{L^2}^{\theta
	}\|f\|_{L^2}^{1-\theta},\quad 0\leq k\leq m,\, 2\leq p\leq \infty,\, \theta=\frac{1}{m}\left(k+\dfrac{3}{2}- \dfrac{3}{p}\right)\leq 1.
\end{equation}
This completes the proof of Theorem~\ref{Thm1.3}.

\begin{remark}\label{rk5.2}
In this remark, we indicate how to modify the argument to achieve the optimal stable mixing estimate of the $Q$-tensor (in the sense of comparison with the associated linearized system) by assuming that $\|(Q,u)\|_{L^1(\R^3)}$ is sufficiently small.
	
	First, in order to characterize the desired decay properties of $Q$-tensor, we need to replace \eqref{eq518} by
	\begin{equation}
		\widetilde{\mathcal{H}}(t):=\sum_{k=0}^{s-2} (1+t)^{\frac{3}{4}+\frac{k}{2}}  \|\partial^k u(t)\|_{L^2}  + \sum_{k=0}^{s-1} (1+t)^{\frac{3}{4}+\frac{k}{2}} e^{a\Gamma t} \|\partial^k Q(t)\|_{L^2}.
	\end{equation}
	Note that it is easy to see that the following two terms in \eqref{eq5.20} can also be estimated by
	\begin{equation}
		\| Q\|_{L^\infty\cap L^2}\|\partial^kQ\|_{L^2} +  \| Q\|_{L^\infty}\| Q\|_{L^\infty\cap L^2}\|\partial^kQ\|_{L^2} \lesssim (1+t)^{-\frac{3}{4}-\frac{k}{2}}e^{-a\Gamma t}\left(\widetilde{\mathcal{H}}^2(t)+ \mathcal{N}^{\frac{1}{2}}(t)\widetilde{\mathcal{H}}^2(t)\right).
	\end{equation}
	Then, arguing as in the proof of Lemma~\ref{lem5.1}, the following improved energy estimates hold for all $k\leq s-1$:
	\begin{align}
		\|\partial^k f_1\|_{L^2\cap L^1}\leq& C  \mathcal{N}^{\frac{1}{2}}(t)\widetilde{\mathcal{H}}(t) (1+t)^{-\frac{3}{4}-\frac{k}{2}}e^{-a\Gamma t},\\
		\|\partial^k f_2\|_{L^2\cap L^1}\leq& C \left( \widetilde{\mathcal{H}}^2(t) + \sum_{m=1}^{2} \mathcal{N}^{\frac{1}{2}}(t)\widetilde{\mathcal{H}}^m(t)\right) (1+t)^{-\frac{5}{4}-\frac{k}{2}}e^{-a\Gamma t}.
	\end{align}
	Combining these results, one can simply repeat the argument presented above to conclude that
	\begin{equation}\label{eq-5.30}
		\widetilde{\mathcal{V}}(t)\lesssim   \widebar{E}_0 + \sum_{m=1}^{3} \mathcal{W}^m(t)\widetilde{\mathcal{V}}(t) +  \left(1+ \mathcal{W}(t)\right)\widetilde{\mathcal{V}}^2(t),\quad \text{\rm with}\,\, \widetilde{\mathcal{V}}(t):=\sup_{0\leq \tau\leq t} \widetilde{\mathcal{H}}(\tau).
	\end{equation}
	Now, with the smallness of $\widebar{E}_0$, we claim that
	\begin{equation}
		\|\partial^kQ(t) \|_{L^2}  \lesssim (1+t)^{-\frac{3}{4}-\frac{k}{2}}e^{-a\Gamma t}, \quad \text{\rm for all $k\leq s-1$}.
	\end{equation}
	We point out that the structure of estimate \eqref{eq-5.30} is similar to what we proved for the viscid 1D models of blood flow \cite{Yang-blood}, where the smallness of initial data in the $L^1$ norm is also required in the $H^s$-framework.
\end{remark}

\subsection{Sharp decay estimate for velocity}\label{sec5.4}
At the end, we give the proof of
Theorem~\ref{Thm1.4}.
%
As shown by Theorem~\ref{Thm1.3}, the upper bound for the decay rate of $\partial^k u(t)$ is
\[
\|\partial^ku(t) \|_{L^2(\R^3)}  \lesssim (1+t)^{-\frac{3}{4}-\frac{k}{2}},
\]
for any $t>0$ and $k\leq s-2$ with $s\geq 4$.
So we only need to show the lower bound for $\partial^k u(t)$ under the additional assumption \eqref{eq1.20}.
More precisely, we have the following result.
\begin{prop}\label{prop5.3}
	Let the assumptions of Theorem~\ref{Thm1.3} and \eqref{eq1.20} be satisfied.
	Then, the following inequality holds for all $0\leq k\leq s-2:$
	\begin{equation}\label{eq5.26}
		\left\|\partial^k u(t) \right\|_{L^2(\R^3)}
		\geq  C (1+t)^{-\frac{3}{4}-\frac{k}{2}},
	\end{equation}
	where $C$ is a positive constant depending on $\widebar{E}_0$, $\|u_0\|_{L^1}$, $\mu$, $a$, $\Gamma$, $\kappa$.
\end{prop}
\begin{proof}
Recall that $u_0(x)\in L^1$, it follows from the Riemann-Lebesgue lemma that $\widehat{u_0}(\xi)\in C_c(\R^3)$. Together with the assumption \eqref{eq1.20} (which implies that $\widehat{u_0}(0)\ne 0$), we have
\begin{equation}
	\widehat{u_0}(\xi)\ne 0 \quad \text{for} \quad |\xi|\leq \epsilon,
\end{equation}
where $\epsilon$ is a fixed small positive constant.
Then, there exists a constant $r\in (0,\epsilon)$ such that we conclude that
\begin{equation}
	\inf_{\xi\in\, \mathcal{U}} |\widehat{u_0}(\xi)| \geq C_0>0,
\end{equation}
where $\mathcal{U}:=B_{r}(0) \subseteq \{ \xi\in \R^3\mid |\xi|\leq  \frac{R}{2}\}$
and $R$ is chosen to be the same as in Lemma~\ref{lem4.2}. Note that the constant $C_0$ depends only on $\|u_0\|_{L^1}$.
To derive the decay of $\partial^ku(t)$, we rewrite \eqref{eq-5.2} as
\[
u(t) =\check{\mathcal{C}}(t)\ast u_0 + \check{\mathcal{B}}(t)\ast \sigma_0  + \int_0^t  \check{\mathcal{B}}(t-s)\ast \mathbb{P}{\rm  div}\left( f_2(s)-{\rm div}f_1(s) \right) ds -\int_0^t  \check{\mathcal{C}}(t-s) \ast \mathbb{P}{\rm  div}f_3(s)  ds.
\]
Applying $\partial^k$ $(0\leq |k|\leq s-2)$ to the above equation, it follows that‌
	\begin{align}\label{eq5.29}
		\left\|\partial^k u(t)\right\|_{L^2(\R^3)}^2
		\geq& \frac{1}{2}\| |\xi|^k \mathcal{C}(t)\widehat{u_0} \|_{L^2(\mathcal{U})}^2 \nonumber\\
		&-\left\| |\xi|^k \mathcal{B}(t)\widehat{\sigma_0} + \int_0^t |\xi|^k \mathcal{B}(t-s)\left(|\xi| \widehat{f_2}(s)- |\xi|^2 \widehat{f_1}(s)  \right) - |\xi|^k\mathcal{C}(t-s)|\xi| \widehat{f_3}(s) ds\right\|_{L^2(\mathcal{U})}^2 \nonumber\\
		\geq& \frac{1}{2}\| |\xi|^k \mathcal{C}(t)\widehat{u_0} \|_{L^2(\mathcal{U})}^2- \| |\xi|^k \mathcal{B}(t)\widehat{\sigma_0} \|_{L^2(\mathcal{U})}^2\\
		&-\underbrace{\int_0^t\left\||\xi|^k \mathcal{B}(t-s)\left(|\xi| \widehat{f_2}(s)- |\xi|^2 \widehat{f_1}(s)  \right) + |\xi|^k\mathcal{C}(t-s)|\xi| \widehat{f_3}(s) \right\|_{L^2(\mathcal{U})}^2 ds}_{\Upsilon}.\nonumber
	\end{align}
where we have used the Plancherel theorem and the basic inequality $(a+b)^2\geq \frac{1}{2}a^2-b^2$.

	For the first term on the right-hand side of \eqref{eq5.29}, one can see that
    \begin{equation}\label{eq5.30}
    	\begin{split}
    		\| |\xi|^k \mathcal{C}(t,\xi)\widehat{u_0} \|_{L^2(\mathcal{U})}^2
    		= &  \int_{\mathcal{U}} |\xi|^{2k} |\mathcal{C}(t,\xi)|^2|\widehat{u_0}|^2 d\xi \\
    		\geq & \inf_{\xi\in\, \mathcal{U}} |\widehat{u_0}(\xi)|^2 \int_{|\xi|\leq r} |\xi|^{2k} e^{-2\mu|\xi|^2t} d\xi \\
    		\geq & C_0^2 \int_{|\varsigma|< r\sqrt{t}} t^{-\frac{3}{2}-k} |\varsigma|^{2k} e^{-2\mu|\varsigma|^2 } d\varsigma\\
    		\geq & C_1 (1+t)^{-\frac{3}{2}-k},
    	\end{split}
    \end{equation}
    for $t>0$. On the other hand,
    we only estimate the remainder terms for the case of $\mu>\Gamma$ and $a> 0$, as the proofs are analogous. Indeed, by a similar argument as in Lemma~\ref{lem4.2}, we have
\begin{equation}\label{eq5.31}
	\begin{split}
		\| |\xi|^k \mathcal{B}(t,\xi)\widehat{\sigma_0} \|_{L^2(\mathcal{U})}^2
		\lesssim &  \int_{|\xi|\leq r} |\xi|^{2k+2} |\mathcal{B}(t,\xi)|^2|\widehat{Q_0}|^2 d\xi \\
		\lesssim & \|\widehat{Q_0}\|_{L^\infty}^2\int_{|\xi|\leq r} |\xi|^{2(k+1)}e^{-2\min\{\mu,\Gamma\}|\xi|^2t} d\xi \\
		\leq & C (1+t)^{-\frac{5}{2}-k} \|Q_0\|_{L^1}^2,
	\end{split}
\end{equation}
and
    \begin{align}\label{eq5.32}
    	\Upsilon\lesssim&\int_0^t\left\||\xi|^k \mathcal{B}(t-s)\left(|\xi| \widehat{f_2}(s)- |\xi|^2 \widehat{f_1}(s)  \right) + |\xi|^k\mathcal{C}(t-s)|\xi| \widehat{f_3}(s) \right\|_{L^2(\mathcal{U})}^2 ds\nonumber\\
    	\lesssim &\int_0^{\frac{t}{2}} \left(\left\|\hat{f_2}(s,\xi)\right\|_{L^\infty}^2 + \left\|\hat{f_3}(s,\xi)\right\|_{L^\infty}^2\right) \int_{|\xi|\leq r} |\xi|^{2(k+1)} e^{-2c_{\mu,\Gamma} |\xi|^2 (t-s)} d\xi  ds\nonumber\\
    	& + \int_{\frac{t}{2}}^t \left(\left\||\xi|^{k+1}\hat{f_2}(s,\xi)\right\|_{L^\infty}^2 +\left\||\xi|^{k+1}\hat{f_3}(s,\xi)\right\|_{L^\infty}^2 \right) \int_{ |\xi|\leq r} e^{-2c_{\mu,\Gamma} |\xi|^2 (t-s)} d\xi  ds\nonumber\\
    	&+\int_0^{\frac{t}{2}} \left( \left\||\xi|\hat{f_1}(s,\xi)\right\|_{L^\infty}^2\right) \int_{|\xi|\leq r} |\xi|^{2(k+1)} e^{-2c_{\mu,\Gamma} |\xi|^2 (t-s)} d\xi  ds\nonumber\\
    	& + \int_{\frac{t}{2}}^t  \left\||\xi|^{k+1}\hat{f_1}(s,\xi)\right\|_{L^\infty}^2 \int_{|\xi|\leq r} |\xi|^2 e^{-2c_{\mu,\Gamma} |\xi|^2 (t-s)} d\xi  ds \\
    	\lesssim &\int_0^{\frac{t}{2}} (1+t-s)^{-\frac{5}{2}-k}  \left(\|\nabla f_1(s)\|_{L^1}^2 +\|f_2(s)\|_{L^1}^2 + \|f_3(s)\|_{L^1}^2\right) ds \nonumber\\
    	& + \int_{\frac{t}{2}}^t (1+t-s)^{-\frac{3}{2}}  \left(\|\partial^{k+1} f_1(s)\|_{L^1}^2 +\|\partial^{k+1}f_2(s)\|_{L^1}^2 + \|\partial^{k+1}f_3(s)\|_{L^1}^2\right) ds\nonumber \\
    	\lesssim& \sum_{m=1}^3 E_0^{m}\widebar{E}_0^2 \left(\int_0^{\frac{t}{2}} (1+t-s)^{-\frac{5}{2}-k} (1+s)^{-\frac{3}{2}}ds + \int_{\frac{t}{2}}^t (1+t-s)^{-\frac{3}{2}} (1+s)^{-\frac{5}{2}-k}ds \right)\nonumber\\
    	\leq& C\sum_{m=1}^3 E_0^{m}\widebar{E}_0^2 (1+t)^{-\frac{5}{2}-k},\nonumber
    \end{align}
    where we have also used Lemma~\ref{lem5.1}. Then, by substituting \eqref{eq5.30}-\eqref{eq5.32} into \eqref{eq5.29}, we arrive at
	\begin{equation}
		\left\|\partial^k u(t)\right\|_{L^2(\R^3)}^2 \geq \frac{C_1}{2} (1+t)^{-\frac{3}{2}-k} - C\sum_{m=0}^3 E_0^{m}\widebar{E}_0^2 (1+t)^{-\frac{5}{2}-k},
	\end{equation}
	for $t>0$ and $k\leq s-2$.
	This, together with the smallness assumption on initial data, then completes the proof of \eqref{eq5.26} regarding lower bounds for decay rates of the velocity field.
\end{proof}

\noindent\textbf{The proof of Theorem~\ref{Thm1.4}.}
It follows directly from Proposition~\ref{prop5.3} and Theorem~\ref{Thm1.3}.$\hfill\qedsymbol$



%

\section*{Acknowledgments}

The work of F. Yang was supported in part by NSFC-12031012, NSFC-12250710674, NSFC-W2531006 and the Institute of Modern Analysis-A Frontier Research Center of Shanghai and the Scientific Research Foundation for Advanced Talent of Nantong University under Grant 135424639079 and the Jiangsu Innovation and Entrepreneurship Talent Program under Grant JSSCBS20250352. The work of X. F. Yang was supported in part by National key research and development program of China (~2024YFA1013303~), National Natural Science Foundation of China (~No. 12271356~) and the Fundamental Research Funds for the Central Universities.


\begin{appendices}
	
	\section{Some basic theories and lemmas}\label{AP}
	In this section, we begin with some key inequalities that are used in this paper.
	First, the well-known Moser-type calculus inequalities are stated as follows (see \cite{ms-eq,majda_local} for details):
	\begin{lemma}\label{lem-a1}
		Let $f$, $g\in H^s$, we then have
		\begin{align}
			\|\partial^k\left(fg\right)\|_{L^2}\lesssim&  \|f\|_{L^\infty}\|\partial^k g\|_{L^2}  +  \|g\|_{L^\infty}\|\partial^k f\|_{L^2}, \\
			\|\partial^k\left(f g\right)\|_{L^1}\lesssim&  \|f\|_{L^2}\|\partial^k g\|_{L^2}  +  \|g\|_{L^2}\|\partial^k f\|_{L^2},\\
			\|[\partial^k,f] g\|_{L^2}\lesssim& \|\nabla f\|_{L^\infty}\|\partial^{k-1}g\|_{L^2}  +  \|g\|_{L^\infty}\|\partial^k f\|_{L^2},
		\end{align}
		for all $|k|\leq s$.
	\end{lemma}
	In addition, we include the following lemma for the convenience of the readers; see \cite{Liu1987}.
	\begin{lemma}\label{lem-a2}
		Let $\alpha$, $\beta$, $\gamma$ be positive constants. Then
		\begin{enumerate}[(i).]
			\item If $\alpha\leq \min\{\beta, \beta+\gamma-1\}$, $\gamma\ne 1$ or if $\alpha< \beta$, $\alpha \leq \beta+\gamma-1$, $\gamma= 1$, then
			\begin{equation}
				\int_0^{\frac{t}{2}} (1+t-s)^{-\beta}(1+s)^{-\gamma}ds \leq C(1+t)^{-\alpha},
			\end{equation}
			\item If $\alpha\leq \min\{\gamma, \beta+\gamma-1\}$, $\beta\ne 1$ or if $\alpha< \gamma$, $\alpha \leq \beta+\gamma-1$, $\beta= 1$, then
			\begin{equation}
				\int_{\frac{t}{2}}^t (1+t-s)^{-\beta}(1+s)^{-\gamma}ds \leq C(1+t)^{-\alpha}.
			\end{equation}
		\end{enumerate}
	\end{lemma}
	Next, let us recall the well-known Gagliardo-Nirenberg inequality.
	\begin{lemma}\label{lem.GN}
		Let $1 \leq q, r \leq  \infty$ and $m$ is a natural number. Suppose that there is a real number $\alpha$ and a natural number $j<m$ are such that
		\[
		\frac{1}{p} = \frac{j}{d} + \left( \frac{1}{r} - \frac{m}{d} \right) \alpha + \frac{1 - \alpha}{q}
		\]
		and
		\[
		\frac{j}{m} \leq \alpha \leq 1.
		\]
		In particular, under the folowing two exceptional cases:
		\begin{enumerate}[(i).]
			\item  If $j = 0, mr < d$ and $q = \infty$, then it is necessary to make the additional assumption that either $u$ tends to zero at infinity or that $u$ lies in $L^s$ for some finite $s > 0$.
			\item If $1 < r <\infty$ and $m-j- \frac{d}{r}$ is a non-negative integer, then it is necessary to assume also that $\alpha\ne 1$.
		\end{enumerate}
		Then it holds that
		\begin{equation}
			\| \mathrm{D}^{j} u \|_{L^{p}} \leq C \| \mathrm{D}^{m} u \|_{L^{r}}^{\alpha} \| u \|_{L^{q}}^{1 - \alpha},
		\end{equation}
		where $C$ is a constant depending only on $d$, $m$, $j$, $q$, $r$ and $\alpha$.
	\end{lemma}
	Note that we also need the following Lemma regarding the cancellations rules of system \eqref{eq1.1}; see \cite{weak-ALC,Qtensor12,weak-strong-3d} and references therein.
	\begin{lemma}\label{lem.cancel}
		Let $u\in L_\sigma^2(\R^3)$ and $Q, G\in H^2(\R^3)\cap S_0^3$, and let
		\[
		D=\frac{1}{2}\left( \nabla u + \nabla u^{T} \right)\quad{\rm and}\quad\Omega=\frac{1}{2}\left( \nabla u - \nabla u^{T} \right).
		\]
		Then we have
		\begin{enumerate}[(1).]
			\item $(u\cdot\nabla u,u)= (u \cdot\nabla Q,Q)=0;$
			\item $(Q \Omega-\Omega Q,Q)=0;$
			\item $(u\cdot\nabla Q,\Delta Q)-\left( \nabla\cdot (\nabla Q \odot \nabla Q), u \right)=0 ;$
			\item $(G \Omega - \Omega G,\Delta Q) -\left( G \Delta Q - \Delta Q G, \nabla u \right) =0;$
			\item $(| G| \Delta Q, \nabla u)-(|G|D, \Delta Q)=0.$
		\end{enumerate}
	\end{lemma}
	Finally, we close this section with a lemma for the estimate of $|Q|$ in the $H^s$ norm. In other words, we aim to obtain a result similar to that in \cite{active-limit}. To do this, we begin with a short review of the Littlewood-Paley theory; see \cite[Chapter~2]{Danchin-book} for details.
	\begin{prop}[Dyadic partition of unity]
		Let $\mathcal{C}$ be annulus $\{ \xi\in\R^d:\frac{3}{4}\leq |\xi|\leq \frac{8}{3} \}$. There exist radial functions $\chi$ and $\psi$, valued in interval $[0,1]$, belonging to $\mathcal{D}(B(0,\frac{4}{3}))$ and $\mathcal{D}(\mathcal{C})$, respectively, such that
		\[
		\begin{split}
			&\chi(\xi) +\sum_{j\geq 1}\psi(2^{-j}\xi)=1, \quad \forall \xi\in\R^d,\\
			&\sum_{j\in \mathbb{Z}}\psi(2^{-j}\xi)=1, \quad \forall \xi\in\R^d~\backslash\{0\},\\
			&|j-j'|\geq 2 \Longrightarrow \supp(\psi(2^{-j}\cdot))\cap \supp(\psi(2^{-j'}\cdot))=\emptyset,\\
			&j\geq 1 \Longrightarrow \supp(\chi)\cap\supp(\psi(2^{-j}\cdot))= \emptyset.
		\end{split}
		\]
	\end{prop}
	Then for all $j\in \mathbb{Z}$, the homogeneous dyadic blocks $\Delta_j$ and the homogeneous low-frequency cut-off operators $S_j$ are defined by
	\[
	\begin{split}
		\dot{\Delta}_ju=&\mathcal{F}^{-1}( \psi(2^{-j}\xi)\mathcal{F}u ) = 2^{jd}\int_{\R^d}h(2^jy)u(x-y)dy,\\
		\dot{S}_ju=&\mathcal{F}^{-1}( \chi(2^{-j}\xi)\mathcal{F}u ) = 2^{jd}\int_{\R^d}\tilde{h}(2^jy)u(x-y)dy,
	\end{split}
	\]
	where
	\[
	h:=\mathcal{F}^{-1}\psi \quad{\rm and}\quad \tilde{h}:=\mathcal{F}^{-1}\chi.
	\]
	For the homogeneous Besov space $\dot{B}_{p,q}^s$, the associated norm is defined by
	\[
	\|u\|_{\dot{B}_{p,q}^s}:= \left( \sum_{j\in \mathbb{Z}} 2^{qjs}\| \dot{\Delta}_j u \|_{L^p}^q   \right)^{\frac{1}{q}}
	\]
	Moreover, we need to introduce the following Bernstein-type lemma that we use extensively.
	\begin{lemma}
		Let $0<r<R$, $1\leq p\leq q\leq \infty$, and $k\in \mathbb{N}$. For any function $u\in L^p$ and $\lambda>0$, it holds that
		\[
		\begin{split}
			& \supp(\mathcal{F}u)\subset\{ \xi\in\R^d\mid |\xi|\leq\lambda R \} \Longrightarrow \|\partial^k u\|_{L^q} \lesssim \lambda^{k+d\left( \frac{1}{p}- \frac{1}{q} \right)}\|u\|_{L^p},\\
			& \supp(\mathcal{F}u)\subset\{ \xi\in\R^d\mid \lambda r\leq |\xi|\leq\lambda R \} \Longrightarrow \|\partial^k u\|_{L^p} \sim \lambda^k\|u\|_{L^p}.
		\end{split}
		\]
	\end{lemma}
	Now, we are ready to state our key Lemma as follows:
	\begin{lemma}\label{lem-A6}
		Let $s\geq 1$ be an integer. If $Q\in H^s(\R^d)$, then $|Q|:=\sqrt{\sum_{\alpha,\beta=1}^d Q_{\alpha\beta}^2}\in H^s(\R^d)$ and
		\begin{equation}
			\| |Q| \|_{H^s} \leq C \| Q \|_{H^s},
		\end{equation}
		where $C$ is a positive constant depending only on $s$ and $d$.
	\end{lemma}
	\begin{proof}
		Similarly to the proof of \cite[Lemma~A.1]{active-limit},
		we denote
		\[
		\mathcal{Q} := \sqrt{\sum_{\alpha=1}^d\sum_{\beta=1}^d Q_{\alpha\beta}^2 +\epsilon^2} -\epsilon.
		\]
		Obviously, for any $\epsilon>0$, we have $\mathcal{Q}\geq 0$ and
		\begin{equation}\label{eq-a8}
		\left|	\dfrac{\partial \mathcal{Q}}{\partial Q_{\alpha\beta}}\right|=\left| \dfrac{Q_{\alpha\beta}}{\sqrt{\sum_{\alpha=1}^d\sum_{\beta=1}^d Q_{\alpha\beta}^2 +\epsilon^2}}\right|<1,\quad \alpha,\beta\in \{ 1,2,\cdots,d \}.
		\end{equation}
		Since the norms $\|\cdot\|_{\dot{H}^s}$ and $\|\cdot\|_{\dot{B}_{2,2}^s}$ are equivalent.
		Together with the Bernstein-type inequalities, one has
		\begin{align*}
			\|\mathcal{Q}\|_{\dot{H}^s}^2\leq C \|\mathcal{Q}\|_{\dot{B}_{2,2}^s}^2=& C\sum_{j\in \mathbb{Z}} 2^{2js} \|\dot{\Delta}_j \mathcal{Q}\|_{L^2}^2
			\leq   C\sum_{j\in \mathbb{Z}} 2^{2j(s-1)} \|\nabla \dot{\Delta}_j \mathcal{Q}\|_{L^2}^2\\
			\leq & C\sum_{j\in \mathbb{Z}}\sum_{\alpha=1}^d\sum_{\beta=1}^d 2^{2j(s-1)} \left\|  h_j \ast \left(\dfrac{\partial \mathcal{Q}}{\partial Q_{\alpha\beta}} \nabla Q_{\alpha,\beta} \right)\right\|_{L^2}^2\\
			= & C\sum_{j\in \mathbb{Z}} \sum_{\alpha=1}^d\sum_{\beta=1}^d 2^{2j(s-1)} \int_{\R^d} \left|\int_{\R^d} h_j (x-y) \dfrac{\partial \mathcal{Q}}{\partial Q_{\alpha\beta}}(y)\nabla Q_{\alpha,\beta}(y)  dy\right|^2 dx\\
			\leq  & C\sum_{j\in \mathbb{Z}} \sum_{\alpha=1}^d\sum_{\beta=1}^d 2^{2j(s-1)} \int_{\R^d} \left|\int_{\R^d} h_j (x-y) \nabla Q_{\alpha,\beta}(y)  dy\right|^2 dx\\
			\leq & C \sum_{j\in \mathbb{Z}} \sum_{\alpha=1}^d\sum_{\beta=1}^d 2^{2js} \left\| h_j \ast  Q_{\alpha,\beta}\right\|_{L^2}^2 \\
			\leq & C \sum_{\alpha=1}^d\sum_{\beta=1}^d \|Q_{\alpha\beta}\|_{\dot{H}^s}^2.
		\end{align*}
		On the other hand, one can see that
		\begin{equation}
			\mathcal{Q} = \sqrt{\sum_{\alpha=1}^d\sum_{\beta=1}^d Q_{\alpha\beta}^2 +\epsilon^2} -\epsilon = \dfrac{\sum_{\alpha=1}^d\sum_{\beta=1}^d Q_{\alpha\beta}^2}{\sqrt{\sum_{\alpha=1}^d\sum_{\beta=1}^d Q_{\alpha\beta}^2 +\epsilon^2} + \epsilon} < \sum_{\alpha=1}^d\sum_{\beta=1}^d |Q_{\alpha\beta}|.
		\end{equation}
		Combining these results, we obtain that
		\begin{equation}
			\|\mathcal{Q}\|_{H^s}\leq C \|Q\|_{H^s}, \quad \forall \epsilon >0.
		\end{equation}
		Moreover, by the Fatou's Lemma and $\mathcal{Q}\rightarrow |Q|$ as $\epsilon\rightarrow 0$, it holds that
		\begin{equation}
			\||Q|\|_{H^s}\leq C \|Q\|_{H^s}.
		\end{equation}
		This completes the proof of Lemma~\ref{lem-A6}.	
	\end{proof}
	
\end{appendices}

%

\end{document}